\documentclass[bsl]{asl}
\usepackage{rotate}
\usepackage{ifsym}
\usepackage{pifont}

\title{The Logic of Resources and Capabilities}

\author[M.~B\'{i}lkov\'{a}]{Marta B\'{i}lkov\'{a}}
\revauthor{B\'{i}lkov\'{a}, Marta}
\address{Department of Logic, Faculty of Arts, Charles University}
\email{marta.bilkova@ff.cuni.cz}
\thanks{The research of the first author has been supported by the project SEGA: From Shared Evidence to Group Agency, of the Czech Science Foundation, and DFG no.~16-07954J.}

\author[G.~Greco]{Giuseppe Greco}
\revauthor{Greco, Giuseppe}
\address{Department of Languages, Literature and Communication, 
University of Utrecht}
\email{G.Greco@uu.nl}

\author[A.~Palmigiano]{Alessandra Palmigiano}
\revauthor{Palmigiano, Alessandra}
\address{Faculty of Technology, Policy and Management, Delft University of Technology,\\ Department of Pure and Applied Mathematics, University of Johannesburg}
\email{A.Palmigiano@tudelft.nl}

\author[A.~Tzimoulis]{Apostolos Tzimoulis}
\revauthor{Tzimoulis, Apostolos}
\address{Faculty of Technology, Policy and Management, Delft University of Technology}
\email{A.Tzimoulis-1@tudelft.nl}
\thanks{The research of the second, third and fourth author has been made possible by the NWO Vidi grant 016.138.314, the NWO Aspasia grant 015.008.054, and a Delft Technology Fellowship awarded in 2013.}
\thanks{The research of the second author was also supported by the Values4Water project, subsidised by the Responsible Innovation research programme, which is partly financed by the Netherlands Organisation for Scientific Research (NWO) under Grant Number 313-99-316.}

\author[N.~Wijnberg]{Nachoem Wijnberg}
\revauthor{Wijnberg, Nachoem}
\address{Faculty of Economics and Business, University of Amsterdam,\\ College of Business and Economics, University of Johannesburg}
\email{N.M.Wijnberg@uva.nl}

\usepackage{pnets}

\usepackage{cmll}
\usepackage{etex}
\usepackage{graphics}
\usepackage{color}
\usepackage{cancel}
\usepackage{stmaryrd}
\usepackage{verbatim,cmll}
\usepackage{setspace}
\usepackage{lscape}
\usepackage{latexsym,relsize}
\usepackage{amscd}
\usepackage{multicol}
\usepackage{bussproofs}
\usepackage[position=top]{subfig}
\usepackage{tikz}
\usetikzlibrary{arrows,decorations.pathmorphing,backgrounds,positioning,fit,petri}
\usepackage{authblk}
\usepackage{multicol}
\usetikzlibrary{arrows}
\usetikzlibrary{matrix}
\usetikzlibrary{patterns}
\usetikzlibrary{shapes}

\usepackage[left=3cm,top=3cm,right=3cm,bottom=2cm]{geometry}
\newcommand{\commment}[1]{}

\def\aol{\rule[0.5865ex]{1.38ex}{0.1ex}}

\def\pdra{\mbox{$\,>\mkern-8mu\raisebox{-0.065ex}{\aol}\,$}}

\def\pdla{\mbox{\rotatebox[origin=c]{180}{$\,>\mkern-8mu\raisebox{-0.065ex}{\aol}\,$}}}

\def\mANDORatom#1{\hbox{\hbox to 0pt{$#1\TriangleUp$\hss}$#1\TriangleDown$}}

\newcommand{\mcAND}{%
\mathrel{\ooalign{\raisebox{-0.39ex}{$\mbox{\TriangleUp}$}\cr\kern4.2pt{\raisebox{-0.13ex}{$\cdot$}}}}}
\newcommand{\mcand}{%
\mathrel{\ooalign{$\vartriangle$\cr\kern1.99pt{\raisebox{-0.17ex}{$\cdot$}}}}}

\newcommand{\nAND}{%
\mathrel{\ooalign{$\mbox{\TriangleUp}$\cr\kern0pt$\mbox{\rotatebox[origin=c]{180}{\TriangleUp}}$}}}

\newcommand{\nand}{%
\mathrel{\ooalign{$\vartriangle$\cr\kern0pt$\triangledown$}}}

\newcommand{\mcBAND}{%
\mathrel{\ooalign{\raisebox{-0.39ex}{$\mbox{\FilledTriangleUp}$}\cr\kern4.2pt{\raisebox{-0.13ex}{${\color{white}\cdot}$}}}}}

\newcommand{\mcband}{%
\mathrel{\ooalign{$\blacktriangle$\cr\kern1.99pt{\raisebox{-0.17ex}{${\color{white}\cdot}$}}}}}

\newcommand{\mcRA}{%
\mathrel{\ooalign{
                  \raisebox{-0.3ex}{$\rotatebox[origin=c]{-90}{$\mbox{{\TriangleUp}}$}$}
                                                                            \cr\kern2.7pt{\raisebox{0.2ex}{$\cdot\mkern1.3mu$}}}}}

\newcommand{\mcra}{%
\mathrel{\ooalign{$\,{\vartriangleright\,}$\cr\kern3pt{\raisebox{0ex}{$\cdot$}}}}}

\newcommand{\mcraline}{%
-{\mkern-6mu{\mathrel{\ooalign{$\,{\vartriangleright\,}$\cr\kern3pt{\raisebox{0ex}{$\cdot$}}}}}}}

\newcommand{\mdraline}{%
{\mathrel{\ooalign{$\,{\vartriangleright\,}$\cr\kern3pt{\raisebox{0ex}{$\cdot$}}}}}{\mkern-6mu}-}

\newcommand{\cra}{%
\mathrel{\ooalign{$\,-{\mkern-3mu\vartriangleright\,}$\cr\kern8pt{\raisebox{0ex}{$\cdot$}}}}}

\newcommand{\mcBRA}{%
\mathrel{\ooalign{
                  \raisebox{-0.3ex}{$\rotatebox[origin=c]{-90}{$\mbox{\FilledTriangleUp}$}$}
                                                                            \cr\kern2.7pt{\raisebox{0.2ex}{${\color{white}\cdot}$}}}}}

\newcommand{\mcbra}{%
\mathrel{\ooalign{$\,-{\mkern-3mu\blacktriangleright\,}$\cr\kern8pt{\raisebox{0ex}{$\cdot$}}}}}

\newcommand{\mcLA}{%
\mathrel{\ooalign{
                  \raisebox{-0.3ex}{$\rotatebox[origin=c]{90}{$\mbox{\TriangleUp}$}$}
                                                                                     \cr\kern5.5pt{\raisebox{0.2ex}{$\cdot$}}
                                                                                                                              }}}

\newcommand{\mcla}{%
\mathrel{\ooalign{$\,{\vartriangleleft\,}$\cr\kern5pt{\raisebox{0ex}{$\cdot$}}}}}

\newcommand{\mclaline}{%
-{\mkern-6mu{\mathrel{\ooalign{$\,{\vartriangleleft\,}$\cr\kern5pt{\raisebox{0ex}{$\cdot$}}}}}}}

\newcommand{\mcBLA}{%
\mathrel{\ooalign{
                  \raisebox{-0.3ex}{$\rotatebox[origin=c]{90}{$\mbox{\FilledTriangleUp}$}$}
                                                                                     \cr\kern5.5pt{\raisebox{0.2ex}{${\color{white}\cdot}$}}
                                                                                                                                            }}}

  \numberwithin{equation}{section}
\newcommand\val[1]{{\lbrack\!\lbrack} {#1}{\rbrack\!\rbrack}}

\newcommand{\pand}{\wedge}

\newcommand{\por}{\vee}

\newcommand{\pra}{\rightarrow}
\newcommand{\plra}{\leftrightarrow}
\newcommand{\pla}{\leftarrow}

\newcommand{\gI}{%
\mathrel{\ooalign{$\mbox{T}$\cr\kern0pt$\mbox{\rotatebox[origin=c]{180}{T}}$}}}

\newcommand{\gbot}{\rotatebox[origin=c]{180}{$\tau$}}
\def\aga{\texttt{a}}
\def\agb{\texttt{b}}

\def\aol{\rule[0.5865ex]{1.38ex}{0.1ex}}

\makeatletter
\newcommand{\WKnowProxy}[2]{%
  {\mathbin{\ooalign{$#1\circ#2 $\cr\hidewidth
   \raise.155ex\hbox{$#1{\scriptstyle{\ast}}#2$}\hidewidth\cr  }}}}
\makeatother

\makeatletter
\newcommand{\BKnowProxy}[2]{%
  {\mathbin{\ooalign{$#1\bullet#2 $\cr\hidewidth
   \raise.155ex\hbox{$#1{\scriptstyle{\color{white}{\ast}}}#2$}\hidewidth\cr  }}}}
\makeatother









\newcommand{\fns}{\footnotesize}
\newcommand{\mc}{\multicolumn}

\newcommand{\dc}{\mbox{$\Diamond$}}
\newcommand{\dr}{\mbox{$\Diamond\mkern-7.58mu\rule[0.05ex]{0.12ex}{1.42ex}\ \,$}}

\newcommand{\bc}{\mbox{$\,\vartriangleright\,$}}
\newcommand{\br}{\mbox{$\,\vartriangleright\mkern-17mu\rule[0.52ex]{1.25ex}{0.12ex}\,$}}
\newcommand{\dcs}{\dc^{\!\sigma}}
\newcommand{\bcs}{\bc^{\!\!\!\sigma}}
\newcommand{\bcp}{\bc^{\!\!\!\pi}}
\newcommand{\drs}{\dr^{\!\sigma}}

\newcommand{\brp}{\br^{\!\!\pi}}
\newcommand{\drd}{\dr^{\!\delta}}
\newcommand{\brd}{\br^{\!\!\!\delta}}
\newcommand{\adr}{{{{\mbox{\scriptsize$\blacksquare$}}
\mkern-5.5mu{\textcolor{white}{\rule[-0.05ex]{0.12ex}{1.65ex}}}\,\,}}}

\newcommand{\abr}{\mbox{$\blacktriangleright\mkern-17.65mu\textcolor{white}{\rule[0.52ex]{1.6ex}{0.12ex}}$}}
\newcommand{\labr}{\mbox{$\blacktriangle\mkern-6.97mu\textcolor{white}{\rule[-0.06ex]{0.12ex}{1.5ex}}\ $}}

\newcommand{\DC}{\mbox{$\circ$}}
\newcommand{\BC}{\mbox{$\,\cdot\joinrel\vartriangleright\joinrel\cdot\,$}}
\newcommand{\DR}{\mbox{$\circ\mkern-5mu\rule[0.185ex]{0.12ex}{0.8ex}\ $}}

\newcommand{\BR}{\mbox{$\cdot\,{\vartriangleright\mkern-17mu\rule[0.52ex]{1.25ex}{0.12ex}}\,\,\cdot$}}

\newcommand{\ADR}{\mbox{${\bullet\mkern-5.03mu\textcolor{white}{\rule[0ex]{0.12ex}{1.1ex}}\ }$}}
\newcommand{\ABC}{\,\cdot\joinrel\blacktriangleright\joinrel\cdot\,}
\newcommand{\ADC}{{\bullet}}
\newcommand{\ABR}{\mbox{$\,\cdot\joinrel\!\!{\blacktriangleright\mkern-17.65mu\textcolor{white}{\rule[0.52ex]{1.6ex}{0.12ex}}}\!\joinrel\!\cdot\,$}}

\newcommand{\LABR}{\mbox{$\,\cdot\joinrel\!\!{{\blacktriangle\mkern-6.97mu\textcolor{white}{\rule[-0.06ex]{0.12ex}{1.5ex}}}}\joinrel\cdot\,$}}

\newcommand{\RCDOT}{\gtrdot}
\newcommand{\LCDOT}{\lessdot}
\newcommand{\RPLUS}{\sqsupset}
\newcommand{\LPLUS}{\sqsubset}

\newcommand{\rand}{\sqcap}
\newcommand{\ror}{\sqcup}
\newcommand{\RAND}{\,,\,}

\newcommand{\adc}{{\scriptstyle\blacksquare}}

\newcommand{\abc}{\blacktriangleright}

\def\fCenter{{\mbox{$\ \vdash\ $}}}

\EnableBpAbbreviations

\newcommand{\ac}{\texttt{\!c}}
\newcommand{\ad}{\texttt{\!d}}
\newcommand{\aj}{\texttt{\!j}}

\newcommand{\bba}{\mathbb{A}}
\newcommand{\bbA}{\mathbb{A}}
\newcommand{\bbQ}{\mathbb{Q}}
\newcommand{\bbM}{\mathbb{M}}

\newcommand{\I}{\textrm{I}}

\newcommand{\marginnote}[1]{\marginpar{\raggedright\tiny{#1}}}

\theoremstyle{plain}
\newtheorem{thm}{Theorem}[section]

\newtheorem{lem}[thm]{Lemma}
\newtheorem{cor}[thm]{Corollary}
\newtheorem{prop}[thm]{Proposition}

\newtheorem{lemma}[thm]{Lemma}

\theoremstyle{definition}

\newtheorem{definition}[thm]{Definition}

\renewcommand{\labelenumi}{\arabic{enumi}.}

\begin{document}

\begin{abstract}
We introduce the logic LRC, designed to describe and reason about agents' abilities and capabilities in using resources. The proposed framework bridges two -- up to now -- mutually independent  strands of literature: the one on logics of abilities and capabilities, developed within the theory of agency, and the one on logics of resources, motivated by program semantics. The logic LRC is suitable to describe and reason about key aspects of social behaviour in organizations. We prove a number of properties enjoyed by LRC (soundness, completeness, canonicity, disjunction property) and its associated analytic calculus (conservativity, cut elimination and subformula property).
These results  lay at the intersection of the algebraic theory of unified correspondence and the theory of multi-type calculi in structural proof theory. Case studies are discussed which showcase several ways in which this framework can be extended and enriched while retaining its basic properties, so as to model an array of issues, both practically and theoretically  relevant, spanning from planning problems to the logical foundations of the theory of organizations.
\end{abstract}

\maketitle

%\tableofcontents

\section{Introduction}
Organizations are social units of agents structured and managed to meet a need, or  pursue collective goals. %The study of organizations
In economics and social science,  organizations are studied % explains the effect various forms of organizational structures have on the generation of competitive advantage
in terms of %notions such as
agency, goals, capabilities, and inter-agent coordination \cite{scott1961organization,tsoukas2005oxford, gibbons2013handbook}. In strategic management, the dominant approach in the study of organizational performances is the so-called {\em resource-based view} %theory of organizations is the  %Moreover, the mainstream  literature in the theory of organizations in management science
\cite{wernerfelt1984resource, barney1991firm, mahoney1995management}, which has recognized that a central role in determining  the   success of an organization in market competition is played by the  acquisition, management, and transformation of {\em resources} within that organization. In order to capture  this insight and create the building blocks of the logical foundations of the theory of organizations, a formal framework is needed in which it is possible to express and reason about  agents' abilities and capabilities to use resources for achieving  goals, to transform resources into other resources, and to coordinate the use of resources with other agents; i.e., a formal framework is needed for capturing and reasoning about the  {\em resource flow within organizations}. The present paper aims at introducing such a framework. %laying the groundwork for the logical study of  {\em resource flow within organizations}, that is, the acquisition, management,  and transformation of resources,  resulting  from the activity of a single agent or of multiple agents.

There is extensive  literature in philosophical logic and formal AI accounting for  agents' abilities (cf.~e.g.~\cite{cross1986can,brown1988ability}) and capabilities (cf.~e.g.~\cite{van1994logic,dignum2004model,Dignum}) and their interaction, embedding in  the wider context of the logics of   agency (cf.~e.g.~\cite{chellas1969logical, segerberg1982deliberate,belnap1988stit, belnap1991backwards, chellas1995bringing, elgesem1997modal}); some of these frameworks (viz.~\cite{dignum2004model,Dignum}) have been used to formalize some aspects of the theory of organizations.  There is also  literature in theoretical computer science on the logic of resources (cf.~e.g.~\cite{pym2004possible, pym2006resources}), motivated by the build-up of mathematical models of computational systems. However, these two strands of research have been pursued independently, and in particular, the {\em interaction} between abilities, capabilities and resources has not been explored before.    %are all necessary items in  the toolbox for studying  organizations,\footnote{Specifically,  actions which are relevant to this context are those  brought about by the activity of agents (rather than actions capturing a more general notion of change, which can be formulated independently of agency).} the formal toolbox for studying organizations
%need to be augmented so as to   using resources to achieve a goal, transforming resources, and coordinating the use of resources are the building blocks of a logical theory of organizations.

   %presently there are not many examples in the literature of logical systems specifically designed to describe the internal dynamics of organizations, \cite{Dignum} being one seminal work in this direction. The logical formalism of \cite{Dignum} is similar in spirit to STIT logics \cite{}, in that it aims at describing  actions which are specifically brought about by the activity of agents (rather than actions capturing a more general notion of change, which can be formulated independently of agency). Hence, rather than being treated as  primitive, as is done in PDL,  actions are regarded  as emerging from a hierarchy of more primitive notions, the most basic of which are agency, and agents' {\em capabilities}.

The present paper introduces a logical framework,  the {\em logic of resources and capabilities} (LRC),  %extends and refines the logical system introduced in \cite{Dignum}, and
%aims at contributing to the logical foundations of organization theory  by
designed as an environment for the logical modelling of  the behaviour of agents motivated and mediated by the use and transformation of resources. In this framework, agents' capabilities are not captured via primitive actions, as is done e.g.~in \cite{van1994logic}, but rather via dedicated modalities,  similarly to the frameworks adopting the STIT logic approach \cite{belnap1988stit,dignum2004model,Dignum}.
However, LRC  differs from these logics in two main respects; the first is the focus on {\em resources}, discussed above; the second is that, as a modal extension of intuitionistic logic, LRC inherits its {\em constructive} character: it comes equipped with a constructive proof theory which provides an explicit computational content brought out by the cut elimination theorem. This guarantees that each LRC-theorem (prediction) translates into an effective procedure, thus allowing for a greater amenability to concrete applications in planning, and paving the way for implementations in constructive programming environments. In particular, LRC enjoys the disjunction property, proof of which we have included in Section \ref{ssec:disjunction property}.
In the present  paper, the basic mathematical theory of the logic of resources and capabilities is developed in an   algebraic and proof-theoretic environment.
Specifically,  the most important technical tool we introduce for LRC is the {\em proof calculus} D.LRC (cf.\ Section \ref{sec: Display-style sequent calculus D.LRC}). This calculus is
designed according to the  {\em multi-type} methodology, introduced in \cite{Multitype,PDL, Trends}, and further developed in \cite{linearlogPdisplayed, latticelogPdisplayed,Inquisitive,SDM}. This methodology exploits facts and insights coming from various semantic theories: from the coalgebraic semantics of dynamic epistemic logics (cf.\ \cite{GKPLori}), to the algebraic dual of the team semantics for inquisitive logic (cf.\ \cite{Inquisitive}), the representation theorems for lattices (cf.\ \cite{latticelogPdisplayed}), and the recently developed algebraic theory of unified correspondence \cite{CoGhPa14,CoPa12,CFPS15,PaSoZh15,CPZ:Trans,CPSZ,ConPalSou}, in the context of which, systematic connections have been developed (cf.\ \cite{GMPTZ,MZ16}) between Sahlqvist-type correspondence results and the theory of analytic rules for proper display calculi (cf.\ \cite{Wa98}) and Gentzen calculi.

Multi-type languages make it possible to express constituents such as actions, agents, or resources not as {\em parameters} in the generation of formulas, but as {\em terms} in their own right. They thus are regarded as  first-class citizens of the multi-type framework, and are  endowed with their corresponding structural connectives and rules. In this rich environment, %many features which were insurmountable hurdles to the standard treatment can be understood as symptoms of  the original languages of these logics lacking the necessary expressivity to
it is possible to encode certain key interactions {\em within the language}, by means of structural analytic rules. This approach has made it possible to develop analytic calculi for  logics notoriously impervious to the standard proof-theoretic treatment, such as Public Announcement Logic \cite{Plaza}, Dynamic Epistemic Logic \cite{BMS}, their nonclassical counterparts \cite{AMM, KP13}, and PDL \cite{kozen}.   %The success of the multi-type methodology lies in its providing a mathematical environment in which the expressivity problems can be clearly identified.
%For this reason,

One of the most important  benefits of multi-type calculi is the degree of {\em modularity} for which they allow. When applied to the present setting,   the metatheory of multi-type calculi makes it possible to add (resp.~remove) analytic structural rules to (resp.~from) the basic calculus D.LRC, and  obtain  variants endowed with a package of basic properties (soundness, completeness, cut-elimination, subformula property, conservativity) as  immediate consequences of general results. This feature is illustrated and exploited in Section \ref{sec:casestudies}, where we specialize D.LRC to various situations by adding certain analytic structural rules to it. More in general, an infinite class of axiomatic extensions and combinatoric variants of LRC  can be captured in a systematic way within this framework. Hence, LRC can be regarded not just as one single logic, but as a  {\em class} of  interconnected logical systems. Besides being of theoretical interest, this feature is of great usefulness in practice, since this class of logics forms  a coherent  framework which can be adapted to very different concrete settings with minimum effort. The combined strengths of this class of logics make the resulting LRC framework into a viable proposal for capturing and reasoning about the   resource flow within organizations.  %, in the sense that the rules added are all analytic, no further work is required. In particular, the axioms corresponding to the rules Ex$_c$ and Ex$_d$ are analytic inductive, and hence canonical (a direct proof of their canonicity is also given below); hence, the axiomatic extension of LRC corresponding to these axioms is sound and complete w.r.t.~the corresponding class of heterogeneous LRC-models. Conservativity can be argued by repeating verbatim the same argument given in Section ??? which uses the soundness of the augmented calculus w.r.t.~the corresponding class of perfect heterogeneous LRC-models, and the completeness of the Hilbert-style presentation of the axiomatic extension which holds because the additional axioms are canonical. Finally, cut-elimination and subformula property follow from the general cut-elimination metatheorem.

Finally,   LRC is the first example of a logical system  designed from first principles according to the multi-type methodology. As this example  shows,  multi-type calculi can serve not only to provide existing logics with well-performing calculi, but also   as a methodological platform for  the analysis and the meta-design of new logical frameworks. %The logic of resources and capabilities can be of interest also in this respect, since  %\marginnote{We need to link up with unified correspondence and say that this methodology draws from the state of the art of semantic theories, so we have an integrated syntactic/semantic approach, since these calculi come endowed with a semantics, so it's a package}

\paragraph{Structure of the paper.} In Section \ref{ssec:HLRC}, the logic LRC is introduced by means of a Hilbert-style presentation, which  is shown to be complete w.r.t.\ certain algebraic models (cf.~Section \ref{ssec:algebraic completeness}),  canonical (cf.~Section \ref{ssec:algebraic canonicity}) and to enjoy the disjunction property (cf.~Section \ref{ssec:disjunction property}). %In Section \ref{ssec:quasi-def}, we report on the cut-elimination metatheorem which we will appeal to when proving the
Then, in Section \ref{sec: Display-style sequent calculus D.LRC}, the multi-type calculus D.LRC is introduced, and is shown to be sound w.r.t.\ the algebraic models (cf.\ Section \ref{ssec:soundness}),  complete (cf.\ Section \ref{ssec:completeness}), and conservative (cf.\ Section \ref{ssec:conservativity}) w.r.t.\ the Hilbert-style presentation introduced in Section \ref{ssec:HLRC}.
%
%The present paper contains also contributions to the general theory of multi-type calculi. Specifically, in Section \ref{ssec:quasi-def}, we introduce the notion of  {\em proper semi-display multi-type  calculus}, which  generalizes Wansing's notion of proper display calculus (cf.\ \cite{Wa98}), and simultaneously encompasses the generalizations of this notion introduced in \cite{GAV,Multitype}. In particular, the semi-display condition requires that, for each type, it is either the case that principal occurrences in antecedent position of terms of that type be in display, or principal occurrences in succedent position be in display. %This requirement expands the scope of applicability of this metatheorem beyond display calculi.
%In Section \ref{ssec:metatheorem}, we prove the Belnap-style cut elimination metatheorem proven in , and
In Section \ref{ssec:cut elim subformula}, we prove that the calculus D.LRC %is sound,  complete, conservative; moreover, it
satisfies the assumptions of the cut-elimination  metatheorem proven in \cite{Trends}, and hence enjoys cut-elimination and subformula property.
In Section \ref{sec:casestudies}, we start exploring  various ways in which D.LRC can be modified and adapted to different contexts so that the resulting systems  retain all the properties enjoyed by the basic system. Specifically, Section \ref{ssec:homework} illustrates how {\em coordination} among agents helps optimizing capabilities towards a goal;  Section \ref{ssec:crow} explores the solution of a planning problem which requires the suitable concatenation of {\em reusable} and {\em non-reusable} resources; Section \ref{ssec:gifts} focuses on a situation in which the possibility of resources to be used in different {\em roles} becomes relevant; Section \ref{ssec:resilience} illustrates how the {\em resilience} of a fragment of a system can propagate to the system as a whole.
 %Future directions include logical modelling of organization theory in social science and management science. In the present paper, we start exploring
	
%\marginnote{Discuss informally case studies and the features of each of them. Moreover, discuss the technical advancements from the perspective of Display Calculi.}

\section{The logic of resources and capabilities and its algebraic semantics}
\subsection{Hilbert-style presentation of LRC}
\label{ssec:HLRC}

As mentioned in the introduction, the key idea is to introduce a language in which resources are not  accounted for as {\em parameters} indexing  the capability connectives, but as {\em logical terms} in their own right.
Accordingly,  we start by defining a {\em multi-type language} in which the different types interact via special  connectives. The present setting consists of the types $\mathsf{Res}$ for resources and $\mathsf{Fm}$ for formulas (describing states of affairs).
We stipulate that  $\mathsf{Res}$ and $\mathsf{Fm}$ are disjoint.

Similarly to the binary connectives introduced in \cite{Multitype}, the  connectives $\bc$, $\dr$ and $\br$ (referred to as {\em heterogeneous connectives}) facilitate the interaction between  resources and formulas:\footnote{As discussed below, these modal operators intend to capture agents' abilities and capabilities vis-\`a-vis resources; in this section, for the sake of a simpler exposition, we present the single-agent version of LRC, where any explicit mention of the agent is omitted.}
\begin{center}
\begin{tabular}{ll}
$\bc  :  \mathsf{Res} \times \mathsf{Fm} \to \mathsf{Fm}$ & $\br: \mathsf{Res} \times \mathsf{Res} \to \mathsf{Fm}$ \\
$\dc: \mathsf{Fm} \to \mathsf{Fm}$ & $\dr:  \mathsf{Res} \to \mathsf{Fm}$\\
\end{tabular}\end{center}

\noindent As discussed in the next section, the mathematical environment of heterogeneous LRC-algebras provides a natural interpretation for all these connectives.
Let us introduce the language of the logic of resources and capabilities. Let \textsf{AtProp} and \textsf{AtRes} %and \textsf{Ag}($\mathcal{L}$)
be  countable and disjoint sets of atomic propositions and atomic resources, respectively. The  set $\mathcal{R} = \mathcal{R}(\textsf{AtRes})$ of the {\em resource-terms} $\alpha$ {\em over} $\textsf{AtRes}$, and the set $\mathcal{L} = \mathcal{L}(\mathcal{R}, \mathsf{AtProp})$ of the {\em formula-terms} $A$ {\em over} $\mathcal{R}$ and  $\textsf{AtProp}$ of the Logic of Resources and Capabilities (LRC) are defined  as follows:
\begin{center}
	$\alpha ::= a\in \mathsf{AtRes} \mid\ 1 \mid 0 \mid  \alpha \cdot \alpha \mid \alpha\ror\alpha\mid \alpha\rand \alpha,$ %\mid \alpha\cap\alpha \mid \wn A \mid \alpha^\ror\;\; (A\in \mathcal{L})$.
\end{center}
\begin{center}
	$A ::= p\in \mathsf{AtProp} \mid\top \mid \bot \mid A \vee A \mid A \wedge A\mid A \to A \mid \alpha \bc A \mid \dc A \mid \dr \alpha \mid \alpha \br \alpha. \;\; $
\end{center}
%\marginnote{I created a macro `rand' for the conjunction of resources $\rand$, which is now in macrosPDL. Giuseppe, could you please modify it?}
When writing formulas, we will omit brackets whenever the functional type of the connectives allows for a unique reading. Hence, for instance, we will write $\alpha\bc(\dc A)$ as $\alpha\bc\dc A$ and $(\alpha \cdot \beta)\bc A$ as $\alpha \cdot \beta\bc A$. We will also  abide by the convention that $\vee$, $\wedge$, $\dc$, $\dr$, $\bc$ and $\br$ bind more strongly than $\rightarrow$,  that $\dc$, $\dr$, $\bc$ and $\br$ bind more strongly than $\vee$ and $\wedge$, and that $\leftrightarrow$ is a weaker binder than any other connective. With this convention, for instance, $\alpha\bc A\wedge B$ has the same reading as $(\alpha\bc A)\wedge B$.

\smallskip
The (single-agent version of the) logic of resources and capabilities LRC, in its Hilbert-style presentation H.LRC, is defined as the smallest set of formulas containing the axioms and rules of intuitionistic propositional logic\footnote{The classical propositional logic counterpart of LRC can be obtained as usual  by adding e.g.\  excluded middle  to the present axiomatization. Notice that classical propositional base is not needed in any  of the case studies of Section \ref{sec:casestudies}.}  plus the following axiom schemas:

\begin{center}
\begin{tabular}{rl}
& Pure-resource entailment schemas \\
%\marginnote{G: we need to specify the axiomatization of the propositional base: classical / intuitionistic or, simply, a distributive-based logic? See the structural rules for distributive lattices in Section 4. A: I would say: intuitionistic, but we remark that different propositional bases are also possible}
%\marginnote{\textcolor{red}{Note that we don't need axiom 9: given axiom 5, replacement of equivalents (or Cut) and Balance (see the first rule below on the right, it is derivable).
%The axioms 7 and 8 are not sound with the resources reading. Axiom 1 is derivable using Balance and axiom 6.}}

%\medskip
%\begin{tabular}{rl}
R1. & $\ror$ and $\rand$ are commutative, associative, idempotent, and distribute over each other; \\
R2. & $\cdot$ is associative with unit $1$; \\

R3. & $\alpha \vdash 1$ and $0 \vdash \alpha$ \\

R4. & $\alpha\cdot(\beta\ror\gamma)\vdash(\alpha\cdot\beta)\ror(\alpha\cdot\gamma)$ and $(\beta\ror\gamma)\cdot\alpha\vdash(\beta\cdot\alpha)\ror(\gamma\cdot\alpha)$. \\
\end{tabular}
\end{center}

\begin{center}
\begin{tabular}{rlrl}
         & Axiom schemas for $\dc$ and $\dr$                                                             &        &                                                                                                                                   \\
D1.   & $\dc (A \lor B) \plra \dc A \lor \dc B$                                                 & D3. & $\dr (\alpha \ror \beta) \plra \dr \alpha \lor \dr \beta$                                                   \\
D2.   & $\dc \bot \plra \bot$                                                                          & D4. & $\dr 0 \plra \bot$                                                                                                         \\
%D5.   & $\textcolor{red}{\dc \top \plra \top}$                                              & D6. &   $\textcolor{red}{\dr 1 \plra \top}$                                                                              \\
%D5. &   $\dr 1 \plra \top$                                                                               &        &                                                                                                                                   \\

 & & & \\

         & Axiom schemas for $\bc$ and $\br$                                                             &       &                                                                                                                                     \\
B1.   &  $(\alpha\ror\beta) \bc A \plra \alpha \bc A \land \beta \bc A$         & B4. & $(\alpha\ror\beta) \br \gamma \plra \alpha \br \gamma \land \beta \br \gamma$          \\
B2.   & $0 \bc A$                                                                                          & B5. & $0 \br \alpha$                                                                                                              \\
B3.   & $\alpha \bc \beta \bc A \pra \alpha\cdot\beta \bc A$                         &  B6. & $\alpha \br (\beta \rand \gamma) \plra \alpha \br \beta \land \alpha \br \gamma$ \\
         &                                                                                                          & B7. & $\alpha \br 1$                                                                                                              \\
%         &                                                                                                          & B8. & $\alpha \br \alpha$                                                                                                      \\

          & & & \\
         & Interaction axiom schemas                                                                           &       &                                                                                                                                     \\
%BD1. & $\alpha \bc (A \lor B) \pra \alpha \bc A \lor (\dr \alpha \land \dc B)$ &       %&                                                                                                                                     %\\
BD1. & $\dr \alpha \land \alpha \bc A \pra \dc A$                                         &       &                                                                                                                                      \\
BD2. & $\alpha \br \beta \pra \alpha \bc \dr \beta$                                       &       &                                                                                                                                      \\
\end{tabular}
\end{center}

\medskip

\noindent and closed under modus ponens, uniform substitution and the following rules:
\begin{center}
	\begin{tabular}{cccc}
		\AX$\alpha \fCenter \beta$
		\RightLabel{MF}
		\UI$\alpha\cdot \gamma  \fCenter \beta \cdot \gamma$
		\DisplayProof
&

			\AX$A \fCenter B$
		\RightLabel{MB}
		\UI$\alpha \bc A \fCenter \alpha \bc B$
		\DisplayProof
	   &
	  \AX$A \fCenter B$
	  \RightLabel{MD}
	  \UI$\dc A \fCenter \dc B$
	  \DisplayProof
	   & 	\AX$\alpha \fCenter \beta$
		\RightLabel{MB'}
		\UI$\gamma \br \alpha \fCenter \gamma \br \beta$
		\DisplayProof\\
	 &  & \\
\AX$\alpha \fCenter \beta$
		\RightLabel{MF'}
		\UI$\gamma\cdot \alpha  \fCenter \gamma\cdot \beta $
		\DisplayProof
&
	\AX$\alpha \fCenter \beta$
	\RightLabel{AB}
	\UI$\beta \bc A \fCenter \alpha \bc A$
	  \DisplayProof
	 &
	\AX$\alpha \fCenter \beta$
	\RightLabel{MD'}
	  \UI$\dr \alpha \fCenter \dr\beta$
	  \DisplayProof
	  &
	\AX$\alpha \fCenter \beta$
		\RightLabel{AB'}
		\UI$\beta \br \gamma \fCenter \alpha \br \gamma$
		\DisplayProof\\
         \\
	  \end{tabular}
\end{center}
Finally, for all $A, B\in \mathcal{L}$, we let   $A\vdash_{\mathrm{LRC}}B$ iff a proof of $B$ exists in H.LRC which possibly uses $A$.

\medskip

Let us expand on the intuitive meaning of the connectives, axioms and rules introduced above, and their formal properties.
\paragraph{The pure-resource fragment of LRC.} The pure-resource fragment of the logic LRC is  %similar to the logic of resources introduced e.g.~in  \cite{pym2006resources, pym2004possible}, which in turn are
inspired by  (distributive) linear logic.\footnote{However, the conceptual distinction is worth being stressed that, while formulas in linear logic {\em behave like} resources, pure-resource terms of LRC {\em literally denote} resources. In this respect, the pure-resource fragment of LRC is similar to the logic of resources introduced in  \cite{pym2006resources, pym2004possible}. %Hence, e.g.\ truth and falsity are  applicable to the former but not to the latter.
}
Indeed, as is witnessed by conditions R1-R4 and rules MF and MF', the algebraic behaviour of $\rand$ (with unit 1), $\ror$ (with unit 0) and $\cdot$ (with unit $1$) is that of the  additive conjunction,  additive disjunction and multiplicative conjunction in (distributive) linear logic, respectively. The intuitive understanding of the difference between $\alpha\cdot \beta$ and $\alpha\rand\beta$ is also borrowed from linear logic (cf.\ \cite[Section 1.1.2]{girard1995linear}): indeed, $\alpha\cdot \beta$ can be intuitively understood as the resource obtained by {\em putting} $\alpha$ and $\beta$ {\em together}. This `putting resources together' can be interpreted in many ways in different contexts: one of them is e.g.\ when  $\alpha$ (water) and $\beta$ (flour) are mixed together to obtain $\alpha\cdot \beta$ (dough); another is e.g.\ when  $\alpha$ (water) and $\beta$ (flour), juxtaposed in separate jars, are  %of $\alpha$ and $\beta$, like e.g.\ when water and flour, in separate jars, are
used at the same time so to form the counterweight $\alpha\cdot\beta$ to keep something in balance. Notice that under both interpretations, $\alpha\cdot \alpha$ is distinct from $\alpha$. We understand $\alpha\rand \beta$ as the resource which is as powerful as  $\alpha$ and  $\beta$ taken {\em separately}. % the power of which (to bring about states of affairs) is exactly that of $\alpha$ and of $\beta$ taken {\em separately}.
In other words, if we identify any resource $\gamma$ with the (upward-closed) set of the states of affairs which can be brought about using $\gamma$ (for brevity let us call such set the {\em power} of $\gamma$), then the resource $\alpha\rand \beta$ is uniquely identified by % the set of the states of affairs which can be brought about using $\alpha\rand\beta$, which is exactly
the {\em union} of the power of $\alpha$  and the power of $\beta$.  %\marginnote{do you have a less ridiculous example?}
Finally, we understand $\alpha\ror \beta$ as the resource the power of which  is the intersection of the power of $\alpha$ and the power of $\beta$.  %{\em have in common}. In other words, $\alpha\ror \beta$ is uniquely identified by the set of what can be brought about using $\alpha\ror\beta$, which is exactly the {\em intersection} of  what can be obtained using $\alpha$ alone and what can be obtained  using $\beta$ alone.
%$\alpha\ror\beta$ indicates a resource which can either coincide with $\alpha$ or with $\beta$, and hence can be attributed only the power (e.g.\ to bring about states of affairs) that $\alpha$ and $\beta$ {\em have in common}.
More in general,  the intended meaning of the resource-type entailment $\alpha\vdash \beta$   (namely `$\alpha$ is at least as powerful a resource as $\beta$'), together with the identification of the lattice of resources with the lattice of their powers (which is a lattice of  sets closed under union and intersection and hence distributive), %each resource $\alpha$ with its power, %the (upward-closed) set of the states of affair which can be brought about using $\alpha$,
explain intuitively the validity of resource-type entailments such as $\alpha\rand\alpha\dashv\vdash \alpha$, $\alpha\ror\alpha\dashv\vdash \alpha$, $\alpha\vdash \alpha\ror\beta$ and $\beta\vdash \alpha\ror\beta$, as well as $\alpha\rand(\beta\ror\gamma)\vdash (\alpha\rand\beta)\ror (\alpha\rand\gamma)$ and $(\alpha\ror\beta)\rand (\alpha\ror\gamma)\vdash \alpha\ror(\beta\rand\gamma)$. Moreover, under this reading of $\vdash$, by R3, the bottom $0$ and top $1$ of the lattice of resources can respectively be understood as the resource that is at least as powerful as any other resource (hence $0$ is impossibly powerful), and the resource  any other resource, no matter how weak, is at least as powerful as (hence $1$ is the resource with no power, or the {\em empty resource}). This intuition, together with the uniqueness of the neutral element, also justifies one of the main differences between this setting and general linear logic; namely, the fact that the unit of $\cdot$ is the  unit of $\rand$. Indeed, it seems intuitively plausible that, under the most common interpretations of $\cdot$, putting together  (e.g.~mixing or juxtaposing)  the empty resource  and any  resource $\alpha$  yields  $\alpha$ as  outcome.\footnote{\label{footnote:composition} This is one of the main differences between actions and resources: the idle action \texttt{skip}, represented as the identity relation, is the unit of the product operation on actions, and is clearly different from the top element in the lattice of actions (the total relation).}
Our inability to distinguish between the units of $\rand$ and  of $\cdot$ yields as a consequence that the following entailments hold, which are also valid in linear affine  logic \cite{onokomori, kopylov1995decidability} \begin{equation}
\label{eq:weakening}\alpha\cdot \beta\vdash \alpha\quad\mbox{ and }\quad\alpha\cdot \beta\vdash \beta.\end{equation} Indeed, by R3, R2 and MF', $\alpha\cdot \beta\vdash \alpha\cdot 1\vdash \alpha$, and the second entailment goes likewise. This  restricts  the scope of applications  of the present setting: for instance, the fact that the compound resource $\alpha\cdot\beta$ must be at least as powerful as its two components rules out the general examples of e.g.\ those chemical reactions in which the compound and its components are resources of incomparable power. On the other hand, it includes the case of all resources which can be quantified: two 50 euros bills are at least as powerful a resource than each  50 euros bill; two hours of time are at least as powerful a resource than one hour time,  and so on. %However, %this is not the case for all resources. Indeed, if $\alpha\cdot \beta$ denotes the product of two chemical components (e.g.\ water and alcohol in solution), then $\alpha\cdot\beta$ is in general neither more nor less powerful than its components. However, for the sake of the present paper we choose to restrict our setting to resources for which conditions R3 hold.
Moreover, this restriction does not rule out the possibility that the power of $\alpha\cdot\beta$ be strictly {\em greater} than the union of the separate powers of $\alpha$ and $\beta$ (which is the power of $\alpha\rand\beta$). This is  the case for instance when a critical mass of  fuel  is needed for reaching a certain temperature, or a certain outcome (e.g.\ a nuclear chain reaction). Another difference between the pure-resource fragment of LRC and linear logic is that, in LRC, the connective $\cdot$ is not necessarily commutative.  %The intuitive understanding of the  $\alpha$ and $$   %whereas $\cdot$  the multiplicative conjunction like Algebraically, they  say that  the algebra of resources is a quantale, in which the join is $\ror$ and the non necessarily commutative product is $\cdot$.

%\medskip

\paragraph{The modal operators.} The intended meaning of the formulas $\dc A$ and $\dr\alpha$ is `the agent is able to bring about state of affairs $A$' and `the agent is in possession of resource $\alpha$', respectively. %In general, the formula does not provide any information on whether $\alpha$ needs to be used up exhaustively or not. Also, the general interpretation of this formula does not imply that the resource has been actually used to bring about $A$. This allows to interpret the formulas $\alpha \dc \top$ as `the agent is in possession of the resource $\alpha$ and can use it.'
%In some situations (cf.\ case study of Section \ref{ssec:gifts}), the difference between being in possession of a resource and being in a position to {\em use} that resource becomes of the essence. In these situations, one might add a new connective $\Diamond: \mathsf{Res}\to \mathcal{L}$, so that the intended meaning of $\Diamond\alpha$ is `the agent is in possession of resource $\alpha$', with additional axiom $\alpha\dc\top\rightarrow \Diamond \alpha$.
By axioms D1 and D2 (resp.\ D3 and D4), the connective $\dc$ (resp.\ $\dr$) is a normal  diamond-type connective  (i.e.\ its algebraic interpretation is finitely join-preserving). %By axioms D4 and D6, the connective $\dc$ is a normal diamond-type connective in its first coordinate (i.e.\ its algebraic interpretation is finitely join-preserving in that coordinate), when the second coordinate is fixed to $\top$.
%The equivalence expressed by D1 %is based on the disjunction property: namely,
Axiom D1 expresses that being able to bring about $A\vee B$ is tantamount to either being able to bring about $A$ or being able to bring about $B$. Axiom D2 encodes the fact that  the agent can never bring about logical contradictions. Analogously, Axiom D3 %is also based on a form of disjunction property; indeed, it %
says that the agent is  in possession of  $\alpha\ror \beta$ exactly in case is in possession of $\alpha$ or is in possession of $\beta$. Axiom D4 encodes the fact that the agent is never in possession of the `impossibly powerful resource' $0$. %Axiom D5 is based on a form of disjunction property: if the agent is either in possession of $\alpha$ or of $\beta$ and can use the resource it has to bring about $A$, then it has to be either the case that it is in possession of $\alpha$, using which it can bring about $A$, or it is the case that the agent is in possession of $\beta$, using which it can again bring about $A$. However, the converse direction is arguably not valid. Indeed, assume that $\alpha\dc A$ is true but $\beta\dc A$ is false. Then $\alpha\dc A \vee \beta\dc A$ is true. However, to claim that $\alpha \ror\beta\dc A$ is true, we would need to guarantee that also in case $\alpha\ror\beta$ is instantiated as  $\beta$, the agent  can bring about $A$, and we cannot, given the assumptions. Axiom D6 also encodes a form of disjunction property for resources: the agent is in possession of $\alpha\ror\beta$ iff  it is either in  possession of $\alpha$ or in possession of $\beta$.

The intended meaning of the formula $\alpha \bc A$ is `whenever  resource $\alpha$ is in possession of the agent, using $\alpha$ the agent is capable to bring about $A$'. %Hence, this formula is always true if $\dr\alpha$ is false.\marginnote{I don't think we can derive this}
By axioms B1 and B2, the connective $\bc$ is an antitone normal box-type operator in the first coordinate (i.e.\ its algebraic interpretation is finitely join-reversing in that coordinate).   Axiom B1 says that the agent is capable of bringing about $A$ whenever in possession of $\alpha\ror\beta$ iff the agent is capable of bringing about $A$ {\em both} whenever  in possession of $\alpha$ {\em and}  whenever  in possession of $\beta$. Axiom B2 means that if the agent were in possession of the impossibly powerful resource (which is never the case by D4), the agent could bring about any state of affairs. The justification of axiom B3 is connected with the constraint, encoded in \eqref{eq:weakening}, that the fusion $\alpha\cdot \beta$ of two resources is at least as powerful as each of its components. Taking this fact into account, let us assume that the agent is in possession of $\alpha\cdot\beta$. Hence, by \eqref{eq:weakening}, the resource in its possession is at least as powerful as the resources $\alpha$ and $\beta$ taken in isolation. If $\alpha\bc \beta\bc A$ is the case, then by using $\alpha\cdot \beta$ up to $\alpha$, the agent can bring about $\beta\bc A$, and by using the remainder of $\alpha\cdot \beta$, the agent can bring about $A$, which motivates B3.
However, the converse direction is arguably not valid. Indeed, let $\alpha\cdot\beta\bc A$ express the fact that a certain temperature is reached by burning a critical mass $\alpha\cdot\beta$ of fuel. However, burning $\alpha$ and then $\beta$ in sequence might not be enough to reach the same temperature.\footnote{There is a surface similarity between  B3 and Axiom Ac4 of \cite[Section 4]{van1994logic}, which captures the interaction between  the capabilities of agents to perform actions and composition of actions; however, as remarked in Footnote \ref{footnote:composition}, composition of actions behaves differently from composition of resources, which is why B3 is an implication and not a bi-implication. }

The intended meaning of the formula $\alpha \br \beta$ is `the agent is capable of getting $\beta$ from $\alpha$, whenever in possession of $\alpha$'. By axioms B4 and B5, the connective $\br$ is an antitone normal box-type operator in the first coordinate (i.e.\ its algebraic interpretation is finitely join-reversing in that coordinate).  Axiom B4 says that the agent is capable of getting resource $\gamma$ whenever in possession of $\alpha\ror\beta$ iff the agent is capable of getting resource $\gamma$ {\em both} whenever  in possession of $\alpha$ {\em and}  whenever  in possession of $\beta$. Axiom B5 says that if the agent were in possession of the impossibly powerful resource (which is never the case by D4), the agent could get any  resource. By axioms B6 and B7, the connective $\br$ is a monotone normal box-type operator in the second coordinate (i.e.\ its algebraic interpretation is finitely meet-preserving in that coordinate). Axiom B6 says that the agent is capable of getting resource $\beta\rand\gamma$ whenever in possession of $\alpha$ iff the agent is capable of getting {\em both} $\beta$ and $\gamma$   whenever  in possession of $\alpha$. Axiom B7 says that any agent is capable to get  the empty resource whenever in possession of any resource. %\marginnote{(but notice that $\dr 1$ is not an axiom)} %and axiom B8 means that any agent is capable to get  any resource in its possession.\marginnote{are we sure we want to have this? It might be not true, for instance, with time it might be not the case that because you have it you keep it...}

%Axiom BD1 is  based on the disjunction property: indeed, assume that $\alpha \bc (A \lor B)$  is the case  but $\alpha \bc A$ is not; the latter condition implies that $\alpha$ is in possession of the agent (hence $\dr \alpha$ is the case), who however is not capable of bringing about $A$. But then, by $\alpha \bc (A \lor B)$, the agent must be capable of bringing about $B$. Hence $\dc B$ is the case. %The converse direction follows as a theorem (see below).
%Using BD1, BD3 and RB, it is not difficult to show that $\bc$ is a regular diamond-type operator in its second coordinate (i.e.\ its algebraic interpretation  preserves finite nonempty joins in that coordinate).
Axiom BD1 encodes  the link between the agent's capabilities and  abilities: indeed, it expresses the fact that if the agent is capable to bring about $A$ whenever in possession of $\alpha$ ($\alpha\bc A$), and moreover the agent is actually in possession of $\alpha$ ($\dr\alpha$), then the agent is able to bring about $A$ ($\dc A$). %The converse direction follows as a theorem (see below).
Notice also the analogy between this axiom and the intuitionistic axiom $A\wedge (A\rightarrow B)\leftrightarrow A\wedge B$.
Axiom BD2 establishes a link between $\bc$ and $\br$, via $\dr$; indeed, it says that the agent's being capable to get $\beta$ implies that the agent is capable to bring about a state of affairs in which the agent is in possession of $\beta$.  %says that if the agent is able to bring about  $A$ using the resource $\alpha$ then this is not by chance: indeed, it must be capable to bring about $A$ whenever  is given $\alpha$.

The rules MB and AB (resp.\ MB' and AB') encode the fact that $\bc$ (resp.\ $\br$) is monotone  in its second  coordinate and antitone in its first. In fact, AB, MB' and AB' can be derived using  B1,  B4 and B6. %which follows from  the fact, already discussed above, that $\bc$ and $\br$ are finitely meet-preserving (resp.\ join-reversing) in their second  coordinate and finitely join-reversing in their first.
%The rule RB expresses the fact that  $\bc$ is monotone in its second coordinate.
 The monotonicity of $\bc$ in its second coordinate expresses the intuition that if the agent is capable, whenever in possession of $\alpha$, to bring about $A$, then  is capable to bring about any state of affairs which is logically implied by $A$. The remaining rules encode the monotonicity of $\dc$, $\dr$ and $\cdot$.
%\marginnote{are we happy with this, or shall I try and justify why being in possession of a resource $\alpha$ implies being in possession of all the resources weaker than $\alpha$? it might not be very intuitive, but maybe if we rephrase the interpretation it becomes more natural...} %Moreover, if an agent is in possession of $\alpha$ and
%The rule RD' encodes the closure of $\dc$ under substitution of equivalents in its first coordinate. Indeed, $\dc$ is neither monotone nor antitone in its first coordinate: if $\alpha\leq \beta$ and $\alpha\dc A$ then we cannot conclude that  less powerful resources will be enough for the agent to bring about  $A$. Likewise, if $\beta\dc A$, then we cannot conclude that the agent is in possession of more powerful resources than $\beta$, although if this was the case, the agent could bring about $A$.

%The converse of BD2 follows from BD3 and RD; the converse of BD1 follows from BD3 and RB.

%To justify RDB', assume that $\alpha$ is more powerful than $\beta$ and the agent is in possession of $\beta$ and using $\beta$ it can bring about $A$. Then, whenever

%\medskip

\paragraph{Some additional axioms.} We conclude the present discussion by mentioning some analytic axioms which %although generally plausible, are not theorems of the present framework, and
might perhaps be interesting for different settings. We start mentioning  $\dc \top$, $\dr 1$, and  $\alpha\bc \top$, respectively stating that the agent is able to bring about what is always the case, such as logical tautologies; the agent is  in possession of the empty resource;   the agent is   capable of using any resource (hence also the empty one) to  bring about what is always the case. We also mention $\alpha\br\alpha$, stating that the agent is capable to get  any resource already in the possession of the agent; $\dr \alpha\land \alpha\br \beta \pra \dr \beta$, and $\dr \alpha\land \alpha\br \beta \pra \dc \dr \beta$. The latter is a consequence of BD1 and BD2, while the former is used in the case study  %We did not need them  in treating the case studies
in Section \ref{ssec:resilience}. For the sake of achieving greater generality we chose not to include it in the general system. %\marginnote{find a motivation why we did not include it}
Axioms which might also be considered in special settings are $\alpha\bc (A\lor B)\rightarrow \alpha\bc A\lor \alpha\bc B$, and $\alpha\bc A\land \alpha\bc B\rightarrow \alpha\bc (A\land B)$. The first one %is a variant of BD1, and
would imply the distributivity of $\bc$ over disjunction in its second coordinate. %however, this axiom seems to be less general than BD1 in the presence of the disjunction property. Namely, the truth of $\alpha\bc A\lor \alpha\bc B$ requires deciding which of $\alpha\bc A$ and $\alpha\bc B$ is the case. Hence,  the intuitive meaning of $\alpha\bc (A\lor B)$ only requires that the agent chooses which state of affairs to bring about after being in possession of $\alpha$ and not before, which  is a less restrictive requirement than the one for the satisfaction of $\alpha\bc A\lor \alpha\bc B$.
The axiom $\alpha\bc A\land \alpha\bc B\rightarrow \alpha\bc (A\land B)$ is not applicable in general, given that the consequence would require the duplication of the resource $\alpha$. More generally applicable variants are $\alpha\bc A\land \alpha\bc B\rightarrow \alpha\cdot \alpha\bc (A\land B)$ and $\alpha\br \beta\land \alpha\br \gamma\rightarrow \alpha\cdot \alpha\br (\beta\cdot \gamma)$. The latter encodes the behaviour of scalable resources, and will be used in the case study of Sections \ref{ssec:crow} and \ref{ssec:resilience}. Another interesting axiom is the converse of B3, which we have discussed above. %\marginnote{there was a story about scalability, in the case studies. Shall I move it here, or shall I only mention this discussion?}
%Some of these axioms and also additional rules will be discussed further in Section \ref{sec:casestudies}, when we will make additions to the basic system to treat specific situations.
% Finally, we wish to stress that being able to rely on the basic properties (soundness, completeness, cut elimination, subformula property, conservativity) of D.LRC being stable under the addition of analytic structural rules is a remarkable and useful feature of the present system, which has been hardwired in its design.

%\marginnote{The comment on why $\cdot$ is total and not partial is still to be added. Shall we mention some of the analytic rules these axioms correspond to?}
%%%
\subsection{Algebraic completeness} %\marginnote{A: edited. Please check that the statements I make are correct}
\label{ssec:algebraic completeness}
In the present section we outline the completeness of LRC w.r.t.\ the  heterogeneous LRC-algebras\footnote{This notion specializes  the more general notion of heterogeneous algebras introduced in \cite{birkhoff1970heterogeneous} to the setting of interest of the present paper.} defined below, via a Lindenbaum-Tarski type construction.

\begin{definition}
	\label{def:LRC-algebraic structure}
	A {\em heterogeneous} LRC-{\em algebra} is a tuple $F=(\bbA,\bbQ,\bc, \dc, \br, \dr)$ such that $\bbA$ is a Heyting algebra, $\bbQ = (Q, \ror, \rand, \cdot, 0, 1)$ is a  bounded distributive lattice with binary operator $\cdot$ which preserves finite joins in each coordinate and the unit of which is $1$,\footnote{It immediately follows from the definition that $\alpha\cdot\beta\leq \alpha$ and $\alpha\cdot\beta\leq \beta$ for all $\alpha, \beta \in Q$.
} %such that $\alpha\cdot\beta\leq \alpha$ and $\alpha\cdot\beta\leq \beta$ for all $\alpha, \beta \in Q$,
and $\bc: \bbQ\times \bbA\to \bbA$, $\dc:\bbA\to\bbA$, $\br:\bbQ\times \bbQ\to\bbA$, $\dr:\bbQ\to\bbA$ verify the (quasi-)inequalities corresponding to the axioms and rules of LRC as presented in the previous section. A heterogeneous LRC-algebra   is {\em perfect} if both $\bbA$ and $\bbQ$ are perfect,\footnote{\label{footnote: perfect HA} A bounded distributive lattice (BDL) is {\em perfect} if it is complete, completely distributive and completely join-generated by its completely join-irreducible elements. A BDL is perfect iff it is isomorphic to the lattice of the upward-closed subsets of some poset. A  Heyting algebra is {\em perfect} if its lattice reduct is a perfect BDL. A bounded distributive lattice with operators (abbreviated DLO. {\em Operators} are additional operations which are finitely join-preserving in each coordinate) is {\em perfect} if its lattice reduct is a perfect BDL, and  each operator is completely join-preserving in each coordinate.} and the operations $\bc$, $\dc$, $\br$, and $\dr$ satisfy the infinitary versions of the join- and meet-preservation properties satisfied by definition in any heterogeneous LRC-algebra.
An  {\em algebraic} LRC-{\em model} is a tuple $\bbM: = (F, v_{\mathsf{Fm}}, v_{\mathsf{Res}})$ such that $F$ is a heterogeneous LRC-algebra,  $v_{\mathsf{Fm}}: \mathsf{AtProp}\to \bbA$ and $v_{\mathsf{Res}}: \mathsf{AtRes}\to \bbQ$. Clearly, for every algebraic LRC-model $\bbM$, the assignments $v_{\mathsf{Fm}}$ and $v_{\mathsf{Res}}$ have unique homomorphic extensions which we identify with $v_{\mathsf{Fm}}$ and $v_{\mathsf{Res}}$ respectively.  For each $\mathsf{T}\in\{\mathsf{Fm}, \mathsf{Res}\}$  and all terms $a, b$ of type $\mathsf{T}$, we let $a\models_{\mathrm{LRC}} b$ iff $v_{\mathsf{T}}(a)\leq v_{\mathsf{T}}(b)$ for every model $\bbM$.
\end{definition}
Given $\mathsf{AtProp}$ and $\mathsf{AtRes}$, the {\em Lindenbaum-Tarski heterogeneous LRC-algebra} over $\mathsf{AtProp}$ and  $\mathsf{AtRes}$ is defined to be the following structure: $$F^\star:=(\mathbb{A}^\star,\bbQ^\star,\bc^{\!\!\!\star},\dc^{\!\star},\br^{\!\!\!\star},\dr^{\!\star})$$ where:
\begin{enumerate}
	\item $\mathbb{A}^\star$ is the quotient algebra $\mathsf{Fm}/{\dashv\vdash}$, where $\mathsf{Fm}$ is the formula algebra corresponding to the language $\mathcal{L}$ defined in the previous subsection, and     $\dashv\vdash$ is the equivalence relation on $\mathsf{Fm}$ defined as $A\dashv\vdash A'$ iff $A\vdash A'$ and $A'\vdash A$. Notice that the rules MD, MB,AB, MD', MB' and AB' guarantee that $\dashv\vdash$ is compatible with $\dc$, $\bc$, $\dr$ and $\br$, hence the quotient algebra construction is well defined. The elements of $\mathbb{A}^\star$ will be typically denoted $[B]$ for some formula $B\in \mathcal{L}$;
	\item $\bbQ^\star$ is the quotient algebra $\mathsf{Res}/{\dashv\vdash}$, where $\mathsf{Res}$ is the resource algebra corresponding to the language $\mathcal{R}$ defined in the previous subsection, and     $\dashv\vdash$ is the equivalence relation on $\mathsf{Res}$ defined as $\alpha\dashv\vdash \alpha'$ iff $\alpha\vdash \alpha'$ and $\alpha'\vdash \alpha$. Notice that the rules MF and MF' guarantee that  $\dashv\vdash$ is compatible with $\cdot$, hence the quotient algebra construction is well defined. The elements of $\bbQ^\star$ will be typically denoted $[\alpha]$ for some resource $\alpha\in \mathcal{R}$;
\begin{comment}
	\item $\bbQ^\star$ is defined as follows: for every  formula $A$, let $\mathsf{Res}(A)$ denote the set of resource terms occurring in $A$. Let $M$ be the (finite) poset of strings of elements in $\mathsf{AtRes}\cap\mathsf{Res}(A)$ %\marginnote{check that $\mathsf{Res}(A)$ is defined as intended. A: I think so, any other opinion?}
of length smaller than or equal to $n$, where $n$ can be taken as the number of symbols in $A$. The partial order on $M$ is defined as $s\leq t$ iff $t$ is a substring of $s$. Let $Q^\star$ be the poset of downsets of $M$, ordered by inclusion.  This poset is finite, has a natural structure of distributive lattice, and can be endowed with a monoidal structure by letting $D_1\cdot D_2:=\{s\cdot t \mid s\in D_1, t\in D_2, s\cdot t\in M\}$ for each $D_1$ and $D_2$ in $Q^\star$. Then by construction, $\cdot$ is associative and distributes over $\cup$ in each coordinate. Moreover,  it is very easy to verify that  $M\in Q^\star$ is both the top of $Q^\star$ as a distributive lattice, and the unit of the product $\cdot$.
\end{comment}
	\item $\bc^{\!\!\!\star}:Q^\star\times\mathbb{A}^\star\to\mathbb{A}^\star$ is defined as $[\alpha] \bc^{\!\!\!\star} [B]:=[\alpha\bc B]$;
	\item $\dc^{\!\star}:\mathbb{A}^\star\to\mathbb{A}^\star$ is defined as $\dc^{\!\star}[B]:=[\dc B]$;
	\item $\br^{\!\!\!\star}:Q^\star\times Q^\star\to\mathbb{A}^\star$ is defined as $[\alpha_1]\br^{\!\!\!\star}[\alpha_2]:= [\alpha_1\br\alpha_2]$;
	\item  $\dr^{\!\star}:Q^\star\to\mathbb{A}^\star$ is defined as $\dr^{\!\star}[\alpha]:=[\dr\alpha]$;
\end{enumerate}
%where $\ror D$ abbreviates  $d_1\ror\cdots\ror d_n$ for every $D= \{d_1,\ldots, d_n\}\in Q^\star$.
\begin{lemma}\label{lem:canonicalmodel}
	For any $\mathsf{AtProp}$ and $\mathsf{AtRes}$, $F^\star$ is a heterogeneous  LRC-algebra.
\end{lemma}
\begin{proof}
	It is a standard verification that $\bbA^\star$ is a Heyting algebra and that $\bbQ^\star$ is a bounded distributive lattice with binary operator $\cdot$ which preserves finite joins in each coordinate and the unit of which is $1$.  It is also an easy verification that    $\bc^{\!\!\!\star}$, $\dc^{\!\star}$, $\br^{\!\!\!\star}$ and $\dr^{\!\star}$ are well-defined, and verify the additional conditions by construction.
	%As discussed above,  $(Q^\star, \cap, \cup, \varnothing, M, \cdot)$ is such that $(Q^\star, \cap, \cup, \varnothing, M,)$ is a finite distributive lattice, and $(Q^\star, \cup, \varnothing, M, \cdot)$ is a quantale. It is also an easy verification that    $\bc^{\!\!\!\star}$, $\dc^{\!\star}$, $\br^{\!\!\!\star}$ and $\dr^{\!\star}$ are well-defined, and verify the additional conditions by construction.
\end{proof}

The canonical assignments can be defined as usual, i.e.\ mapping atomic propositions and resources to their canonical value in $F^\star$.
Let $\bbM^*$ be the resulting LRC-algebraic model. With this definition, the proof of the following proposition is routine, and is omitted.
\begin{prop}
\label{prop:algebraic completeness}
For all $X\subseteq \mathcal{L}$ and $A\in\mathcal{L}$, if $X\not\vdash_{\mathrm{LRC}} A$, then $X\not \models _{\mathrm{LRC}} A$.
\end{prop}

\subsection{Algebraic canonicity}
\label{ssec:algebraic canonicity}

%The theory of canonical extensions provides a way to extract a relational semantics from the algebraic semantics via algebraic canonicity.
The present subsection is aimed at showing that LRC is strongly complete w.r.t.\ perfect heterogeneous LRC-algebras. % we  show the algebraic canonicity of the axioms and rules of LRC (see definition below), %Together with the completeness via Lindenbaum construction of the previous subsection, this implies that the logic LRC is  complete w.r.t.\ {\em perfect} algebraic LRC-models, which can be associated, via standard discrete Stone duality, with set-based structures similar to Kripke models, thus providing complete relational semantics for LRC. In the present paper, we will not pursue the specification of this relational semantics, but rather,
This will be a key ingredient in the conservativity proof of Section \ref{ssec:conservativity}. % \marginnote{how much detail do we want to give?}
\begin{definition}
	Let $F=(\bbA,\bbQ,{\bc}, {\dc},\br,\dr)$ be a heterogeneous LRC-algebra. The {\em canonical extension} of $F$ is $$F^\delta=(\bbA^\delta,\bbQ^\delta,\bcp,\dcs,\brp,\drs),$$
	where $\bbA^\delta$ and $\bbQ^\delta$ are the canonical extensions of $\bbA$ and $\bbQ$ respectively\footnote{ \label{def:can:ext}
		The \emph{canonical extension} of a BDL  $L$ is a complete distributive lattice $L^\delta$ containing $L$ as a sublattice, such that:
		\begin{enumerate}
			\item \emph{(denseness)} every element of $L^\delta$ can be expressed both as a join of meets and as a meet of joins of elements from $L$;
			\item \emph{(compactness)} for all $S,T \subseteq L$, if $\bigwedge S \leq \bigvee T$ in $L^\delta$, then $\bigwedge F \leq \bigvee G$ for some finite sets $F \subseteq S$ and $G\subseteq T$.
		\end{enumerate}
		%\end{definition}
		
		It is well known that  the canonical extension of a BDL is a \emph{perfect} BDL (cf.~Footnote \ref{footnote: perfect HA}). %i.e.\ a complete and completely distributive lattice which is completely join-generated by its completely join-irreducible elements and completely meet-generated by its completely meet-irreducible elements.
Completeness and complete distributivity  imply that each  perfect BDL is naturally endowed with a Heyting algebra structure, and hence each perfect BDL is also a perfect Heyting algebra. Moreover, if $L$ is the lattice reduct of some Heyting algebra $\bbA$, then $\bbA$ is a subalgebra of $L^\delta$, seen as a perfect Heyting algebra. The {\em canonical extension} $\bbA^\delta$ of a Heyting algebra $\bbA$ is defined as the canonical extension  of the lattice reduct of $\bbA$ endowed with its natural Heyting algebra structure. The {\em canonical extension} $\bbQ^\delta$ of a DLO $\bbQ$ is defined as the canonical extension  of the lattice reduct of $\bbQ$ endowed with the $\sigma$-extension of each additional operator. It is well known that the canonical extension of a Heyting algebra (resp.~DLO) is a perfect Heyting algebra (resp.~DLO).}, the operations $\drs: \bbQ^\delta\to \bbA^\delta$ and $\brp: \bbQ^\delta\times \bbQ^\delta\to \bbA^\delta$
	and $\dcs:\bbA^\delta\to\bbA^\delta$ and $\bcp: \bbQ^\delta\times \bbA^\delta\to \bbA^\delta$ are defined as follows: for any $k\in K(\bbA^\delta)$, $\kappa\in K(\bbQ^\delta)$ and $o\in O(\bbA^\delta)$, $\omega\in O(\bbQ^\delta)$,\footnote{For any BDL $L$, an element $k \in L^\delta$ (resp.\ $o \in L^\delta$) is \emph{closed} (resp.\ \emph{open}) if  is the meet (resp.\ join) of some subset of $L$.  The set of closed (resp.\ open) elements of $L^\delta$ is $K(L^\delta)$ (resp.\ $O(L^\delta)$). We will slightly abuse notation and write $K(\bbA^\delta)$ (resp.\ $O(\bbA^\delta)$) and $K(\bbQ^\delta)$ (resp.\ $O(\bbQ^\delta)$) to refer to the sets of closed and open elements of their lattice reducts.} %By the remark in footnote \ref{footnote:Qisalattice}, $Q$ is the canonical extension of itself. As is well-known
\begin{multline*}
	\drs \kappa: =\bigwedge\{\dr \alpha\mid \alpha\in \bbQ\mbox{ and } \kappa\leq\alpha\} \quad \kappa \brp \omega: = \bigvee\{\alpha \br \beta\mid \beta\in \bbQ\mbox{, } \beta\leq\omega\mbox{, } \alpha\in\bbQ\mbox{ and }\kappa\leq\alpha\},
\end{multline*}
\begin{multline*}
	\dcs k: = \bigwedge\{\dc a\mid a\in \bbA\mbox{ and } k\leq a\} \quad\quad \kappa \bcp o: = \bigvee\{\alpha \bc a\mid a\in \bbA\mbox{, }a\leq o\mbox{, } \alpha\in\bbQ\mbox{ and }\kappa\leq\alpha\},
\end{multline*}
	and for any $u\in \bbA^\delta$ and $q,w\in \bbQ^\delta$,
\begin{center}
\begin{tabular}{l}
$\drs q: = \bigvee\{\drs \kappa\mid \kappa\in K(\bbQ^\delta)\mbox{ and } \kappa\leq q\}$ \\
$q \brp w: = \bigwedge\{\kappa \brp \omega\mid \omega\in O(\bbQ^\delta)\mbox{, } w\leq\omega\mbox{, }\kappa\in K(\bbQ^\delta)\mbox{ and }\kappa\leq q\} \textcolor{white}{\rule[-0.05ex]{0.12ex}{4ex}}$ \\
$\dcs u: = \bigvee\{\dcs k\mid k\in K(\bbA^\delta)\mbox{ and } k\leq u\} \textcolor{white}{\rule[-0.05ex]{0.12ex}{4ex}}$ \\
$q \bcp u: = \bigwedge\{\kappa \bcp o\mid o\in O(\bbA^\delta)\mbox{, } u\leq o\mbox{, }\kappa\in K(\bbQ^\delta)\mbox{ and }\kappa\leq q\} \textcolor{white}{\rule[-0.05ex]{0.12ex}{4ex}}$. \\
\end{tabular}
\end{center}
 Below we also report the definition of  $\cdot^\sigma$ for the reader's convenience: For $\kappa_1,\kappa_2\in K(\bbQ^\delta)$ $$\kappa_1\cdot^\sigma\kappa_2 =\bigwedge\{\alpha\cdot\beta\mid\alpha,\beta\in\bbQ\mbox{ and }\kappa_1\leq\alpha,\kappa_2\leq\beta\}, $$ and for any $q_1,q_2\in\bbQ^\delta$ $$q_1\cdot^\sigma q_2=\bigvee\{\kappa_1\cdot^\sigma\kappa_2\mid\kappa_1,\kappa_2\in K(\bbQ^\delta)\mbox{ and }\kappa_1\leq q_1,\kappa_2\leq q_2\}.$$
In what follows, for the sake of readability, we will write $\cdot^\sigma$ without the superscript. This will not create ambiguities, since we use different variables to denote the elements of $\bbQ$, $K(\bbQ^\delta)$, $O(\bbQ^\delta)$ and $\bbQ^\delta$, and since $\cdot$ and $\cdot^\sigma$ coincide over $\bbQ$.
\end{definition}
\begin{lem}\label{lem:canonicity}
	 For any   heterogeneous LRC-algebra $F$, the canonical extension  $F^\delta$  is a perfect heterogeneous LRC-algebra.
\end{lem}
\begin{proof} As discussed in  Footnote \ref{def:can:ext}, $\bba^\delta$ is a perfect Heyting algebra and $\bbQ^\delta$ is a perfect DLO, so to finish the proof it is enough to show that the validity  of all axioms and rules of LRC transfers from $F$ to $F^\delta$, and moreover, the join-and meet-preservation properties of the operations of $F$ hold in their infinitary versions in $F^\delta$. Conditions R2 hold in $\bbQ^\delta$ as consequences of the general theory of canonicity of terms purely built on operators (cf.~\cite[Theorem 4.6]{gehrke1994bounded}).
	  As to D1 and D2, by assumption the operation  $\dc$  preserves finite  joins. Hence, by a well known fact of the theory of the  $\sigma$-extensions of finitely join-preserving maps, $\dcs$   preserves arbitrary  joins  (cf.\ \cite[Theorem 3.2]{gehrke1994bounded}). The same argument applies to D3, D4, B4, B5, B6, B7. Furthermore, by \cite[Lemma 2.22]{GJ04} it follows that $\bcp$  turns arbitrary joins into arbitrary meets in the first coordinate, which is the infinitary version of B1.

	As to axiom B2, it is enough to show that for every $u\in \bbA^\delta$, $$0\bcp u = \top.$$
	Let us preliminarily show the identity above for $o\in O(\bbA^\delta)$. Notice that the set $\{a\mid a \in\bbA,  a\leq o\}$ is always nonempty since $\bot$ belongs to it. Hence,
	\begin{center}
		\begin{tabular}{r c l}
			$0\bcp o$ &$=$& $\bigvee \{0\bc a\mid a\in\bbA, a\leq o\}$\\
			&$=$& $\bigvee \{\top\mid a\in\bbA, a\leq o\}$\\
			&$=$& $\top$.\\
		\end{tabular}
	\end{center}
	Hence, for arbitrary $u\in\bbA^\delta$
	\begin{center}
		\begin{tabular}{r c l}
			$0\bcp u$ &$=$& $\bigwedge \{0\bcp o\mid o\in O(\bbA^\delta), u\leq o\}$\\
			&$=$& $\bigwedge \{\top\mid o\in O(\bbA^\delta), u\leq o\}$\\
			&$=$& $\top$.\\
		\end{tabular}
	\end{center}

	%As to axiom B2, it is enough to show that for every $u\in \bbA^\delta$, $$0\bcs u = \top.$$
	%	Let us preliminarily show the identity above for $k\in K(\bbA^\delta)$. %Notice that the set $\{a\mid a \in\bbA k\leq a\}$ is always nonempty since $\top$ belongs to it. Hence,
	%	\begin{center}
	%		\begin{tabular}{r c l}
	%			$0\bcs k$ &$=$& $\bigwedge \{0\bc a\mid a \in\bbA, k\leq a\}$\\
	%			&$=$& $\bigwedge \{\top\mid a \in\bbA, k\leq a\}$\\
	%			&$=$& $\top$.\\
	%		\end{tabular}
	%	\end{center}
	%	Hence, for arbitrary $u\in\bbA^\delta$ (notice that  $\{k\mid k\in K(\bbA^\delta), k\leq u\}$ is nonempty, since $\bot$ belongs to it),	
	%	\begin{center}
	%		\begin{tabular}{r c l}
	%			$0\bcs u$ &$=$& $\bigvee \{0\bcs k\mid k\in K(\bbA^\delta), k\leq u\}$\\
	%			&$=$& $\bigvee \{\top\mid k\in K(\bbA^\delta), k\leq u\}$\\
	%			&$=$& $\top$.\\
	%		\end{tabular}
	%	\end{center}

	As to B3, let us show that for all $q,w\in \bbQ^\delta$ and $u\in \bbA^\delta$, $$q\bcp w\bcp u\leq q\cdot w\bcp u.$$ Let us preliminarily  show that the inequality above is true for any $o\in O(\bba^\delta)$, $\kappa_1,\kappa_2\in K(\bbQ^\delta)$. By definition, if $o\in O(\bba^\delta)$ and $\kappa_1,\kappa_2\in K(\bbQ^\delta)$ then $\kappa_2\bcp o\in O(\bba^\delta)$ and $\kappa_1\cdot\kappa_2\in K(\bbQ^\delta)$.   Therefore:
	
	\begin{center}
		\begin{tabular}{c l l}
			&$\kappa_1 \bcp \kappa_2  \bcp o $&\\
			$= $&$\bigvee\{\alpha \bc a\mid a\in \bbA, \alpha\in\bbQ, \kappa_1\leq\alpha,  a\leq \kappa_2 \bcp o \}$& {\scriptsize(by definition)}\\
			$= $&$\bigvee\{\alpha \bc a\mid a\in \bbA, \alpha\in\bbQ, \kappa_1\leq\alpha,  a\leq \bigvee\{\beta \bc b\mid b\in \bbA, b\leq o,\beta\in\bbQ,\kappa_2\leq\beta\}\}$& {\scriptsize(by definition)}\\
			$=$ & $\bigvee\{\alpha \bc \beta\bc a\mid a\in \bbA,  a\leq o,\alpha,\beta\in\bbQ,\kappa_1\leq\alpha,\kappa_2\leq\beta\}$& {\scriptsize($\ast$)}\\
			$\leq$ &$\bigvee\{\alpha\cdot\beta\bc a\mid a\in \bbA, a\leq o,\alpha,\beta\in\bbQ,\kappa_1\leq\alpha,\kappa_2\leq\beta \}$& {\scriptsize(B3 holds in $\bba$)}\\
			$\leq$ &$\bigvee\{\gamma\bc a \mid a\in \bbA, a\leq o,\gamma\in\bbQ, \kappa_1\cdot\kappa_2\leq \gamma\}$ & {\scriptsize($\ast\ast$)}\\
			$=$ & $\kappa_1\cdot \kappa_2  \bcp o$& {\scriptsize(by definition)}
			%& $\geq $&$\bigwedge\{\alpha\bc \beta\bc  a\mid a\in\bbA, k\leq a \}$.\\ &&\\
		\end{tabular}
	\end{center}
	Let us prove the equality marked with $(\ast)$. If $a\in\bba$, $\beta\in\bbQ$, $a\leq o$ and $\kappa\leq\beta$, then $\beta\bc a\in \bbA$ and $\beta\bc a\in\{\beta \bc b\mid b\in \bbA, b\leq o,\beta\in\bbQ,\kappa_2\leq\beta\}$, hence $\beta\bc a\leq\bigvee\{\beta \bc b\mid b\in \bbA, b\leq o,\beta\in\bbQ,\kappa_2\leq\beta\}$. This, in turn, implies that $$\alpha\bc\beta\bc a\in \{\alpha \bc a\mid a\in \bbA, \alpha\in\bbQ, \kappa_1\leq\alpha,  a\leq \bigvee\{\beta \bc b\mid b\in \bbA, b\leq o,\beta\in\bbQ,\kappa_2\leq\beta\}\}.$$ Therefore
\begin{center}
	\begin{tabular}{r c l l}
		&       & $\{\alpha \bc \beta\bc a\mid a\in \bbA,  a\leq o,\alpha,\beta\in\bbQ,\kappa_1\leq\alpha,\kappa_2\leq\beta\}$\\
		&$\subseteq$ & $\{\alpha \bc a\mid a\in \bbA, \alpha\in\bbQ, \kappa_1\leq\alpha,  a\leq \bigvee\{\beta \bc b\mid b\in \bbA, b\leq o,\beta\in\bbQ,\kappa_2\leq\beta\}\}$\\
	\end{tabular}
\end{center}
and thus
\begin{center}
\begin{tabular}{r c l l}
 &       & $\bigvee\{\alpha \bc \beta\bc a\mid a\in \bbA,  a\leq o,\alpha,\beta\in\bbQ,\kappa_1\leq\alpha,\kappa_2\leq\beta\}$\\
 &$\leq$ & $\bigvee\{\alpha \bc a\mid a\in \bbA, \alpha\in\bbQ, \kappa_1\leq\alpha,  a\leq \bigvee\{\beta \bc b\mid b\in \bbA, b\leq o,\beta\in\bbQ,\kappa_2\leq\beta\}\}.$\\
\end{tabular}
\end{center}
  To prove the converse inequality, it is enough to show that if $a\in\bba$ and $a\leq \bigvee\{\beta \bc b\mid b\in \bbA, b\leq o,\beta\in\bbQ,\kappa_2\leq\beta\}$, then $\alpha \bc a\leq \alpha \bc \beta\bc b$ for some $b\in \bbA$ such that  $b\leq o$ and some $\beta\in\bbQ$ such  that $\kappa_2\leq\beta$.  By compactness (cf.\ Footnote \ref{def:can:ext}), $a\leq \bigvee\{\beta \bc b\mid b\in \bbA, b\leq o,\beta\in\bbQ,\kappa_2\leq\beta\}$ implies that  $a\leq \bigvee\{\beta_i \bc b_i\mid  1\leq i\leq n\}$ for some $b_i\in \bbA$, $\beta_i\in\bbQ$ such that $b_i\leq o$, $\kappa_2\leq\beta_i$  for every $1\leq i\leq n$.  Since $\bc$ is monotone in its second coordinate and antitone in its first, this implies that $$a\leq\beta_1\bc b_1\lor\ldots\lor \beta_n\bc b_n\leq(\beta_1\rand\ldots\rand\beta_n)\bc(b_1\lor\ldots\lor b_n).$$ Let $b: = b_1\lor\ldots\lor b_n$ and $\beta=\beta_1\rand\ldots\rand\beta_n$. By definition, $b\in \bbA$, $\beta\in\bbQ$ and  $b\leq o$, $\kappa_2\leq\beta$. Moreover, again by monotonicity, the displayed inequality implies that $\alpha \bc a\leq \alpha \bc \beta\bc b$, as required.  This finishes the proof of ($\ast$). % $\alpha\bc a\leq\alpha\bc b_1\lor\ldots\lor b_n$ implies that $$\bigvee\{\alpha \bc a\mid a\in \bbA,   a\leq \bigvee\{\beta \bc b\mid b\in \bbA, b\leq o\}\}\leq\bigvee\{\alpha \bc \beta\bc a\mid a\in \bbA,   a\leq o \}.$$
	The inequality marked with $(\ast\ast)$ holds since if $\kappa_1\leq\alpha$ and $\kappa_2\leq\beta$ then $\kappa_1\cdot\kappa_2\leq\alpha\cdot\beta$, so $\alpha\cdot\beta\bc a\in \{\gamma\bc a \mid a\in \bbA, a\leq o,\gamma\in\bbQ, \kappa_1\cdot\kappa_2\leq \gamma\}$ and therefore $$\{\alpha\cdot\beta\bc a\mid a\in \bbA, a\leq o,\alpha,\beta\in\bbQ,\kappa_1\leq\alpha,\kappa_2\leq\beta \}\subseteq \{\gamma\bc a \mid a\in \bbA, a\leq o,\gamma\in\bbQ, \kappa_1\cdot\kappa_2\leq \gamma\}.$$
	Let us show that B3 holds for arbitrary $u\in\bbA^\delta$ and $q,w\in\bbQ^\delta$.
	
	\begin{center}
		\begin{tabular}{c l l}
			&$q \bcp w  \bcp u $&\\
			$= $&$\bigwedge\{\kappa_1 \bcp o\mid o\in O(\bbA^\delta),\kappa_1\in K(\bbQ^\delta),\kappa_1\leq q,  w \bcp u\leq o \}$&{\scriptsize(by definition)}\\
			$= $&$\bigwedge\{\kappa_1 \bcp o\mid o\in O(\bbA^\delta),\kappa_1\in K(\bbQ^\delta),\kappa_1\leq q,  \bigwedge\{\kappa_2\bcp o'\mid u\leq o',\kappa_2\leq w\}\leq o \}$& {\scriptsize(by definition)}\\
			$\leq$ & $\bigwedge\{\kappa_1 \bcp \kappa_2\bcp o\mid o\in O(\bbA^\delta), u\leq o,\kappa_1,\kappa_2\in K(\bbQ^\delta),\kappa_1\leq q,\kappa_2\leq w \}$& {\scriptsize($\ast\ast\ast$)}\\
			$\leq$ & $\bigwedge\{\kappa_1\cdot\kappa_2 \bcp o \mid o\in O(\bbA^\delta),   u\leq o, \kappa_1,\kappa_2\in K(\bbQ^\delta),\kappa_1\leq q,\kappa_2\leq w \}$& {\scriptsize($\dagger$)}\\
			$=$ & $\bigwedge\{\left(\bigvee\{\kappa_1\cdot\kappa_2\mid \kappa_1,\kappa_2\in K(\bbQ^\delta),\kappa_1\leq q,\kappa_2\leq w \}\right) \bcp o \mid o\in O(\bbA^\delta),   u\leq o\}$& {\scriptsize($\ddagger$)}\\
			$=$ & $\bigwedge\{q\cdot w \bcp o \mid o\in O(\bbA^\delta),   u\leq o\}$& {\scriptsize(by definition)}\\
			$=$ & $q\cdot w\bcp u$ & {\scriptsize(by definition)}
			%& $\geq $&$\bigwedge\{\alpha\bc \beta\bc  a\mid a\in\bbA, k\leq a \}$.\\ &&\\
		\end{tabular}
	\end{center}
	The inequality marked with $(\ast\ast\ast)$ holds since, for any $o\in O(\bba^\delta)$ and $\kappa\in K(\bbQ^\delta)$, if  $u\leq o$ and $\kappa\leq w$ then $\kappa\bcp o\in O(\bba^\delta)$ and $\kappa\bcp o\in\{\kappa_2\bcp o'\mid u\leq o',\kappa_2\leq w\}$, hence $\bigwedge\{\kappa_2\bcp o'\mid u\leq o',\kappa_2\leq w\}\leq \kappa\bcp o$. This implies that $$\kappa_1 \bcp \kappa_2\bcp o\in \{\kappa_1 \bcp o\mid o\in O(\bbA^\delta),\kappa_1\in K(\bbQ^\delta),\kappa_1\leq q,  \bigwedge\{\kappa_2\bcp o'\mid u\leq o',\kappa_2\leq w\}\leq o \}$$ and therefore
	\begin{center}
		\begin{tabular}{r c l l}
			&       & $\{\kappa_1 \bcp \kappa_2\bcp o\mid o\in O(\bbA^\delta), u\leq o,\kappa_1,\kappa_2\in K(\bbQ^\delta),\kappa_1\leq q,\kappa_2\leq w \}$\\
			&$\subseteq$ & $\{\kappa_1 \bcp o\mid o\in O(\bbA^\delta),\kappa_1\in K(\bbQ^\delta),\kappa_1\leq q,  \bigwedge\{\kappa_2\bcp o'\mid u\leq o',\kappa_2\leq w\}\leq o \}$\\
		\end{tabular}
	\end{center}
which implies that
	\begin{center}
	\begin{tabular}{r c l l}
		&       & $\bigwedge \{\kappa_1 \bcp o\mid o\in O(\bbA^\delta),\kappa_1\in K(\bbQ^\delta),\kappa_1\leq q,  \bigwedge\{\kappa_2\bcp o'\mid u\leq o',\kappa_2\leq w\}\leq o \}$\\
		&$\leq$ & $\bigwedge \{\kappa_1 \bcp \kappa_2\bcp o\mid o\in O(\bbA^\delta), u\leq o,\kappa_1,\kappa_2\in K(\bbQ^\delta),\kappa_1\leq q,\kappa_2\leq w \}$.\\
	\end{tabular}
\end{center}
 The inequality marked with $(\dagger)$ holds since as we showed above B3 holds for any $o\in O(\bba^\delta)$, $\kappa_1,\kappa_2\in K(\bbQ^\delta)$. The equality marked with ($\ddagger$) holds because $\bcp$ is completely join reversing in the first coordinate.

	As to axiom BD1, let us show that for any $q \in \bbQ^\delta$ and  $u\in \bbA^\delta$, $$\drs q\land  q \bcp u\leq  \dcs u.$$
	Let us preliminarily  show that the inequality above is true for any $o\in O(\bba^\delta)$ and $\kappa\in K(\bbQ^\delta)$:
	\begin{center}
		\begin{tabular}{c l l}
			& $\drs\kappa\land \kappa \bcp o$&\\
			$= $&$\bigwedge\{\dr\beta\mid \beta\in\bbQ,\kappa\leq\beta\}\land\bigvee\{\alpha \bc a \mid a\in\bbA, a\leq o,\alpha\in\bbQ, \kappa\leq\alpha\}$ & {\scriptsize(by definition)}\\
			$= $&$\bigvee\{\left(\bigwedge\{\dr\beta\mid \beta\in\bbQ,\kappa\leq\beta\}\right)\land\alpha \bc a\mid a\in\bbA, a\leq o \}$&{\scriptsize(distributivity)}\\
			$\leq $&$\bigvee\{\dr\alpha\land\alpha \bc a\mid a\in\bbA, a\leq o \}$&{\scriptsize($\ast$)}\\
			$\leq $&$\bigvee\{\dc a\mid a\in\bbA, a\leq o \}$ &{\scriptsize(BD2 holds in $\bba$)}\\
			$=$ & $\dcs\bigvee\{ a\mid a\in\bbA, a\leq o \}$ & {\scriptsize($\dcs$ is completely join-preserving)}\\
			$=$ & $\dcs o$. & {\scriptsize(by definition)}
		\end{tabular}
	\end{center}
	The inequality marked with $(\ast)$ holds because if $\kappa\leq\alpha$, then $\dr\alpha\in\{\dr\beta\mid \beta\in\bbQ,\kappa\leq\beta\}$ and therefore $\bigwedge\{\dr\beta\mid \beta\in\bbQ,\kappa\leq\beta\}\leq\dr\alpha$.
	Let us show the inequality for  arbitrary $u\in\bbA^\delta$ and $q\in\bbQ^\delta$. In what follows, let $\adc$ denote the right adjoint of $\dcs$. It is well known (cf.~\cite[Lemma 10.3.3]{CoPa:non-dist}) that $\adc o\in O(\bba^\delta)$ for any $o\in O(\bba^\delta)$.
	
	\begin{center}
		\begin{tabular}{c l}
			&$\drs q \land q \bcp u$ \\
			$= $&$ \bigvee\{\drs\kappa\mid\kappa\in K(\bbQ^\delta),\kappa\leq q\}\land\bigwedge\{\kappa'\bcp o\mid o\in O(\bbA^\delta), u\leq o,\kappa'\in K(\bbQ^\delta),\kappa'\leq q\}$ {\scriptsize(by definition)}\\
			$= $&$\bigvee\{\drs\kappa\land\bigwedge\{\kappa'\bcp o\mid o\in O(\bbA^\delta), u\leq o,\kappa'\in K(\bbQ^\delta),\kappa'\leq q\}\mid\kappa\in K(\bbQ^\delta),\kappa\leq q\}$ {\scriptsize(distributivity)}\\
			$\leq $&$\bigvee\{\drs\kappa\land\bigwedge\{\kappa\bcp o\mid o\in O(\bbA^\delta), u\leq o\}\mid\kappa\in K(\bbQ^\delta),\kappa\leq q\}$ \ \ \qquad\qquad\qquad\qquad{\scriptsize($\ast$)}\\
			$\leq $&$\bigvee\{\drs\kappa\land\bigwedge\{\kappa\bcp \adc o\mid o\in O(\bbA^\delta), u\leq \adc o\}\mid\kappa\in K(\bbQ^\delta),\kappa\leq q\}$ \,\ \ \ \ \qquad\qquad\qquad {\scriptsize($\adc o\in O(\bba^\delta)$)}\\
			$\leq $&$\bigvee\{\bigwedge\{\drs\kappa\land\kappa\bcp \adc o\mid o\in O(\bbA^\delta), u\leq \adc o\}\mid\kappa\in K(\bbQ^\delta),\kappa\leq q\}$ \qquad\qquad\qquad\ \ \ \ \,{\scriptsize(distributivity)}\\
			$\leq $&$ \bigvee\{\bigwedge\{\dcs\adc o\mid o\in O(\bbA^\delta), u\leq\adc o\}\mid\kappa\in K(\bbQ^\delta),\kappa\leq q\}$ \qquad\qquad\qquad\qquad\,\,\,\,{\scriptsize(BD2 holds in $O(\bba^\delta)$)}\\
			$=$ & $\bigwedge\{\dcs\adc o\mid o\in O(\bbA^\delta), u\leq\adc o\}$ \qquad\qquad\qquad\qquad\qquad\qquad\qquad\qquad\ \ \ \ {\scriptsize($\dcs\adc o$ does not contain $\kappa$)}\\
			$\leq$&$ \bigwedge\{o\in O(\bbA^\delta)\mid u\leq\adc o\}$ \qquad\qquad\qquad\qquad\qquad\qquad\qquad\qquad\qquad\qquad\qquad\qquad{\scriptsize($\dcs\adc o\leq o$)}\\
			$=$&$ \bigwedge\{o\in O(\bbA^\delta)\mid \dcs u\leq o\}$ \qquad\qquad\qquad\qquad\qquad\qquad\qquad\qquad\qquad\qquad\qquad\ \ \ \ {\scriptsize(by adjunction)}\\
			$= $&$\dcs u$. \qquad\qquad\qquad\qquad\qquad\qquad\qquad\qquad\qquad\qquad\qquad\qquad\qquad\qquad\qquad\ \ \ \ \ \,{\scriptsize(by denseness)}\\
		\end{tabular}
	\end{center}
	The inequality marked with $(\ast)$ holds because if $\kappa\leq q$ and $u\leq o$, then $$\kappa\bcp o\in \{\kappa'\bcp o\mid o\in O(\bbA^\delta), u\leq o,\kappa'\in K(\bbQ^\delta),\kappa'\leq q\}$$ and therefore $$\{\kappa\bcp o\mid o\in O(\bbA^\delta), u\leq o\}\subseteq\{\kappa'\bcp o\mid o\in O(\bbA^\delta), u\leq o,\kappa'\in K(\bbQ^\delta),\kappa'\leq q\}$$ which yields $$\bigwedge \{\kappa'\bcp o\mid o\in O(\bbA^\delta), u\leq o,\kappa'\in K(\bbQ^\delta),\kappa'\leq q\}\leq\bigwedge\{\kappa\bcp o\mid o\in O(\bbA^\delta), u\leq o\}.$$

	Finally for axiom BD2 let us show that for any $q,w \in \bbQ^\delta$, $$q\brp w\leq  q\bcp \drs w.$$ Let us preliminarily  show that the inequality above is true for any $\omega\in O(\bbQ^\delta)$ and $\kappa\in K(\bbQ^\delta)$. Notice that since $\drs$ is completely join preserving, if $\omega\in O(\bbQ^\delta)$ then $\drs\omega\in O(\bbA^\delta)$.
	\begin{center}
		\begin{tabular}{c l l}
		& $\kappa\brp \omega$&\\
		$=$ & $\bigvee\{\alpha\br\beta\mid \alpha,\beta\in\bbQ, \kappa\leq\alpha,\beta\leq\omega\}$ &\scriptsize{(by definition)}\\
		$=$ & $\bigvee\{\alpha\bc\dr\beta\mid \alpha,\beta\in\bbQ, \kappa\leq\alpha,\beta\leq\omega\}$ &\scriptsize{(BD2 holds in $\bbA$)}\\
		$\leq$ & $\bigvee\{\alpha\bc a\mid \alpha\in\bbQ, a\in\bbA, \kappa\leq\alpha,a\leq\drs\omega\}$ &\scriptsize{($\ast$)}\\
		$=$ & $\kappa\bcp\drs\omega$ &\scriptsize{(by definition)	}
	\end{tabular}
\end{center}
	The inequality marked with $(\ast)$ holds because if $\beta\leq\omega$ then $\dr\beta\leq\drs\omega$, thus if $\kappa\leq\alpha$ we have $$\alpha\bc\dr\beta\in\{\alpha\bc a\mid \alpha\in\bbQ, a\in\bbA, \kappa\leq\alpha,a\leq\drs\omega\}$$ and therefore  $$\{\alpha\bc\dr\beta\mid \alpha,\beta\in\bbQ, \kappa\leq\alpha,\beta\leq\omega\} \subseteq\{\alpha\bc a\mid \alpha\in\bbQ, a\in\bbA, \kappa\leq\alpha,a\leq\drs\omega\}.$$
	Let us show the inequality for  arbitrary $q,w\in\bbQ^\delta$. In what follows, let $\adr:\bbA^\delta\to\bbQ^\delta$ denote the right adjoint of $\drs$.
	\begin{center}
		\begin{tabular}{c l l}
			& $q\brp w$&\\
			$=$ & $\bigwedge\{\kappa\brp\omega\mid \kappa\in K(\bbQ^\delta),\omega\in O(\bbQ^\delta), \kappa\leq q,w\leq\omega\}$ &\scriptsize{(by definition)}\\
			$\leq$ & $\bigwedge\{\kappa\brp\adr o\mid \kappa\in K(\bbQ^\delta),o\in O(\bbA^\delta), \kappa\leq q,w\leq\adr o\}$ &\scriptsize{($\adr o\in O(\bbQ^\delta)$)}\\
			$\leq$ & $\bigwedge\{\kappa\bcp\drs\adr o\mid \kappa\in K(\bbQ^\delta),o\in O(\bbA^\delta), \kappa\leq q,w\leq\adr o\}$ &\scriptsize{(BD2 holds for $\omega\in O(\bbQ^\delta)$ and $\kappa\in K(\bbQ^\delta)$)}\\
			$\leq$ & $\bigwedge\{\kappa\bcp o\mid \kappa\in K(\bbQ^\delta),o\in O(\bbA^\delta), \kappa\leq q,w\leq\adr o\}$ &\scriptsize{($\drs\adr o\leq o$)}\\
			$=$ & $\bigwedge\{\kappa\bcp o\mid \kappa\in K(\bbQ^\delta),o\in O(\bbA^\delta), \kappa\leq q,\drs w\leq o\}$ &\scriptsize{(by adjunction)}\\
			$=$ & $q\bcp\drs w$ &\scriptsize{(by definition)}	
		\end{tabular}
	\end{center}
\end{proof}

As an immediate consequence of Proposition \ref{prop:algebraic completeness} and Lemma \ref{lem:canonicity} we get the following

\begin{cor}
	\label{cor:strong completeness perfect algebras}
	The logic LRC is strongly complete w.r.t.\ the class of perfect heterogeneous LRC-algebras. %, and hence also w.r.t.\ the class of their dual relational structures.
\end{cor}

\subsection{Disjunction property}
\label{ssec:disjunction property}
In the present section, we show that  the disjunction property holds for LRC, by   adapting the standard argument to the setting of heterogeneous LRC-algebras.
For any heterogeneous LRC-algebra $F=(\bbA,\bbQ,\bc, \dc, \br, \dr)$, we let $F^\ast:=(\bbA^\ast,\bbQ,\bc^\ast, \dc^\ast, \br^\ast, \dr^\ast)$, where:
\begin{enumerate}
\item $\bbA^\ast$ is the Heyting algebra obtained by adding  a new top element $\top^\ast$ to $\bbA$ (we let  $\top_{\bbA}$ denote the top element of $\bbA$). Joins and meets in $\bbA^\ast$ are defined as expected. The implication $\to^\ast$ of $\bbA^\ast$ maps any $(u, w)\in \bbA^\ast\times \bbA^\ast$ to $\top^\ast$ if $u\leq w$, to $w$ if $u=\top^\ast$, and to $u\to w$ in any other case.

\item $\bc^{\!\ast}: \bbQ\times \bbA^\ast\to \bbA^\ast$ maps any $(\alpha, u)$  to $\top^\ast$ if $\alpha=0$ or $\dr^\ast1\leq u$, and to $\alpha\bc u$ otherwise.

\item $\dc^{\ast}: \bbA^\ast\to \bbA^\ast$ maps any $u$ to $\dc u$ if $u\neq \top^\ast$, and to $\dc\top_{\bbA}$ if $u=\top^\ast$.

\item $\br^{\!\ast}: \bbQ\times \bbQ\to \bbA^\ast$ maps any $(\alpha, \beta)$ to $\top^\ast$ if $\alpha=0$ or $\beta=1$, and to $\alpha\br\beta$ otherwise.

\item $\dr^{\ast}: \bbQ\to \bbA^\ast$ maps any $\alpha$ to  $\dr\alpha$.
\end{enumerate}

%It is trivial to check that $\bbA^\ast$ is a Heyting algebra.

\begin{lemma}
	$F^\ast$ is a heterogeneous LRC-algebra.
\end{lemma}
\begin{proof}
It can be easily verified that the maps $\bc^\ast, \dc^\ast, \br^\ast, \dr^\ast$ satisfy by definition all the monotonicity (resp.~antitonicity) properties that yield the validity of the rules of LRC. Let us verify that $F^\ast$ validates all the axioms  of LRC. By construction, $\top^\ast$ is  join-irreducible, i.e.~if $u\lor w=\top^\ast$ then either $u = \top^\ast$ or $w = \top^\ast$. Hence,  $\dc^{\ast}(u\lor w) = \dc^\ast\top^\ast = \dc\top_{\bbA} = \dc^\ast u\vee \dc^\ast w$. All the remaining cases follow from the assumptions on $\dc$. This finishes the verification of the validity of $D1$. The validity of axioms $D2$, $D3$ and $D4$ immediately follows from their validity in  $F$. The validity of axiom $B1$ can be shown using the identities  $\alpha\ror 0=\alpha$ and $0\ror\beta=\beta$. The validity of $B2$ follows immediately from the definition of $\bc^{\!\ast}$. As to $B3$, if $\alpha = 0$ or $\beta = 0$, the assumption that $\cdot$ preserves finite joins in each coordinate yields $\alpha\cdot\beta = 0$, and hence $\alpha\cdot\beta\bc^{\!\ast} A = \top^\ast$, which implies that the inequality holds.  The remaining cases follow from the definition of $\bc^{\!\ast}$ and the assumption that $B3$ is valid in $F$. Axiom $B4$ is argued similarly to $B1$. The validity of axioms $B5$ and $B7$ follows immediately from the definition of $\br$, and the validity of $B6$ can be shown using the identities $\alpha\rand 1=\alpha=1\rand\alpha$.
	
	As to  $BD2$,  if $\alpha=0$ or $\beta=1$  then $\alpha\bc^{\!\ast}\dr^\ast\beta=\top^\ast$, therefore the inequality holds.  All the remaining cases follow from the assumption that $BD2$ is valid in $F$.
	
	As to $BD1$, if $\alpha=0$ then $\dr^\ast \alpha \land \alpha \bc^{\!\ast} u  =\bot$ for any $u$, therefore the inequality holds. If $\dr^\ast 1\leq u$ then, by definition, $\alpha \bc^{\!\ast} u = \top^\ast$, hence it is enough to show that $\dr^\ast \alpha \leq \dc^\ast u$. We proceed by cases: (a) if $u = \top^\ast$, then $\dr^\ast \alpha = \dr\alpha\leq \top^\bbA = \dc^\ast u$, as required; (b) if $u\in \bbA$, then,   by the assumption that $B7$, $BD2$  and $MB$  hold in $F$,
 \[\top^\bbA\leq \alpha \br 1 \leq \alpha \bc \dr 1 \leq \alpha \bc u. \]
  Since $BD1$ holds in $F$,  this implies that $\dr^\ast\alpha = \dr\alpha\leq\dc u = \dc^\ast u$, as required. %which implies the validity of $BD1$ in $F^\ast$.
  All the remaining cases follow from the assumption that $BD1$ is valid in $F$.
\end{proof}

For every algebraic LRC-model $\bbM = (F, v_{\mathsf{Fm}}, v_{\mathsf{Res}})$, we let $\bbM^\ast: = (F^\ast, v_{\mathsf{Fm}}^\ast, v_{\mathsf{Res}})$, where $v_{\mathsf{Fm}}^\ast$ is defined by composing $v_{\mathsf{Fm}}$  with the natural injection $\bbA\hookrightarrow\bbA^\ast$. Henceforth, we  let $\val{a}$ denote the interpretation of any $\mathsf{T}$-term $a$ in $\mathbb{M}$ and $\val{a}_\ast$  the interpretation of $a$ in $\mathbb{M}^\ast$.

\begin{lemma}\label{lemma:disjunctionkey}
For every formula $A$,
\begin{enumerate}
\item  If $\val{A}_\ast \neq \top^\ast$ then $\val{A}_\ast=\val{A}$.
\item If $\val{A}_\ast=\top^\ast$ then $\val{A}=\top_\bbA$.
\end{enumerate}
\end{lemma}
\begin{proof}
We prove the two statements simultaneously	by induction on  $A$. The  cases of constants and atomic variables  are straightforward. The case of $A = B\land C$  immediately follows from the induction hypothesis. The case of  $A = B\lor C$ immediately follows from the induction hypothesis using  the join-irreducibility of $\top^\ast$. If $A=B\to C$, then $\val{A}_{\ast} = \val{B}_{\ast}\to^\ast\val{C}_{\ast}$. By definition of $\to^\ast$, if $\val{A}_\ast \neq \top^\ast$ then either (a) $\val{B}_{\ast} \not\leq \val{C}_{\ast}$ and $\val{B}_{\ast} \neq \top^\ast$, which implies that $\val{B}_{\ast} \neq \top^\ast\neq \val{C}_{\ast}$ in which case item 1 follows by induction hypothesis; or (b) $\val{B}_{\ast} \not\leq \val{C}_{\ast}$ and $\val{C}_{\ast}\neq\top^\ast$, which implies that $\val{C}_{\ast} = \val{C}$ by induction hypothesis. Then either (b1) $\val{B}_\ast = \top^\ast$, hence by induction hypothesis $\val{B} = \top^\bbA$ and $\val{A}_\ast = \val{C}_{\ast} = \val{C} = \val{A}$, as required; or (b2) $\val{B}_\ast \neq \top^\ast$, hence by induction hypothesis $\val{B}_{\ast} = \val{B}$ and we finish the proof as in case (a).  If $\val{A}_\ast=\top^\ast$, then either (c) $\val{B}_{\ast}  = \top^\ast = \val{C}_{\ast}$, which implies by induction hypothesis that  $\val{B}=\top_{\bbA} = \val{C}$, which %hence $\val{A}_\ast=\val{C}_\ast$ and $\val{A}=\val{C}$, and the induction hypothesis
yields $\val{A} = \top^\bbA$, as required; or (d) $\val{B}_{\ast} \leq \val{C}_{\ast}$, which implies %. Then either  If $\val{C}_\ast=\top^\ast$ then $\val{C}=\top$, hence $\val{A}_\ast=\top^\ast$ and $\val{A}=\top$. Finally assume that $\val{B}_\ast\neq\top^\ast$, which implies that $\val{B}_\ast=\val{B}$ (and likewise for $C$). If $\val{B}_\ast\leq\val{B}_\ast$ then
$\val{B}\leq\val{C}$ and hence $\val{A}=\top_\bbA$, as required. %If  $\val{B}_\ast\nleq\val{B}_\ast$ then $\val{A}_\ast=\val{A}$, which yields what we want.
	
	If  $A=\dc B$, then  $\val{A}_\ast=\dc^\ast\val{B}_\ast$. The definition of $\dc^\ast$ implies that $\val{A}_\ast\neq \top^\ast$, hence to finish the proof of this case we need to show that $\val{A}_{\ast} = \val{A}$.  If $\val{B}_\ast\neq \top^\ast$, then by induction hypothesis $\val{B}_\ast = \val{B}$, hence  $\val{A}_\ast=\dc^\ast\val{B}_\ast = \dc\val{B} = \val{\dc B} = \val{A}$, as required.  If  $\val{B}_\ast=\top^\ast$, then by induction hypothesis $\val{B}=\top^\bbA$, hence $\val{A}_\ast=\dc^\ast\val{B}_\ast = \dc^\ast\top^\ast =\dc\top^\bbA =  \dc\val{ B} =\val{\dc B} = \val{A}$, as required. %$\val{A}=\dc\val{B} = \dc\top^\bbA = \val{\dc B} = \val{A}$  then $\val{A}_\ast=\val{A}$. If $\val{B}_\ast=\top^\ast$ then $\val{B}=\top_\bbA$, and therefore $\val{A}_\ast=\val{A}$ by the definition of $\dc$.
	
	If $A=\dr\alpha$, item 2 is again vacuously true, and item 1   immediately follows from  the definition of $\dr^\ast$.
	
	If $A=\alpha\br\beta$, then $\val{A}_\ast=\val{\alpha}_\ast\br^{\!\ast}\val{\beta}_\ast = \val{\alpha}\br^{\!\ast}\val{\beta}$. Then by definition of $\br^\ast$, if $\val{A}_\ast\neq\top^\ast$, then $\val{A}_\ast=\val{A}$, as required, and if  $\val{A}_\ast = \top^\ast$, then either $\val{\alpha} = 0$ or $\val{\beta} = 1$; since axioms B5 and B7 hold in $F$, each case yields $\val{A} = \top^\bbA$, as required.
	
	Finally, if $A=\alpha\bc B$, then $\val{A}_\ast=\val{\alpha}_\ast\bc^{\!\ast}\val{B}_\ast = \val{\alpha}\bc^{\!\ast}\val{B}_\ast$. By definition of $\bc^{\!\ast}$, if $\val{A}_\ast\neq \top^\ast$, then $\val{\alpha} \neq 0$, $\val{A}_\ast=\val{\alpha}\bc\val{B}_\ast$, and  $\dr^{\ast}1\not \leq \val{B}_\ast$. The latter condition implies that $\val{B}_{\ast}\neq \top^\ast$, hence, by induction hypothesis, $\val{B}_{\ast} = \val{B}$, and so $\val{A}_\ast=\val{\alpha}\bc\val{B} = \val{A}$, as required.  If  $\val{A}_\ast = \top^\ast$, then either (a) $\val{\alpha}=0$, which implies by B2 that $\val{A} = \val{\alpha}\bc\val{B} = \top^\bbA$, as required; or (b) $\dr^{\ast}1 \leq \val{B}_\ast$, which implies by induction hypothesis that $\dr^{\ast}1 \leq \val{B}$. Hence, by  BD2 and monotonicity of $\bc$,
 %either (b1) $\val{B}_\ast = \top^\ast$, which implies by induction hypothesis that $\val{B} = \top^\bbA$, hence,   \[\top\leq \alpha\br 1\leq \alpha \bc\dr 1\leq \alpha\bc B \]
  %$\val{A}=\val{\alpha}\bc\val{B}$???; or (b2) $\val{B}_\ast \neq \top^\ast$, which implies by induction hypothesis that $\val{B}_\ast = \val{B}$, hence $\dr^{\ast}1 \leq \val{B}$, which implies that
 \[\top^\bbA\leq \val{\alpha}\br 1\leq \val{\alpha} \bc\dr 1\leq \val{\alpha}\bc \val{B}, \]
  which finishes the proof that $\val{A} = \val{\alpha\bc B} = \top^\bbA$, as required. %then $\val{A}_\ast=\top^\ast$ and $\val{A}=\top_\bbA$. Notice that $\val{\dr 1}_\ast=\val{\dr 1}$. We have $\val{B}\geq\val{\dr 1}$ if and only if $\val{B}_\ast\geq\val{\dr 1}$. This is enough to show what we want.
\end{proof}

 The {\em product} $F_1\times F_2$ of the heterogeneous LRC-algebras $F_1$ and $F_2$ is defined in the expected way, based on the product algebras
 $\bbA_1\times \bbA_2$ and $\bbQ_1\times \bbQ_2$, and  defining all (i.e.~both internal and external) operations component-wise. It can be readily verified that the resulting construction is a heterogeneous LRC-algebra. The product construction can be extended to algebraic LRC-models in the expected way, i.e.~by pairing the valuations. Such valuations extend as usual to $\mathsf{T}$-terms, and it can be proved  by a straightforward induction that $\val{a}_\times=(\val{a}_1,\val{a}_2)$.

\begin{prop}
	The disjunction property holds for the logic LRC.
\end{prop}
\begin{proof}
	If $B$ and $C$ are not LRC-theorems, by completeness, algebraic LRC-models $\mathbb{M}_1$ and $\mathbb{M}_2$ exist such that $\val{B}_1\neq\top_1$ and $\val{C}_2\neq\top_2$. Consider  the product model $\mathbb{M}: = \mathbb{M}_1\times \mathbb{M}_2$ as described above. Notice that $\val{B}\neq(\top_1,\top_2)$ and likewise for $C$. The model $\mathbb{M}^\ast$ does not satisfy $B\lor C$. Indeed, since $\top^\ast$ is join-irreducible,  if $\val{B\lor C}_\ast=\top^\ast$ then either $\val{B}_\ast=\top^\ast$ or $\val{C}_\ast=\top^\ast$. By Lemma \ref{lemma:disjunctionkey} this implies that either $\val{B}=(\top_1,\top_2)$ or $\val{C}=(\top_1,\top_2)$,  contradicting the assumptions.
\end{proof}

\section{The  calculus D.LRC}
\label{sec: Display-style sequent calculus D.LRC}

In the present section, we introduce the multi-type calculus D.LRC for the logic of resources and capabilities. As is typical of similar existing calculi, the language manipulated by this calculus is built up from structural  and  operational term constructors. In the tables below, each structural symbol in the upper rows corresponds to one or two logical (aka operational) symbols in the lower rows. The idea, which will be made precise in Section \ref{ssec:soundness}, is that each structural connective is interpreted as the corresponding logical connective on the left-hand (resp.\ right-hand) side (if it exists) when occurring in antecedent (resp.\ consequent) position.

\noindent As discussed in the previous section, the mathematical environment of heterogeneous LRC-algebras provides  natural interpretations for all connectives  of the basic language of LRC. In particular, on {\em perfect} heterogeneous LRC-algebras, these interpretations have the following extra properties: the interpretations of $\dc$ and $\dr$ are completely join-preserving, that of $\bc$ is completely join-reversing in its first coordinate and order preserving in its second coordinate, and $\br$ is completely join-reversing in its first coordinate and completely meet-preserving in its second coordinate. This implies that, in each perfect heterogeneous LRC-algebra,
\begin{itemize}
 \item  $\dc$ and $\dr$ have right adjoints, denoted $\adc$ and $\adr$ respectively;
\item  $\bc$ has a Galois-adjoint $\abc$ in its first coordinate, and $\br$ has a Galois-adjoint $\abr$ in its first coordinate and a left adjoint $\labr$ in its second coordinate.
\end{itemize}
     Hence,  the following connectives have a natural interpretation on perfect heterogeneous LRC-algebras:
\begin{eqnarray}
\adc  & :&   \mathsf{Fm} \to \mathsf{Fm}                             \\
\adr  & :&   \mathsf{Fm} \to \mathsf{Res} \\
\abc & :&  \mathsf{Fm} \times \mathsf{Fm} \to \mathsf{Res}\\
\abr & :&  \mathsf{Fm} \times \mathsf{Res} \to \mathsf{Res}\\
\labr & :&  \mathsf{Res} \times \mathsf{Fm} \to \mathsf{Res}.
%{\mbla}_{\!1}, {\mla}_{\!1} & : & \mathsf{Fm}\times \mathsf{Fm} \to \mathsf{Act}\\
%{\mbla}_{\!0}, {\mla}_{\!0} & :&  \mathsf{Fm}  \times \mathsf{Fm} \to \mathsf{TAct}.
\end{eqnarray}

\begin{itemize}
	\item Structural and operational symbols for pure $\mathsf{Fm}$-connectives:
	\begin{center}
		\begin{tabular}{|r|c|c|c|c|c|c|c|c|}
			\hline
			\scriptsize{Structural symbols} & \mc{2}{c|}{I} & \mc{2}{c|}{$;$} & \mc{2}{c|}{$>$} & \mc{2}{c|}{$(<)$}\\
			\hline
			\scriptsize{Operational symbols} & $\top$ & $\bot$ & $\pand$ & $\por$ & $(\pdra)$ & $\pra$ & $(\pdla)$ & $(\leftarrow)$\\
			\hline
		\end{tabular}
	\end{center}
	
		\item Structural and operational symbols for pure $\mathsf{Res}$-connectives:
		\begin{center}
			\begin{tabular}{|r|c|c|c|c|c|c|c|c|c|c|c|c|c|c|c|c}
				\hline
				\scriptsize{Structural symbols} & \mc{2}{c|}{$\Phi$} & \mc{2}{c|}{$\odot$} & \mc{2}{c|}{$\RAND$}    & \mc{2}{c|}{$\RCDOT$} & \mc{2}{c|}{$\LCDOT$} & \mc{2}{c|}{$\RPLUS$} & \mc{2}{c|}{$(\LPLUS)$} \\
				\hline
				\scriptsize{Operational symbols} & $1$ & $0$ & $\cdot$ & \phantom{$\cdot$} & $\rand$ & $\ror$  & \phantom{$(_\cdot\!\backslash)$} & $(_\cdot\!\backslash)$ & \phantom{$(/_{\!\cdot})$} & $(/_{\!\cdot})$ & $(_\ror\!\backslash)$ & $(_\rand\!\backslash)$ & $(/_{\!\ror})$ & $(/_{\!\rand})$ \\
				\hline
			\end{tabular}
		\end{center}

	\item Structural and operational symbols for the modal operators:
		\begin{center}
			\begin{tabular}{|r|c|c|c|c|c|c|c|c|}
				\hline
				\scriptsize{Structural symbols} & \mc{2}{c|}{$\DC$} & \mc{2}{c|}{$\BC$} & \mc{2}{c|}{$\DR$} & \mc{2}{c|}{$\BR$} \\
				\hline
				\scriptsize{Operational symbols} & $\dc$ & $\phantom{\dc}$ & $\phantom{\bc}$ & $\bc$ & $\dr$ & $\phantom{\dr}$ & $\phantom{\br}$ & $\br$ \\
				\hline
			\end{tabular}
		\end{center}

	\item Structural and operational symbols for the adjoints and residuals of the modal operators:
\begin{center}
\begin{tabular}{|r|c|c|c|c|c|c|c|c|c|c|}
\hline
\scriptsize{Structural symbols}    & \mc{2}{c|}{$\ADC$}                   & \mc{2}{c|}{$\ABC$}                   & \mc{2}{c|}{$\ADR$} & \mc{2}{c|}{$\LABR$} & \mc{2}{c|}{$\ABR$} \\
\hline
\scriptsize{Operational symbols} & $\phantom{(\adc)}$ & $(\adc)$ & $\phantom{(\abc)}$ & $(\abc)$ & $\phantom{(\adr)}$& $(\adr)$ & $(\labr)$ & $\phantom{(\labr)}$ & $\phantom{(\abr)}$ & $(\abr)$ \\
\hline
\end{tabular}
\end{center}
\end{itemize}

The display-type calculus $\mathbf{D.LRC}$ consists of the following display postulates, structural rules, and operational rules:

			\begin{enumerate}
			
				\item Identity and cut rules: %\marginnote{Motivate our choice of Cut: not all cuts are sound.}
				
				\begin{center}
					\begin{tabular}{c c c}
						$p \vdash p$
						& &
						$a\vdash a$  \\
						& & \\
						\AX$(X \fCenter Y)[A]^{succ}$
						\AX$ A \fCenter Z$
						\BI$ (X \fCenter Y)[Z/A]^{succ}$
						\DisplayProof
						& &
						\AX$\Gamma \fCenter \alpha$
						\AX$ \alpha \fCenter \Delta$
						\BI$ \Gamma \fCenter \Delta$
						\DisplayProof
						 \\
						\end{tabular}
				\end{center}

				\item Display postulates for pure  $\mathsf{Fm}$-connectives:

				\begin{center}
					\begin{tabular}{cccc}
						\AX$X \,; Y \fCenter Z$
						\doubleLine
						\UI$Y \fCenter X > Z$
						\DisplayProof
						&
						\AX$Z \fCenter X \,; Y$
						\doubleLine
						\UI$X > Z \fCenter Y$
						\DisplayProof
						&
						\AX$X \,; Y \fCenter Z$
						\doubleLine
						\UI$X \fCenter Z < Y$
						\DisplayProof
						&
						\AX$Z \fCenter X \,; Y$
						\doubleLine
						\UI$Z < Y \fCenter X$
						\DisplayProof
						 \\
					\end{tabular}
				\end{center}

				\item Display postulates for pure $\mathsf{Res}$-connectives:
				\begin{center}
					\begin{tabular}{cccc}
						\AX$\Gamma \RAND \Delta \fCenter \Sigma$
						\doubleLine
						\UI$\Delta \fCenter \Gamma \RPLUS \Sigma$
						\DisplayProof
						&
						\AX$\Gamma \RAND \Delta \fCenter \Sigma$
						\doubleLine
						\UI$\Gamma \fCenter \Sigma \LPLUS \Delta$
						\DisplayProof
						&
						\AX$\Gamma \fCenter \Delta \RAND \Sigma$
						\doubleLine
						\UI$\Delta \RPLUS \Gamma \fCenter \Sigma$
						\DisplayProof
						&
						\AX$\Gamma \fCenter \Delta \RAND \Sigma$
						\doubleLine
						\UI$\Gamma \LPLUS \Sigma \fCenter \Delta$
						\DisplayProof
						\\

 & & & \\

						\mc{2}{c}{
						\AX$\Gamma \odot \Delta \fCenter \Sigma$
						\doubleLine
						\UI$\Delta \fCenter \Gamma \RCDOT \Sigma$
						\DisplayProof}
						&
						\mc{2}{c}{
						\AX$\Gamma \odot \Delta \fCenter \Sigma$
						\doubleLine
						\UI$\Gamma \fCenter \Sigma \lessdot \Delta$
						\DisplayProof}
						 \\
					\end{tabular}
				\end{center}

				\item Display postulates for the modal operators: % $\dc$, $\dr$, $\bc$ and $\br$:

				\begin{center}
					\begin{tabular}{ccccc}

							\AX$\DC X \fCenter Y$
							\doubleLine
							\UI$X \fCenter \ADC Y$
							\DisplayProof
						&
							\AX$\DR \Gamma \fCenter X$
							\doubleLine
							\UI$\Gamma \fCenter \ADR X$
							\DisplayProof
						&
							\AX$X \fCenter \Gamma \BC Y$
							\doubleLine
							\UI$\Gamma \fCenter X \ABC Y$
							\DisplayProof
						&
							\AX$X \fCenter \Gamma \BR \Delta$
							\doubleLine
							\UI$\Gamma \LABR X \fCenter \Delta$
							\DisplayProof
						&
							\AX$X \fCenter \Gamma \BR \Delta$
							\doubleLine
							\UI$\Gamma \fCenter X \ABR \Delta$
							\DisplayProof

						\\	
					\end{tabular}
				\end{center}

				\item Pure $\mathsf{Fm}$-type structural rules: %\marginnote{Add remark: if we add the classical Grishin rule, we obtain a calculus for classical logic.}
				\begin{center}
					\begin{tabular}{rlcrl}
						\AX$X \fCenter Y$
						\doubleLine
						\LeftLabel{\fns$\textrm{I}_{L}$}
						\UI$\textrm{I}\,; X \fCenter Y$
						\DisplayProof
						&
						\AX$Y \fCenter X$
						\doubleLine
						\RightLabel{\fns$\textrm{I}_{R}$}
						\UI$Y \fCenter X\,; \textrm{I}$
						\DisplayProof
						& &
						\AX$Y \,; X \fCenter Z$
						\LeftLabel{\fns$E_L$}
						\UI$X \,; Y \fCenter Z $
						\DisplayProof
						&
						\AX$Z \fCenter X \,; Y$
						\RightLabel{\fns$E_R$}
						\UI$Z \fCenter Y \,; X$
						\DisplayProof
						\\
						& & & & \\
						\AX$Y \fCenter Z$
						\LeftLabel{\fns$W_L$}
						\UI$X\,; Y \fCenter Z$
						\DisplayProof
						&
						\AX$Z \fCenter Y$
						\RightLabel{\fns$W_R$}
						\UI$Z \fCenter Y\,; X$
						\DisplayProof
						& &
						\AX$X \,; X \fCenter Y$
						\LeftLabel{\fns$C_L$}
						\UI$X \fCenter Y $
						\DisplayProof
						&
						\AX$Y \fCenter X \,; X$
						\RightLabel{\fns$C_R$}
						\UI$Y \fCenter X$
						\DisplayProof
						\\
						& & & & \\
						\mc{2}{c}{
							\AX$X \,; (Y \,; Z) \fCenter W$
							\doubleLine
							\LeftLabel{\fns$A_{L}$}
							\UI$(X \,; Y) \,; Z \fCenter W $
							\DisplayProof}
						& &
						\mc{2}{c}{
							\AX$W \fCenter (Z \,; Y) \,; X$
							\doubleLine
							\RightLabel{\fns$A_{R}$}
							\UI$W \fCenter Z \,; (Y \,; X)$
							\DisplayProof}
					\end{tabular}
				\end{center}

				\item Pure $\mathsf{Res}$-type structural rules:
				
				\begin{center}
					\begin{tabular}{cccc}
						\AX$\Gamma \odot \Phi \fCenter \Delta$
						\doubleLine
						\LeftLabel{\fns$\Phi_{L1}$}
						\UI$\Gamma \fCenter \Delta$
						\doubleLine
						\LeftLabel{\fns$\Phi_{L2}$}
						\UI$\Phi \odot \Gamma \fCenter \Delta$
						\DisplayProof
						&
						\AX$\Gamma \fCenter \Delta$
						\doubleLine
						\RightLabel{\fns$\Phi_{R}$}
						\UI$\Gamma \fCenter \Delta \RAND \Phi$
						\DisplayProof
						&
						\AX$\Gamma \odot (\Delta \odot \Sigma) \fCenter \Pi$
						\doubleLine
						\LeftLabel{\fns$A_L$}
						\UI$(\Gamma \odot \Delta) \odot \Sigma \fCenter \Pi$
						\DisplayProof
						&
						\AX$\Phi \fCenter \Delta$
						\LeftLabel{\fns$W_\Phi$}
						\UI$\Gamma \fCenter \Delta$
						\DisplayProof

						\\

 & & & \\

						\AX$\Gamma \fCenter \Delta$
						\LeftLabel{\fns$W_L$}
						\UI$\Gamma \RAND \Sigma \fCenter \Delta$
						\DisplayProof
						&
						\AX$\Gamma \fCenter \Delta$
						\RightLabel{\fns$W_R$}
						\UI$\Gamma \fCenter \Delta \RAND \Sigma$
						\DisplayProof
						&
						\AX$\Gamma \RAND \Gamma \fCenter \Delta$
						\LeftLabel{\fns$C_L$}
						\UI$\Gamma \fCenter \Delta$
						\DisplayProof
						&
						\AX$\Gamma \fCenter \Delta \RAND \Delta$
						\RightLabel{\fns$C_R$}
						\UI$\Gamma \fCenter \Delta$
						\DisplayProof
						\\

 & & & \\

						\AX$\Gamma \RAND \Delta \fCenter \Sigma$
						\LeftLabel{\fns$E_L$}
						\UI$\Delta \RAND \Gamma \fCenter \Sigma$
						\DisplayProof
						&
						\AX$\Sigma \fCenter \Gamma \RAND \Delta$
						\RightLabel{\fns$E_R$}
						\UI$\Sigma \fCenter \Delta \RAND \Gamma$
						\DisplayProof
						&
						\AX$\Sigma \RAND (\Delta \RAND \Gamma) \fCenter \Pi$
						\doubleLine
						\LeftLabel{\fns$A_L$}
						\UI$(\Sigma \RAND \Delta) \RAND \Gamma \fCenter \Pi$
						\DisplayProof
						&
						\AX$\Pi \fCenter \Sigma \RAND (\Delta \RAND \Gamma)$
						\doubleLine
						\RightLabel{\fns$A_R$}
						\UI$\Pi \fCenter (\Sigma \RAND \Delta) \RAND \Gamma$
						\DisplayProof
						\\

 & & & \\

						\mc{4}{c}{
							\AX$\Pi \fCenter (\Gamma \RCDOT \Delta) \RAND (\Gamma \RCDOT \Sigma)$
							\RightLabel{\fns$dis$}
							\UI$\Pi \fCenter \Gamma \RCDOT (\Delta \RAND \Sigma)$
							\DisplayProof}
						\\

					\end{tabular}
				\end{center}

\item Structural rules corresponding to the D-axioms:

\begin{center}
\begin{tabular}{cccc}
\bottomAlignProof
\AX$X \fCenter \ADC Y \,; \ADC Z$
\RightLabel{D1}
\UI$X \fCenter \ADC (Y \,; Z)$
\DisplayProof
 &
\bottomAlignProof
\AX$\Gamma \fCenter \ADR X \RAND \ADR Y$
\RightLabel{D3}
\UI$\Gamma \fCenter \ADR (X \,; Y)$
\DisplayProof
 &
\bottomAlignProof
\AX$X \fCenter \I$
\RightLabel{D2}
\UI$X \fCenter \ADC \I$
\DisplayProof
 &
\bottomAlignProof
\AX$\Gamma \fCenter \Phi$
\RightLabel{D4}
\UI$\Gamma \fCenter \ADR \I$
\DisplayProof
 \\
%\bottomAlignProof
%\AX$\I \fCenter X$
%\doubleLine
%\RightLabel{\textcolor{red}{D5}}
%\UI$\DC \I \fCenter X$
%\DisplayProof
% &
%\bottomAlignProof
%\AX$\I \fCenter X$
%\doubleLine
%\RightLabel{\textcolor{red}{D6}}
%\UI$\DR \Phi \fCenter X$
%\DisplayProof
% \\
\end{tabular}
\end{center}

\item Structural rules corresponding to the B-axioms:

\begin{center}
\begin{tabular}{ccc}
\bottomAlignProof
\AX$\Gamma \fCenter (Y \ABR \Delta) \RAND (Z \ABR \Delta)$
\RightLabel{\fns B4}
\UI$\Gamma \fCenter (Y \,; Z) \ABR \Delta$
\DisplayProof
%\bottomAlignProof
%\AX$\Gamma \fCenter \Phi$
%\RightLabel{B2}
%\UI$\Gamma \fCenter \I \ABC X$
%\DisplayProof
 &
\bottomAlignProof
\AX$\Gamma \fCenter (Y \ABC W) \RAND (Z \ABC W)$
\RightLabel{\fns B1}
\UI$\Gamma \fCenter (Y \,; Z) \ABC W$
\DisplayProof
 &
 \\
 & & \\
%\bottomAlignProof
%\AX$\Gamma \fCenter \Phi$
%\RightLabel{B5}
%\UI$\Gamma \fCenter \I \ABR \Delta$
%\DisplayProof

\bottomAlignProof
\AX$(\Gamma \LABR X) \RAND (\Gamma \LABR Y) \fCenter \Delta$
\LeftLabel{\fns B6}
\UI$\Gamma \LABR (X \,; Y) \fCenter \Delta$
\DisplayProof
 &
\bottomAlignProof
\AX$X \fCenter \Gamma \BC (\Delta \BC Y)$
\RightLabel{\fns B3}
\UI$X \fCenter \Gamma \odot \Delta \BC Y$
\DisplayProof
 &
\bottomAlignProof
\AX$\Phi \fCenter \Delta$
\LeftLabel{\fns B7}
\UI$\Gamma \LABR \I \fCenter \Delta$
\DisplayProof
% &
%\bottomAlignProof
%\AX$\Gamma \fCenter \Delta$
%\LeftLabel{\fns B8}
%\UI$\Gamma \LABR \I \fCenter \Delta$
%\DisplayProof
 \\
\end{tabular}
\end{center}

\item Structural rules corresponding to the BD-axioms:

\begin{center}
\begin{tabular}{cc }
%\bottomAlignProof
%\AX$X \fCenter \Gamma \BC (Y \,; \ADC (\DR \Gamma > Z))$
%\RightLabel{BD1}
%\UI$X \fCenter (\Gamma \BC Y) \,; Z$
%\DisplayProof
% &
\bottomAlignProof
\AX$X \fCenter \Gamma \BC \ADC Y$
\RightLabel{BD1}
\UI$X \fCenter \DR \Gamma > Y$
\DisplayProof
&
\bottomAlignProof
\AX$X \fCenter \Gamma \BR \, \ADR Y$
\RightLabel{BD2}
\UI$X \fCenter \Gamma \BC Y$
\DisplayProof
 \\
\end{tabular}
\end{center}

				\item Introduction rules for pure $\mathsf{Fm}$-connectives (in the presence of the exchange rules $E_L$ and $E_R$, the structural connective $<$ and the corresponding operational connectives $\pdla$ and $\pla$ are redundant and they are omitted):

				\begin{center}
				%\hspace{-2.8em}
					\begin{tabular}{rlrl}
						\AXC{\phantom{$\gbot \fCenter \textrm{I}$}}
						\LeftLabel{\fns$\bot_L$}
						\UI$\bot \fCenter \textrm{I}$
						\DisplayProof
						&
						\AX$X \fCenter \textrm{I}$
						\RightLabel{\fns$\bot_R$}
						\UI$X \fCenter \bot$
						\DisplayProof
						&
						\AX$\textrm{I} \fCenter X$
						\LeftLabel{\fns$\top_L$}
						\UI$\top \fCenter X$
						\DisplayProof
						&
						\AXC{\phantom{$\textrm{I} \fCenter \top$}}
						\RightLabel{\fns$\top_R$}
						\UI$\textrm{I} \fCenter \top$
						\DisplayProof
						\\
						& & & \\
						\AX$A \,; B \fCenter X$
						\LeftLabel{\fns$\pand_L$}
						\UI$A \pand B \fCenter X$
						\DisplayProof
						&
						\AX$X \fCenter A$
						\AX$Y \fCenter B$
						\RightLabel{\fns$\pand_R$}
						\BI$X \,; Y \fCenter A \pand B$
						\DisplayProof
						&
						\AX$A \fCenter X$
						\AX$B \fCenter Y$
						\LeftLabel{\fns$\por_L$}
						\BI$A \por B \fCenter X \,; Y$
						\DisplayProof
						&
						\AX$X \fCenter A \,; B$
						\RightLabel{\fns$\por_R$}
						\UI$X \fCenter A \por B$
						\DisplayProof
						\\
						& & & \\
						\mc{2}{r}{\AX$X \fCenter A$
						\AX$B \fCenter Y$
						\LeftLabel{\fns$\pra_L$}
						\BI$A \pra B \fCenter X > Y$
						\DisplayProof}
						&
						\mc{2}{l}{\AX$X \fCenter A > B$
						\RightLabel{\fns$\pra_R$}
						\UI$X \fCenter A \pra B$
						\DisplayProof}
						 \\
						%&
						%\AX$A > B \fCenter Z$
						%\LeftLabel{\fns$(\pdra_L)$}
						%\UI$A \pdra B \fCenter Z$
						%\DisplayProof
						%&
						%\AX$A \fCenter X$
						%\AX$Y \fCenter B$
						%\RightLabel{\fns$(\pdra_R)$}
						%\BI$X > Y \fCenter A \pdra B$
						%\DisplayProof
						%%\\
						%& & & \\
						%\AX$X \fCenter A$
						%\AX$B \fCenter Y$
						%\LeftLabel{\fns$(\leftarrow_L)$}
						%\BI$A \leftarrow B \fCenter X < Y$
						%\DisplayProof
						%&
						%\AX$X \fCenter A < B$
						%\RightLabel{\fns$(\leftarrow_R)$}
						%\UI$X \fCenter A \leftarrow B$
						%\DisplayProof
						%&
						%\AX$A < B \fCenter Z$
						%\LeftLabel{\fns$(\pdla_L)$}
						%\UI$A \pdla B \fCenter Z$
						%\DisplayProof
						%&
						%\AX$A \fCenter X$
						%\AX$Y \fCenter B$
						%\RightLabel{\fns$(\pdla_R)$}
						%\BI$X < Y \fCenter A \pdla B$
						%\DisplayProof
					\end{tabular}
				\end{center}

				\item Introduction rules for pure $\mathsf{Res}$-connectives:

				\begin{center}
					\begin{tabular}{rlrl}
						\AXC{\phantom{$\gbot \fCenter \textrm{I}$}}
						\LeftLabel{\fns$0_L$}
						\UI$0 \fCenter \Phi$
						\DisplayProof
						&
						\AX$\Gamma \fCenter \Phi$
						\RightLabel{\fns$0_R$}
						\UI$\Gamma \fCenter 0$
						\DisplayProof
						&
						\AX$\Phi \fCenter \Gamma$
						\LeftLabel{\fns$1_L$}
						\UI$1 \fCenter \Gamma$
						\DisplayProof
						&
						\AXC{\phantom{$\textrm{I} \fCenter \top$}}
						\RightLabel{\fns$1_R$}
						\UI$\Phi \fCenter 1$
						\DisplayProof
						\\
						& & & \\
						\AX$\alpha \odot \beta \fCenter \Gamma$
						\LeftLabel{\fns$\cdot_L$}
						\UI$\alpha \cdot\beta \fCenter \Gamma$
						\DisplayProof
						&
						\AX$\Gamma \fCenter \alpha$
						\AX$\Delta \fCenter \beta$
						\RightLabel{\fns$\cdot_R$}
						\BI$\Gamma \odot \Delta \fCenter \alpha \cdot \beta$
						\DisplayProof
						&
						\AX$\alpha \fCenter \Gamma$
						\AX$\beta \fCenter \Delta$
						\LeftLabel{\fns$\ror_L$}
						\BI$\alpha \ror \beta \fCenter \Gamma \RAND \Delta$
						\DisplayProof
						&
						\AX$\Gamma \fCenter \alpha \RAND \beta$
						\RightLabel{\fns$\ror_R$}
						\UI$\Gamma \fCenter \alpha \ror \beta$
						\DisplayProof
						\\
					\end{tabular}
				\end{center}

				\item Introduction rules for the modal operators: %$\dc$, $\dr$, $\bc$ and $\br$:
				\begin{center}
					\begin{tabular}{cccc}
						\bottomAlignProof
						\AX$\DC A \fCenter X$
						\LeftLabel{$\dc_{L}$}
						\UI$\dc A \fCenter X$
						\DisplayProof
						&
						\bottomAlignProof
						\AX$X \fCenter A$
						\RightLabel{$\dc_{R}$}
						\UI$\DC X \fCenter \dc A$
						\DisplayProof
						&
						\bottomAlignProof
						\AX$\Gamma \fCenter \alpha$
						\RightLabel{$\bc_{L}$}
						\AX$A \fCenter X$
						\BI$\alpha \bc A \fCenter \Gamma \BC X$
						\DisplayProof
						&
						\bottomAlignProof
						\AX$X \fCenter \alpha \BC A$
						\RightLabel{$\bc_{R}$}
						\UI$X \fCenter \alpha \bc A$
						\DisplayProof
						\\ & & & \\
						\bottomAlignProof
						\AX$\DR \alpha \fCenter X$
						\LeftLabel{$\dr_{L}$}
						\UI$\dr \alpha \fCenter X$
						\DisplayProof
						&
						\bottomAlignProof
						\AX$\Gamma \fCenter \alpha$
						\RightLabel{$\dr_{R}$}
						\UI$\DR \Gamma \fCenter \dr \alpha$
						\DisplayProof
						&
						\bottomAlignProof
						\AX$\Gamma \fCenter \alpha$
						\RightLabel{$\br_{L}$}
						\AX$\beta \fCenter \Delta$
						\BI$\alpha \br \alpha \fCenter \Gamma \BR \Delta$
						\DisplayProof
						&
						\bottomAlignProof
						\AX$\Gamma \fCenter \alpha \BR \alpha$
						\RightLabel{$\br_{R}$}
						\UI$\Gamma \fCenter \alpha \br \alpha$
						\DisplayProof
					\end{tabular}
				\end{center}

			\end{enumerate}

We conclude the present section by listing some observations about D.LRC. Firstly, notice that, although very similar in spirit to a display calculus \cite{Belnap,Wa98}, D.LRC does not enjoy the display property, the reason being that a display rule for displaying substructures in the scope of the second coordinate of $\BC$ occurring in consequent position would not be sound. This is the reason why a more general form of cut rule, sometimes referred to as {\em surgical cut}, has been included than the standard one in display calculi where both cut formulas occur in display. However, as discussed in \cite{Trends}, calculi without display property can still verify the assumptions of some Belnap-style cut elimination metatheorem. In Section \ref{ssec:cut elim subformula}, we will verify that this is the case of D.LRC. Secondly, as usual, the version of D.LRC on a classical propositional base can be obtained by adding e.g.\ the following {\em Grishin rules}:

\begin{center}
\begin{tabular}{cc}
\AX$X > (Y \,; Z) \fCenter W$
\UI$(X > Y) \,; Z \fCenter W$
\DisplayProof
 &
\AX$X \fCenter Y > (Z \,; W)$
\UI$X \fCenter (Y > Z) \,; W$
\DisplayProof
\end{tabular}
\end{center}

%\marginnote{Add the following observation (maybe as a footnote?): Note that
Thirdly, the rule $W_\Phi$ encodes (and is used to derive)  $\alpha\cdot\beta\vdash \alpha$, $\alpha\cdot\beta\vdash \beta$, $\alpha\vdash 1$, B2 and B5.   %$\cdot$ has Weakening, 2. the neutral element of $\sqcap$ = the neutral element of $\cdot$ = top.}
%\marginnote{If we have Weakening on resources of $\ror$ ($W_R$), then we do not need the rules B2 and B5, hence we have removed them.}

%%%

\section{Basic properties of D.LRC}
In the present  section, we verify that the calculus D.LRC is sound w.r.t.\ the semantics of perfect heterogeneous LRC-algebras (cf.\ Definition \ref{def:LRC-algebraic structure}), is syntactically  complete w.r.t.\ the Hilbert calculus for LRC introduced in Section \ref{ssec:HLRC}, enjoys cut-elimination and subformula property, and conservatively extends the Hilbert calculus of Section \ref{ssec:HLRC}. %Our proof is indirect and relies on the fact that the Hilbert calculus H.LRC is complete w.r.t.\ the same semantics.
\subsection{Soundness}
\label{ssec:soundness}

In the present subsection, we outline the  verification of the soundness of the rules of D.LRC w.r.t.\ the semantics of perfect heterogeneous LRC-algebras (cf.\ Definition \ref{def:LRC-algebraic structure}). The first step consists in interpreting structural symbols as logical symbols according to their (precedent or consequent) position,\footnote{For any (formula or resource) sequent $x\vdash y$ in the language of D.LRC, we define the signed generation trees $+x$ and $-y$ by labelling the root of the generation tree of $x$ (resp.\ $y$) with the sign $+$ (resp.\ $-$), and then propagating the sign to all nodes according to the polarity of the coordinate of the connective assigned to each node. Positive (resp.\ negative) coordinates propagate the same (resp.\ opposite) sign to the corresponding child node. The only negative coordinates are the first coordinates of $>$, $\BC$ and $\BR$. Then, a substructure $z$ in $x\vdash y$ is in {\em precedent} (resp.\ {\em consequent}) {\em position} if the sign of its root node as a subtree of $+x$ or $-y$ is  $+$ (resp.\ $-$).}
as indicated in the synoptic tables at the beginning of Section \ref{sec: Display-style sequent calculus D.LRC}. This makes it possible to interpret sequents as inequalities, and rules as quasi-inequalities. For example, the rules on the left-hand side below are interpreted as the quasi-inequalities on the right-hand side:
\begin{center}
\begin{tabular}{rcl}
%\AX$X \fCenter \Gamma \BC (Y \,; \ADC (\DR \Gamma > Z))$
%\RightLabel{BD1}
%\UI$X \fCenter (\Gamma \BC Y) \,; Z$
%\DisplayProof
%&$\quad\rightsquigarrow\quad$&
%$\forall x\forall y\forall z\forall \gamma[x\leq  \gamma \bc (y \lor \adc (\dr \gamma \pra z)) \Rightarrow %x\leq (\gamma \bc y) \lor z]$
%\\
%&&\\
\AX$X \fCenter \Gamma \BC \ADC Y$
\RightLabel{BD1}
\UI$X \fCenter \DR \Gamma > Y$
\DisplayProof
&$\quad\rightsquigarrow\quad$& $\forall \gamma\forall x\forall y[x\leq \gamma \bc \adc \, y\Rightarrow  x \leq \dr\gamma \pra y ]$
\\
&&\\
\AX$X \fCenter \Gamma \BR \, \ADR Y$
\RightLabel{BD2}
\UI$X \fCenter \Gamma \BC Y$
\DisplayProof
&$\quad\rightsquigarrow\quad$& $\forall x\forall \gamma\forall y[x\leq \gamma \br \adr y\Rightarrow x \leq \gamma \bc y].$\\
\end{tabular}
\end{center}

The verification that the rules of D.LRC are sound on perfect LRC-algebras then consists in verifying the validity of their corresponding quasi-inequalities in perfect LRC-algebras. The validity of these quasi-inequalities follows straightforwardly from two observations. The first observation is that the quasi-inequality corresponding to each rule is obtained by running the algorithm ALBA on the axiom  of the Hilbert-style presentation of Section \ref{ssec:HLRC} bearing  the same name as the rule. Below we perform the ALBA reduction on  the axiom BD1:

\begin{center}
	\begin{tabular}{c l}
		& $\forall \alpha\forall p[\dr \alpha \land \alpha \bc p \leq \dc p ]$\\
		iff & $\forall \alpha\forall p\forall \gamma\forall x\forall y[(\gamma\leq \alpha\ \&\ x\leq \alpha \bc p\ \&\ \dc p\leq y)\Rightarrow \dr \gamma \land x \leq y ]$\\
		iff & $\forall \alpha\forall p\forall \gamma\forall x\forall y[(\gamma\leq \alpha\ \&\ x\leq \alpha \bc p\ \&\  p\leq \adc \, y)\Rightarrow \dr \gamma \land x \leq y ]$\\
		iff & $\forall \gamma\forall x\forall y[x\leq \gamma \bc \adc \, y\Rightarrow \dr \gamma \land x \leq y ]$\\
		iff & $\forall \gamma\forall x\forall y[x\leq \gamma \bc \adc \, y\Rightarrow  x \leq \dr\gamma \pra y ].$\\
	\end{tabular}
\end{center}

%\begin{center}
%\begin{tabular}{c l}
%& $\forall \alpha\forall p\forall q[\alpha \bc (p \lor q) \leq \alpha \bc p \lor (\dr \alpha \land \dc q)]$\\
%iff & $\forall \alpha\forall p\forall q\forall x\forall y\forall z\forall \gamma[(x\leq \alpha \bc (p \lor q)\ \&\ p\leq y\ \&\ \dr \alpha \land \dc q\leq z\ \&\ \gamma\leq \alpha)  \Rightarrow x\leq \gamma \bc y \lor z]$\\
%iff & $\forall \alpha\forall p\forall q\forall x\forall y\forall z\forall \gamma[(x\leq \alpha \bc (p \lor q)\ \&\ p\leq y\ \&\    \dc q\leq \dr \alpha \pra z\ \&\ \gamma\leq \alpha)  \Rightarrow x\leq \gamma \bc y \lor z]$\\
%iff & $\forall \alpha\forall p\forall q\forall x\forall y\forall z\forall \gamma[(x\leq \alpha \bc (p \lor q)\ \&\ p\leq y\ \&\    q\leq \adc (\dr \alpha \pra z)\ \&\ \gamma\leq \alpha)  \Rightarrow x\leq \gamma \bc y \lor z]$\\
%iff & $\forall x\forall y\forall z\forall \gamma[x\leq \gamma \bc (y \lor \adc (\dr \alpha \pra z))  \Rightarrow x\leq \gamma \bc y \lor z].$\\
%\end{tabular}
%\end{center}
It can be readily checked that the ALBA manipulation rules  applied in the computation above (adjunction rules and Ackermann rules) are sound on perfect LRC-algebras. As discussed in \cite{GMPTZ}, the soundness of these rules only depends on the order-theoretic properties of the interpretation of the logical connectives and their adjoints and residuals. The fact that some of these maps are not internal operations but have different domains and codomains does not make any substantial difference. A more substantial difference with the setting of \cite{GMPTZ} might be  in principle the fact that the connective $\bc$ is only monotone---rather than normal---in its second coordinate. However, notice that each manipulation in the chain of equivalences above involving that coordinate is an application of the Ackermann rule of ALBA, which relies on no more than monotonicity. The second observation is that the axioms of the Hilbert-style presentation of Section \ref{ssec:HLRC} are valid by definition on perfect LRC-algebras. We conclude the present subsection reporting the ALBA-reduction of (the condition expressing the validity of) axiom  BD2.

%\begin{center}
%\begin{tabular}{c l}
%& $\forall \alpha\forall p[\dr \alpha \land \alpha \bc p \leq \dc p ]$\\
%iff & $\forall \alpha\forall p\forall \gamma\forall x\forall y[(\gamma\leq \alpha\ \&\ x\leq \alpha \bc p\ %\&\ \dc p\leq y)\Rightarrow \dr \gamma \land x \leq y ]$\\
%iff & $\forall \alpha\forall p\forall \gamma\forall x\forall y[(\gamma\leq \alpha\ \&\ x\leq \alpha \bc p\ \&\  p\leq \adc y)\Rightarrow \dr \gamma \land x \leq y ]$\\
%iff & $\forall \gamma\forall x\forall y[x\leq \gamma \bc \adc y\Rightarrow \dr \gamma \land x \leq y ]$\\
%iff & $\forall \gamma\forall x\forall y[x\leq \gamma \bc \adc y\Rightarrow  x \leq \dr\gamma \pra y ].$\\
%\end{tabular}
%\end{center}

\begin{center}
\begin{tabular}{c l}
& $\forall \alpha\forall \beta[\alpha \br \beta \leq \alpha \bc \dr \beta]$\\
iff & $\forall \alpha\forall \beta\forall x\forall \gamma\forall y[(x\leq \alpha \br \beta\ \&\ \gamma\leq \alpha\ \&\ \dr \beta\leq y)\Rightarrow x \leq \gamma \bc y]$\\
iff & $\forall \alpha\forall \beta\forall x\forall \gamma\forall y[(x\leq \alpha \br \beta\ \&\ \gamma\leq \alpha\ \&\  \beta\leq \adr y)\Rightarrow x \leq \gamma \bc y]$\\
iff & $\forall x\forall \gamma\forall y[x\leq \gamma \br \adr y\Rightarrow x \leq \gamma \bc y].$\\
\end{tabular}
\end{center}
\subsection{Completeness}
\label{ssec:completeness}
In the present subsection, we show that the axioms of the Hilbert-style calculus H.LRC  introduced in Section \ref{ssec:HLRC} are derivable sequents of D.LRC, and  that the rules of H.LRC   are derivable rules of D.LRC. Since H.LRC  is complete w.r.t.\ the semantics of perfect heterogeneous LRC-algebras (cf.\ Definition \ref{def:LRC-algebraic structure}), we obtain as a corollary that D.LRC is also complete w.r.t.\ the semantics of perfect heterogeneous LRC-algebras. The derivations of the axioms R1-R3 of H.LRC are standard and we omit them.

\begin{itemize}
\item[R4.] $\alpha \cdot (\beta \ror \gamma) \leftrightarrow (\alpha \cdot \beta) \ror (\alpha \cdot \gamma)$
\end{itemize}

\begin{center}
\begin{tabular}{cc}
\AX$\alpha \fCenter \alpha$
\AX$\beta \fCenter \beta$
\BI$\alpha \odot \beta \fCenter \alpha \cdot \beta$
\UI$\beta \fCenter \alpha \RCDOT \alpha \cdot \beta$

\AX$\alpha \fCenter \alpha$
\AX$\gamma \fCenter \gamma$
\BI$\alpha \odot \gamma \fCenter \alpha \cdot \gamma$
\UI$\gamma \fCenter \alpha \RCDOT \alpha \cdot \gamma$

\BI$\beta \ror \gamma \fCenter (\alpha \RCDOT \alpha \cdot \beta) \RAND (\alpha \RCDOT \alpha \cdot \gamma)$
\RightLabel{\fns dis}
\UI$\beta \ror \gamma \fCenter \alpha \RCDOT (\alpha \cdot \beta \RAND \alpha \cdot \gamma)$
\UI$\alpha \odot (\beta \ror \gamma) \fCenter \alpha \cdot \beta \RAND \alpha \cdot \gamma$
\UI$\alpha \cdot (\beta \ror \gamma) \fCenter \alpha \cdot \beta \RAND \alpha \cdot \gamma$
\UI$\alpha \cdot (\beta \ror \gamma) \fCenter (\alpha \cdot \beta) \ror (\alpha \cdot \gamma)$
\DisplayProof

 &

\AX$\alpha \fCenter \alpha$
\AX$\beta \fCenter \beta$
\UI$\beta \fCenter \beta \RAND \gamma$
\UI$\beta \fCenter \beta \ror \gamma$
\BI$\alpha \odot \beta \fCenter \alpha \cdot (\beta \ror \gamma)$
\UI$\alpha \cdot \beta \fCenter \alpha \cdot (\beta \ror \gamma)$

\AX$\alpha \fCenter \alpha$
\AX$\gamma \fCenter \gamma$
\UI$\gamma \fCenter \beta \RAND \gamma$
\UI$\gamma \fCenter \beta \ror \gamma$
\BI$\alpha \odot \gamma \fCenter \alpha \cdot (\beta \ror \gamma)$
\UI$\alpha \cdot \gamma \fCenter \alpha \cdot (\beta \ror \gamma)$

\BI$(\alpha \cdot \beta) \ror (\alpha \cdot \gamma) \fCenter \alpha \cdot (\beta \ror \gamma) \RAND \alpha \cdot (\beta \ror \gamma)$
\UI$(\alpha \cdot \beta) \ror (\alpha \cdot \gamma) \fCenter \alpha \cdot (\beta \ror \gamma)$
\DisplayProof
 \\
\end{tabular}
\end{center}

The proof of $(\beta \ror \gamma) \cdot \alpha \leftrightarrow (\beta \cdot \alpha) \ror (\gamma \cdot \alpha)$ is analogous and we omit it.

\begin{itemize}
\item[D1.] $\dc (A \lor B) \plra \dc A \lor \dc B$
\end{itemize}

\begin{center}
\begin{tabular}{cc}
\AX$A \fCenter A$
\UI$\DC A \fCenter \dc A$
\UI$A \fCenter \ADC \dc A$
\AX$B \fCenter B$
\UI$\DC B \fCenter \dc B$
\UI$B \fCenter \ADC \dc B$
\BI$A \por B \fCenter \ADC \dc A \,; \ADC \dc B$
\RightLabel{\fns D1}
\UI$A \por B \fCenter \ADC (\dc A \,; \dc B)$
\UI$\DC A \por B \fCenter \dc A \,; \dc B$
\UI$\dc (A \por B) \fCenter \dc A \,; \dc B$
\UI$\dc (A \por B) \fCenter \dc A \por \dc B$
\DisplayProof

 &

\AX$A \fCenter A$
\UI$A \fCenter A \,; B$
\UI$A \fCenter A \por B$
\UI$\DC A \fCenter \dc (A \por B)$
\UI$\dc A \fCenter \dc (A \por B)$
\AX$B \fCenter B$
\UI$B \fCenter A \,; B$
\UI$B \fCenter A \por B$
\UI$\DC B \fCenter \dc (A \por B)$
\UI$\dc B \fCenter \dc (A \por B)$
\BI$\dc A \por \dc B \fCenter \dc (A \por B) \,; \dc (A \por B)$
\UI$\dc A \por \dc B \fCenter \dc (A \por B)$
\DisplayProof
 \\
\end{tabular}
\end{center}

\begin{itemize}
\item[D3.] $\dr (\alpha \ror \beta) \plra \dr \alpha \lor \dr \beta$
\end{itemize}

\begin{center}
\begin{tabular}{cc}
\AX$\alpha \fCenter \alpha$
\UI$\DR \alpha \fCenter \dr \alpha$
\UI$\alpha \fCenter \ADR \dr \alpha$
\AX$\beta \fCenter \beta$
\UI$\DR \beta \fCenter \dr \beta$
\UI$\beta \fCenter \ADR \dr \beta$
\BI$\alpha \ror \beta \fCenter \ADR \dr \alpha \RAND \ADR \dr \beta$
\RightLabel{\fns D3}
\UI$\alpha \ror \beta \fCenter \ADR (\dr \alpha \,; \dr \beta)$
\UI$\DR \alpha \ror \beta \fCenter \dr \alpha \,; \dr \beta$
\UI$\dr (\alpha \ror \beta) \fCenter \dr \alpha \,; \dr \beta$
\UI$\dr (\alpha \ror \beta) \fCenter \dr \alpha \por \dr \beta$
\DisplayProof

 &

\AX$\alpha \fCenter \alpha$
\UI$\alpha \fCenter \alpha \RAND \beta$
\UI$\alpha \fCenter \alpha \ror \beta$
\UI$\DR \alpha \fCenter \dr (\alpha \ror \beta)$
\UI$\dr \alpha \fCenter \dr (\alpha \ror \beta)$
\AX$\beta \fCenter \beta$
\UI$\beta \fCenter \alpha \RAND \beta$
\UI$\beta \fCenter \alpha \ror \beta$
\UI$\DR \beta \fCenter \dr (\alpha \ror \beta)$
\UI$\dr \beta \fCenter \dr (\alpha \ror \beta)$
\BI$\dr \alpha \ror \dr \beta \fCenter \dr (\alpha \ror \beta) \RAND \dr (\alpha \ror \beta)$
\UI$\dr \alpha \ror \dr \beta \fCenter \dr (\alpha \ror \beta)$
\DisplayProof
 \\
\end{tabular}
\end{center}

\begin{itemize}
\item[D2.] $\dc \bot \plra \bot$
\end{itemize}

\begin{center}
\begin{tabular}{cc}
\AX$\bot \fCenter \I$
\RightLabel{\fns D2}
\UI$\bot \fCenter \ADC \I$
\UI$\DC \bot \fCenter \I$
\UI$\DC \bot \fCenter \bot$
\UI$\dc \bot \fCenter \bot$
\DisplayProof
 &
\AX$\bot \fCenter \I$
\UI$\bot \fCenter \dc \bot \,; \I$
\UI$\bot \fCenter \dc \bot$
\DisplayProof
 \\
\end{tabular}
\end{center}

\begin{itemize}
\item[D4.] $\dr 0 \plra \bot$
\end{itemize}

\begin{center}
\begin{tabular}{cc}
\AX$0 \fCenter \Phi$
\RightLabel{\fns D4}
\UI$0 \fCenter \ADR \I$
\UI$\DR 0 \fCenter \I$
\UI$\DR 0 \fCenter \bot$
\UI$\dr 0 \fCenter \bot$
\DisplayProof
 &
\AX$0 \fCenter \Phi$
\UI$0 \fCenter \dr 0 \RAND \Phi$
\UI$0 \fCenter \dr 0$
\DisplayProof
 \\
\end{tabular}
\end{center}

\begin{itemize}
\item[B1.] $\alpha\ror\beta \bc A \plra (\alpha \bc A) \land (\beta \bc A)$
\end{itemize}

\begin{center}
{\small
\begin{tabular}{c@{}c}
\AX$\alpha \fCenter \alpha$
\UI$\alpha \fCenter \alpha \RAND \beta$
\UI$\alpha \fCenter \alpha \ror \beta$
\AX$A \fCenter A$
\BI$\alpha \ror \beta \bc A \fCenter \alpha \BC A$
\UI$\alpha \ror \beta \bc A \fCenter \alpha \bc A$
\AX$\beta \fCenter \beta$
\UI$\beta \fCenter \alpha \RAND \beta$
\UI$\beta \fCenter \alpha \ror \beta$
\AX$A \fCenter A$
\BI$\alpha \ror \beta \bc A \fCenter \beta \BC A$
\UI$\alpha \ror \beta \bc A \fCenter \beta \bc A$
\BI$\alpha \ror \beta \bc A \,; \alpha \ror \beta \bc A \fCenter (\alpha \bc A) \wedge (\beta \bc A)$
\UI$\alpha \ror \beta \bc A \fCenter (\alpha \bc A) \wedge (\beta \bc A)$
\DisplayProof

 &

\AX$\alpha \fCenter \alpha$
\AX$A \fCenter A$
\BI$\alpha \bc A \fCenter \alpha \BC A$
\UI$\alpha \fCenter \alpha \bc A \ABC A$
\AX$\beta \fCenter \beta$
\AX$A \fCenter A$
\BI$\beta \bc A \fCenter \beta \BC A$
\UI$\beta \fCenter \beta \bc A \ABC A$
\BI$\alpha \ror \beta \fCenter  (\alpha \bc A \ABC A) \RAND  ( \beta \bc A \ABC A)$
\RightLabel{\fns B1}
\UI$\alpha \ror \beta \fCenter (\alpha \bc A \,; \beta \bc A) \ABC A$
\UI$\alpha \bc A \,; \beta \bc A \fCenter \alpha \ror \beta \BC A$
\UI$\alpha \bc A \,; \beta \bc A \fCenter \alpha \ror \beta \bc A$
\UI$(\alpha \bc A) \wedge (\beta \bc A) \fCenter \alpha \ror \beta \bc A$
\DisplayProof
 \\
\end{tabular}
}
\end{center}

\begin{itemize}
\item[B4.] $\alpha \ror \beta \br \gamma \plra (\alpha \br \gamma) \land (\beta \br \gamma)$
\end{itemize}

\begin{center}
{\small
\begin{tabular}{c@{}c}
\AX$\alpha \fCenter \alpha$
\UI$\alpha \fCenter \alpha \RAND \beta$
\UI$\alpha \fCenter \alpha \ror \beta$
\AX$\gamma \fCenter \gamma$
\BI$\alpha \ror \beta \br \gamma \fCenter \alpha \BR \gamma$
\UI$\alpha \ror \beta \br \gamma \fCenter \alpha \br \gamma$
\AX$\beta \fCenter \beta$
\UI$\beta \fCenter \alpha \RAND \beta$
\UI$\beta \fCenter \alpha \ror \beta$
\AX$\gamma \fCenter \gamma$
\BI$\alpha \ror \beta \br \gamma \fCenter \beta \BR \gamma$
\UI$\alpha \ror \beta \br \gamma \fCenter \beta \br \gamma$
\BI$\alpha \ror \beta \br \gamma \,; \alpha \ror \beta \br \gamma \fCenter (\alpha \br \gamma) \land (\beta \br \gamma)$
\UI$\alpha \ror \beta \br \gamma \fCenter (\alpha \br \gamma) \land (\beta \br \gamma)$
\DisplayProof

 &

\AX$\alpha \fCenter \alpha$
\AX$\gamma \fCenter \gamma$
\BI$\alpha \br \gamma \fCenter \alpha \BR \gamma$
\UI$\alpha \fCenter \alpha \br \gamma \ABR \gamma$
\AX$\beta \fCenter \beta$
\AX$\gamma \fCenter \gamma$
\BI$\beta \br \gamma \fCenter \beta \BR \gamma$
\UI$\beta \fCenter \beta \br \gamma \ABR \gamma$
\BI$\alpha \ror \beta \fCenter (\alpha \br \gamma \ABR \gamma) \RAND (\beta \br \gamma \ABR \gamma)$
\RightLabel{\fns B4}
\UI$\alpha \ror \beta \fCenter (\alpha \br \gamma \,; \beta \br \gamma) \ABR \gamma$
\UI$\alpha \br \gamma \,; \beta \br \gamma \fCenter \alpha \ror \beta \BR \gamma$
\UI$\alpha \br \gamma \,; \beta \br \gamma \fCenter \alpha \ror \beta \br \gamma$
\UI$(\alpha \br \gamma) \land (\beta \br \gamma) \fCenter \alpha \ror \beta \br \gamma$
\DisplayProof
 \\
\end{tabular}
}
\end{center}

\begin{itemize}
\item[B2.] $0 \bc A$
\end{itemize}

\begin{center}
\begin{tabular}{c}
\AX$0 \fCenter \Phi$
\UI$0 \fCenter \I \ABC A \RAND \Phi$
\UI$0 \fCenter \I \ABC A$
\UI$\I \fCenter 0 \BC A$
\UI$\I \fCenter 0 \bc A$
\DisplayProof
 \\
\end{tabular}
\end{center}

\begin{itemize}
\item[B5.] $0 \br \alpha$
\end{itemize}

\begin{center}
\begin{tabular}{c}
\AX$0 \fCenter \Phi$
\UI$0 \fCenter \I \ABR \alpha \RAND \Phi$
\UI$0 \fCenter \I \ABR \alpha$
\UI$\I \fCenter 0 \BR \alpha$
\UI$\I \fCenter 0 \br \alpha$
\DisplayProof

 \\
\end{tabular}
\end{center}

\begin{itemize}
\item[B3.] $\alpha \bc (\beta \bc A) \pra (\alpha \cdot \beta \bc A)$
\end{itemize}

\begin{center}
\begin{tabular}{c}
\AX$\alpha \fCenter \alpha$
\AX$\beta \fCenter \beta$
\AX$A \fCenter A$
\BI$\beta \bc A \fCenter \beta \BC A$
\BI$ \alpha \bc (\beta \bc A) \fCenter \alpha \BC (\beta \BC A)$
\RightLabel{\fns B3}
\UI$\alpha \bc (\beta \bc A) \fCenter (\alpha \odot \beta) \BC A$
\UI$\alpha \odot \beta \fCenter (\alpha \bc (\beta \bc A)) \ABC A$
\UI$\alpha \cdot \beta \fCenter (\alpha \bc (\beta \bc A)) \ABC A$
\UI$\alpha \bc (\beta \bc A) \fCenter (\alpha \cdot \beta) \BC A$
\UI$\alpha \bc (\beta \bc A) \fCenter (\alpha \cdot \beta) \bc A$
\DisplayProof
 \\
\end{tabular}
\end{center}

\begin{itemize}
\item[B6.] $\alpha \br (\beta \rand \gamma) \plra \alpha \br \beta \land \alpha \br \gamma$
\end{itemize}

\begin{center}
{\small
\begin{tabular}{cc}

\AX$\alpha \fCenter \alpha$
\AX$\beta \fCenter \beta$
\UI$\beta \RAND \gamma \fCenter \beta$
\UI$\beta \rand \gamma \fCenter \beta$
\BI$\alpha \br \beta \sqcap \gamma \fCenter \alpha \BR \beta$
\UI$\alpha \br \beta \sqcap \gamma \fCenter \alpha \br \beta$
\AX$\alpha \fCenter \alpha$
\AX$\gamma \fCenter \gamma$
\UI$\gamma \RAND \beta \fCenter \gamma$
\UI$\beta \RAND \gamma \fCenter \gamma$
\UI$\beta \rand \gamma \fCenter \gamma$
\BI$\alpha \br \beta \rand \gamma \fCenter \alpha \BR \gamma$
\UI$\alpha \br \beta \rand \gamma \fCenter \alpha \br \gamma$
\BI$\alpha \br \beta \rand \gamma \,; \alpha \br \beta \rand \gamma \fCenter (\alpha \br \beta) \pand (\alpha \br \gamma)$
\UI$\alpha \br \beta \rand \gamma \fCenter (\alpha \br \beta) \pand (\alpha \br \gamma)$
\DisplayProof

 &

\AX$\alpha \fCenter \alpha$
\AX$\beta \fCenter \beta$
\BI$\alpha \br \beta \fCenter \alpha \BR \beta$
\UI$\alpha \LABR \alpha \br \beta \fCenter \beta$

\AX$\alpha \fCenter \alpha$
\AX$\gamma \fCenter \gamma$
\BI$\alpha \br \gamma \fCenter \alpha \BR \gamma$
\UI$\alpha \LABR \alpha \br \gamma \fCenter \gamma$

\BI$(\alpha \LABR \alpha \br \beta) \RAND (\alpha \LABR \alpha \br \gamma) \fCenter \beta \rand \gamma$
\LeftLabel{\fns B6}
\UI$\alpha \LABR (\alpha \br \beta \,; \alpha \br \gamma) \fCenter \beta \rand \gamma$
\UI$\alpha \br \beta \,;  \alpha \br \gamma \fCenter \alpha \BR \beta \rand \gamma$
\UI$\alpha \br \beta \,; \alpha \br \gamma \fCenter \alpha \br \beta \rand \gamma$
\UI$(\alpha \br \beta) \pand (\alpha \br \gamma) \fCenter \alpha \br \beta \rand \gamma$
\DisplayProof
 \\
\end{tabular}
}
\end{center}

\begin{itemize}
\item[B7.] $\alpha \br 1$
\end{itemize}

\begin{center}
\begin{tabular}{c}
\AX$\Phi \fCenter 1$
\UI$\alpha \LABR \I \RAND \Phi \fCenter 1$
\UI$\alpha \LABR \I \fCenter 1$
\UI$\I \fCenter \alpha \BR 1$
\UI$\I \fCenter \alpha \br 1$
\DisplayProof
\end{tabular}
\end{center}

\begin{comment}
\begin{itemize}
\item[B8.] $\alpha \br \alpha$
\end{itemize}

\begin{center}
\begin{tabular}{c}
\AX$\alpha \fCenter \alpha$
\LeftLabel{\fns B8}
\UI$\alpha \LABR \I \fCenter \alpha$
\UI$\I \fCenter \alpha \BR \alpha$
\UI$\I \fCenter \alpha \br \alpha$
\DisplayProof
\end{tabular}
\end{center}
\end{comment}

%\begin{itemize}
%\item[BD1.] $\alpha \bc (A \lor B) \pra (\alpha \bc A) \lor (\dr \alpha \land \dc B)$
%\end{itemize}

%\begin{center}
%\begin{tabular}{c}
%\AX$\alpha \fCenter \alpha$
%\AX$A \fCenter A$

%\AX$\alpha \fCenter \alpha$
%\UI$\DR \alpha \fCenter \dr \alpha$

%\AX$B \fCenter B$
%\UI$\DC B \fCenter \dc B$

%\BI$\DR \alpha \RAND \DC B \fCenter \dr \alpha \land \dc B$
%\UI$\DC B \fCenter \DR \alpha > \dr \alpha \land \dc B$
%\UI$B \fCenter \ADC (\DR \alpha > \dr \alpha \land \dc B)$

%\BI$A \lor B \fCenter A \,; \ADR (\DR \alpha > \dr \alpha \land \dc B)$

%\BI$\alpha \bc (A \lor B) \fCenter \alpha \BC (A \,; \ADR (\DR \alpha > \dr \alpha \land \dc B))$
%\RightLabel{\fns BD1}
%\UI$\alpha \bc (A \lor B) \fCenter \alpha \BC A \,; (\dr \alpha \land \dc B)$
%\UI$\alpha \bc (A \lor B) \fCenter \alpha \bc A \,; (\dr \alpha \land \dc B)$
%\UI$\alpha \bc (A \lor B) \fCenter (\alpha \bc A) \lor (\dr \alpha \land \dc B)$
%\DisplayProof
%\end{tabular}
%\end{center}
\bigskip

\begin{itemize}
\item[BD1.] $\dr \alpha \land \alpha \bc A \pra \dc A$
\end{itemize}

\begin{center}
\begin{tabular}{c}
\AX$\alpha \fCenter \alpha$
\AX$A \fCenter A$
\UI$\DC A \fCenter \dc A$
\UI$A \fCenter \ADC \dc A$
\BI$\alpha \bc A \fCenter \alpha \BC \ADC \dc A$
\RightLabel{\fns BD1}
\UI$\alpha \bc A \fCenter \DR \alpha > \dc A$
\UI$\DR \alpha \,; \alpha \bc A \fCenter \dc A$
\UI$\DR \alpha \fCenter \dc A < \alpha \bc A$
\UI$\dr \alpha \fCenter \dc A < \alpha \bc A$
\UI$\dr \alpha \,; \alpha \bc A \fCenter \dc A$
\UI$\dr \alpha \land \alpha \bc A \fCenter \dc A$
\DisplayProof
\end{tabular}
\end{center}

\begin{itemize}
\item[BD2.] $\alpha \br \beta \pra \alpha \bc \dr \beta$
\end{itemize}

\begin{center}
\begin{tabular}{c}
\AX$\alpha \fCenter \alpha$
\AX$\beta \fCenter \beta$
\UI$\DR \beta \fCenter \dr \beta$
\UI$\beta \fCenter \ADR \dr \beta$
\BI$\alpha \br \beta \fCenter \alpha \BR \, \ADR \dr\beta$
\RightLabel{\fns BD2}
\UI$\alpha \br \beta \fCenter \alpha \BC \dr\beta$
\UI$\alpha \br \beta \fCenter \alpha \bc \dr\beta$
\DisplayProof
\end{tabular}
\end{center}

The rules of H.LRC immediately follow from applications of the introduction rules of the corresponding logical connectives in the usual way and we omit their derivations.

%%%

\subsection{Cut-elimination and subformula property}
\label{ssec:cut elim subformula}

In the present subsection, we sketch the verification that the D.LRC is a proper multi-type calculus (cf.\ Section \ref{sec:properly semi displayable}). By Theorem \ref{Thm cut elimination}, this is enough to establish that the calculus enjoys  cut elimination and subformula property.
With the exception of C$_8'$, all conditions  are straightforwardly verified by inspecting the rules, and this verification is left to the reader. %Only the condition C$'^*_5$ was slightly modified in order to allow rules whose principal formulas is not in display. In D.LRC the only such non-standard rule is D4 and it satisfies the condition, because the cut rules of D.LRC do not admit resource-terms in precedent position as cut-terms.

As to the verification of condition C$_8'$, the only interesting case is the one in which the cut formula is of the form $\alpha \bc A$, since the connective $\bc$ is monotone rather than normal in its second coordinate, which is the reason why not even a weak form of display property holds for D.LRC. This case is treated  below. Notice that, since all principal formulas are in display, no surgical cuts need to be eliminated in the principal stage.

\begin{center}
\begin{tabular}{ccc}
\bottomAlignProof
\AXC{\ \ \ $\vdots$ \raisebox{1mm}{$\pi_1$}}
\noLine
\UI$X \fCenter \alpha \BC A$
\UI$X \fCenter \alpha \bc A$
\AXC{\ \ \ $\vdots$ \raisebox{1mm}{$\pi_2$}}
\noLine
\UI$\Gamma \fCenter \alpha$
\AXC{\ \ \ $\vdots$ \raisebox{1mm}{$\pi_3$}}
\noLine
\UI$A \fCenter Y$
\BI$\alpha \bc A \fCenter \Gamma \BC Y$
\BI$X \fCenter \Gamma \BC Y$
\DisplayProof

 &
%\rotatebox[origin=c]{-90}{$\rightsquigarrow$}
$\rightsquigarrow$
 &

\bottomAlignProof
\AXC{\ \ \ $\vdots$ \raisebox{1mm}{$\pi_2$}}
\noLine
\UI$\Gamma \fCenter \alpha$

\AXC{\ \ \ $\vdots$ \raisebox{1mm}{$\pi_1$}}
\noLine
\UI$X \fCenter \alpha \BC A$
\UI$\alpha \fCenter X \BC A$

\BI$\Gamma \fCenter X \BC A$
\UI$X \fCenter \Gamma \BC A$

\AXC{\ \ \ $\vdots$ \raisebox{1mm}{$\pi_3$}}
\noLine
\UI$A \fCenter Y$

\BI$X \fCenter \Gamma \BC Y$

\DisplayProof
 \\
\end{tabular}
\end{center}

{}

\subsection{Semantic conservativity}
\label{ssec:conservativity}
To argue that the calculus D.LRC adequately captures LRC, we follow the standard proof strategy discussed in \cite{GMPTZ}. Recall that $\vdash_{\mathrm{LRC}}$ denotes the syntactic consequence relation arising from the Hilbert system for LRC introduced in Section \ref{ssec:HLRC}. We need to show that, for all LRC-formulas $A$ and $B$, if $A\vdash B$ is a provable sequent in the calculus D.LRC, then  $A\vdash_{\mathrm{LRC}} B$. This fact can be verified using  the following standard argument and facts: (a) the rules of D.LRC are sound w.r.t.\  perfect heterogeneous LRC-algebras  (cf.\ Section \ref{ssec:soundness}), and (b) LRC is strongly complete w.r.t.\ perfect heterogeneous LRC-algebras (cf.\ Corollary \ref{cor:strong completeness perfect algebras}). Then, let $A, B$ be LRC-formulas such that $A\vdash B$ is a derivable sequent in D.LRC. By (a), this implies that $A\models_{\mathrm{LRC}} B$, which implies, by (b), that $A\vdash_{\mathrm{LRC}} B$, as required.

%%%
%\marginnote{add just a comment about the adjoints as usual.}

\section{Case studies}
\label{sec:casestudies}
In this section, we present a number of case studies, with the purpose of highlighting various aspects of the basic framework and also various ways in which it can be adapted to different settings. % discuss some variations and applications of the basic framework of D.LRC by means of  examples which illustrate different facets and potentials of the framework.
The most common adaptations performed in the case studies below consist in adding analytic structural rules   to the basic calculus. Interestingly, the resulting calculi still enjoy the same package of basic properties (soundness, completeness, cut-elimination, subformula property, conservativity) which hold of D.LRC as an immediate consequence of general results. Indeed, it can be readily verified that the axioms corresponding to each of the rules introduced below are analytic inductive (cf.~\cite[Definition 55]{GMPTZ}), and hence are canonical (cf.~\cite[Theorem 19]{GMPTZ}).
Therefore, the axiomatic extensions of LRC corresponding to these axioms is {\em sound } and {\em complete} w.r.t.~the corresponding subclass of  LRC-models. %Using the additional structural rule, the axiom in question can be derived in D.LRC with
{\em Conservativity} can be argued by repeating verbatim the same argument given in Section 4.4 which uses the soundness of the augmented calculus w.r.t.~the corresponding class of perfect  LRC-models, and the completeness of the Hilbert-style presentation of the axiomatic extension which holds because the additional axioms are canonical. Finally, {\em cut-elimination} and {\em subformula property} follow from the general cut-elimination metatheorem.

In what follows, we will sometimes abuse terminology and speak of a formula $A$ being derived from certain assumptions $A_1;\ldots; A_n$ meaning that the sequent $A_1;\ldots; A_n\vdash A$ is derivable in the calculus.
\subsection{Pooling capabilities (correcting a homework assignment)}
\label{ssec:homework}

Two teaching assistants, Carl (\texttt{c}) and Dan (\texttt{d}), are assigned the task of grading a set of homework assignments consisting of two exercises, a model-theoretic one ($M$) and a proof-theoretic one ($P$). Carl is  only capable of correcting exercise $P$, while Dan is only capable of correcting exercise $M$. None of the two teaching assistants can individually complete the task they have been assigned. However, they can if they {\em pool} their capabilities. One way in which they can complete the task is by implementing the following plan: they split the set of homework assignments  into two sets $\alpha$ and $\beta$. Initially, Carl grades the solutions to exercise $P$ in $\alpha$ and Dan those of $M$ in $\beta$. Then they switch sets and each of them grades the solutions to the same exercise in the other set.

To capture this case study in (a multi-agent version of) D.LRC, we introduce atomic propositions such as $P_\alpha$ (resp.\  $M_\beta$), the intended meaning of which is that all solutions to exercise $P$ (resp.\ $M$) in $\alpha$ (resp.\ $\beta$) have been graded. We also treat $\alpha$ and $\beta$ as resources. The following table contains formulas expressing  the assumptions about  agents' capabilities, the initial state of affairs (which  resources are initially in possession of which agent), and the plan of switching after completing the correction of one exercise in a given set:
\begin{center}
\begin{tabular}{c c|c|c}
           \mc{2}{c|}{Capabilities} & initial state  & planning \\
           %\mc{2}{c|}{assumptions} & assumptions& assumptions \\
\hline

  $\alpha\bc_{\ac} P_\alpha$  & $\beta \bc_{\ac} P_\beta$ & $\dr_{\ac}\alpha  $  & $M_\beta \pra \dr_{\ac}\beta$ \\
$\alpha \bc_{\ad} M_\alpha$ & $\beta \bc_{\ad} M_\beta$ & $\dr_{\ad}\beta  $     &  $P_\alpha\pra \dr_{\ad}\alpha$ \\
\end{tabular}
\end{center}
%Thus, given the following assumptions: $H1 = \alpha \dc_\ac \top \,; \alpha \bc_\ac P_\alpha$, $H2 = \beta \dc_\ad \top \,; \beta \bc_\ad M_\beta$, $H3 = \beta \bc_\ac P_\beta \,; ( \beta \dc_\ad M_\beta \,; \beta \dc_\ad M_\beta \pra M_\beta) \,; M_\beta \pra \beta \dc_\ac \top$, $H4 = \alpha \bc_\ad M_\alpha \,; ( \alpha \dc_\ac P_\alpha \,; \alpha \dc_\ac P_\alpha \pra P_\alpha) \,; P_\alpha \pra \alpha \dc_\ad \top$, we can prove that they are indeed able to correct everything in the end: $(\alpha \dc_\ac P_\alpha \pand \beta \dc_\ad M_\beta) \pand (\beta \dc_\ac P_\beta \pand \alpha \dc_\ad M_\alpha)$.
In the present setting we also assume that, whenever an agent is {\em able} to bring about a certain state of affairs, the agent will. Formally, this corresponds to the validity of the axioms $\dc_{\texttt{i}} A\pra A$ for every agent \texttt{i} and formula $A$. This axiom does not follow from the logic H.LRC, and in many  settings it would not be sound. However, for the sake of the present case study, we will assume that this axiom holds. In fact, this axiom  corresponds to the following  rules `Ex$_{\texttt{i}}$' (`Ex' stands for Execution), for each \texttt{i}$\in \{\texttt{c}, \texttt{d}\}$:
\[\AXC{$X\fCenter Y$}
\LeftLabel{Ex$_{\texttt{i}}$}
\UI$\DC_{\texttt{i}} X\fCenter Y$
\DisplayProof\]
Notice that these rules are analytic (cf.\  Section \ref{sec:properly semi displayable}). %do not satisfy condition C$_5$ in its restricted form, but  do satisfy C$^{'*}_5$.
Hence, by Theorem \ref{Thm cut elimination}, when adding these rules  to the basic calculus D.LRC,  the resulting calculus (which we refer to as D.LRC + Ex) enjoys cut elimination and subformula property.

We aim at  deriving the formula $(P_\alpha\pand M_\beta)\pand (P_\beta\pand  M_\alpha)$ from the assumptions above in the calculus D.LRC + Ex. This will provide the formal verification that executing the plan yields the completion of the task. Let us start by considering the following derivations:

\begin{center}
\begin{tabular}{c@{}c@{}c}

$\pi_1$ &$\quad$& $\pi_2$\\

&&\\

{\footnotesize
\AXC{\ \ \ \ \ \ \ \ \ \ \ \ \ $\vdots$ \raisebox{1mm}{\tiny{\begin{tabular}{@{}l}proof for\\BD1\\ \end{tabular} }}}
\noLine
\UI$\dr_{\ac} \alpha  \,; \alpha \bc_{\ac} P_{\alpha} \fCenter \dc_{\ac} P_{\alpha}$
\AXC{$P_{\alpha} \fCenter P_{\alpha}$}
\LeftLabel{Ex}
\UI$\DC_{\ac} P_{\alpha} \fCenter P_{\alpha}$
\UI$\dc_{\ac} P_{\alpha} \fCenter P_{\alpha}$
\LeftLabel{Cut}
\BI$\dr_{\ac} \alpha \,; \alpha \bc_{\ac} P_{\alpha} \fCenter P_{\alpha}$
\DisplayProof
}
&&
{\footnotesize
\AXC{\ \ \ \ \ \ \ \ \ \ \ \ \ $\vdots$ \raisebox{1mm}{\tiny{\begin{tabular}{@{}l}proof for\\BD1\\ \end{tabular} }}}
\noLine
\UI$\dr_{\ad} \beta  \,; \alpha \bc_{\ad} M_\beta \fCenter  \dc_{\ad} M_\beta$
\AXC{$M_{\beta} \fCenter M_{\beta}$}
\LeftLabel{Ex}
\UI$\DC_{\ad} M_{\beta} \fCenter M_{\beta}$
\UI$\dc_{\ad} M_{\beta} \fCenter M_{\beta}$
\LeftLabel{Cut}
\BI$\dr_{\ad} \beta  \,; \beta \bc_{\ad} M_{\beta} \fCenter M_{\beta}$
\DisplayProof
}
\\
&&\\
$\pi_3$ &$\quad$& $\pi_4$\\

&&\\
{\footnotesize
\AXC{\ \ \ \ \ \ \ \ \ \ \ \ \ $\vdots$ \raisebox{1mm}{\tiny{\begin{tabular}{@{}l}proof for\\BD1\\ \end{tabular} }}}
\noLine
\UI$\dr_{\ac} \beta  \,; \beta \bc_{\ac} P_{\beta} \fCenter  \dc_{\ac} P_{\beta}$
\AXC{$P_{\beta} \fCenter P_{\beta}$}
\LeftLabel{Ex}
\UI$\DC_{\ac} P_{\beta} \fCenter P_{\beta}$
\UI$\dc_{\ac} P_{\beta} \fCenter P_{\beta}$
\LeftLabel{Cut}
\BI$\dr_{\ac} \beta \,; \beta \bc_{\ac} P_{\beta} \fCenter P_{\beta}$
\DisplayProof
}
&&

{\footnotesize
\AXC{\ \ \ \ \ \ \ \ \ \ \ \ \ $\vdots$ \raisebox{1mm}{\tiny{\begin{tabular}{@{}l}proof for\\BD1\\ \end{tabular} }}}
\noLine
\UI$\dr_{\ad} \alpha  \,; \alpha \bc_{\ad} M_\alpha \fCenter \dc_{\ad} M_\alpha$
\AXC{$M_\alpha\fCenter M_\alpha$}
\LeftLabel{Ex}
\UI$\DC_{\ad} M_{\alpha} \fCenter M_{\alpha}$
\UI$\dc_{\ad} M_{\alpha} \fCenter M_{\alpha}$
\LeftLabel{Cut}
\BI$\dr_{\ad} \alpha  \,; \alpha \bc_{\ad} M_{\alpha} \fCenter M_{\alpha}$
\DisplayProof
}
\\
\end{tabular}
\end{center}

These derivations follow one and the same pattern, and each derives one piece of the desired conclusion. Hence, one would want to suitably prolong these derivations by applying $\pand_R$ to reach the conclusion. However, while the conclusions of $\pi_1$ and $\pi_2$ contain only formulas which are assumptions in our case study as reported in the table above, the formulas $\dr_{\ac} \beta$ and $\dr_{\ad} \alpha$, occurring in the conclusions of $\pi_3$ and $\pi_4$ respectively, are not assumptions. However,  they are  provable from the assumptions. Indeed, they encode states of affairs which hold after \texttt{c} and \texttt{d} have  switched  the sets $\alpha$ and $\beta$.

Notice that the following sequents are provable (their derivations are straightforward and are omitted):
\[M_{\beta}; M_{\beta} \rightarrow \dr_{\ac}\beta \vdash \dr_{\ac}\beta  \quad P_{\alpha}; P_{\alpha} \rightarrow \dr_{\ad}\alpha \vdash \dr_{\ad}\alpha  \]
These sequents say that the formulas $\dr_{\ac}\beta$ and $\dr_{\ad}\alpha$ are provable from the `planning assumptions' (cf.\ table above) using the formulas $M_{\beta}$ and $P_{\alpha}$ which have been derived purely from the assumptions by $\pi_{1}$ and $\pi_{2}$.
Hence, the atoms $P_{\beta}$ and $M_{\alpha}$ can be derived from the original assumptions via cut. Then, applying $\pand_{R}$ and possibly contraction, one can derive the desired sequent.

\subsection{Conjoining capabilities (the wisdom of the crow)}
\label{ssec:crow}

A BBC documentary program shows a problem-solving test conducted on a crow. In the present subsection we formalize an adapted version of this test. There is food ($\phi$) positioned deep in a narrow box, out of the reach of the crow's beak. There is a short stick ($\sigma$) directly available to the crow, two stones ($\rho_1, \rho_2$) each inside a cage, and a long stick ($\lambda$)  inside a transparent box which  releases the stick if enough weight (that of two stones or more) lays inside the box. The stick $\sigma$ is too short for the crow to reach the food using it. However, previous tests have shown that  the crow is capable of performing the following individual steps: (a) reaching the food using the long stick; (b) retrieving the stones from the cages using the short stick; (c) retrieving the long stick by dropping  stones into a slot in the box. The crow succeeded in executing these individual steps in the right order and got to the food.

An interesting feature of this case study is the interplay of different kinds of resources. Specifically, $\sigma$ is a {\em reusable} resource (indeed, the crow uses the same stick to reach the two stones), which fact can be expressed by the sequent $\sigma\vdash \sigma\cdot\sigma$. Also, the following formula holds of all resources relevant to the present case study: $\alpha\br\gamma\wedge \beta\br\delta \rightarrow \alpha\cdot \beta \br \gamma\cdot\delta$. This formula implies a form of {\em scalability} of resources,\footnote{That is, if the agent is capable of getting  one (measure of) $\beta$ from one (measure of) $\alpha$, then is also capable to get two or $n$ (measures of) $\beta$ from two or $n$ (measures of) $\alpha$.} which is not a property holding in general, and hence has not been added to the general calculus. The crow passing the test shows to be able to conjoin the separate capabilities together.  This is expressed by the following {\em transitivity}-type axiom: $\alpha\br\beta\land\beta\br \gamma\rightarrow \alpha\br\gamma$. The crow's achievement is remarkable precisely because this axiom cannot be expected to hold of any agent.
These conditions translate into the following analytic rules:

\begin{center}
\begin{tabular}{c c c c c }
\AX$\Sigma\odot \Sigma\fCenter \Omega$
\LeftLabel{Contr}
\UI$\Sigma \fCenter \Omega$
\DisplayProof
						&&
\AX$(\Gamma\LABR X)\odot(\Pi\LABR Y) \fCenter \Delta$
\LeftLabel{Scalab}
\UI$(\Gamma\odot \Pi)\LABR(X\, ; Y)  \fCenter \Delta$
\DisplayProof
 &&
 \AX$(\Gamma\LABR X)\LABR Y \fCenter \Delta$
\LeftLabel{Trans}
\UI$\Gamma\LABR(X\, ; Y)  \fCenter \Delta$
\DisplayProof

\\
\end{tabular}
\end{center}
In order for the  rule Contr to satisfy C$_6$ and C$_9$, we need to work with a version of D.LRC which admits {\em two} types of resources: the {\em reusable} ones (for which the contraction rule is sound) and the general ones for which contraction is not sound. Hence, the contraction would be introduced only for the reusable type. Once the new type has been introduced, the language  and calculus of LRC need to be expanded with copies of each original connective, so as to account for the fact that each copy takes in input and outputs exactly one type unambiguously. Correspondingly, copies of each original rule have to be added so that each copy accounts for exactly one reading of the original rule. This is a tedious but entirely safe procedure that guarantees that a proper multi-type calculus (cf.\ Definition \ref{def:proper multi-type calculus}) can be introduced which admits {\em all} the rules above. The reader is referred to \cite{PDL,Multitype} for  examples of such a disambiguation procedure.

%\marginnote{notice that in some cases, reusability can be modelled as an intrinsic property of resources (and hence expressed purely at the level of resources with identities of type $\alpha  = \alpha\cot \alpha$), but there are also cases in which the reusability is a property of the resource in connection with a specific agent. This is for instance the case of the stick, which is reusable w.r.t.\ the crow but not w.r.t.\ another agent who e.g.\ sets it on fire. Accordingly, we encode the reusability of $\sigma$ with a rule which also involves $\dc$.}
The following table shows the assumptions of the present case study:

\begin{center}
\begin{tabular}{c|r}
          Initial state         & Capabilities            \\
          %assumption            & assumptions              \\
\hline
                  & $\sigma \br \rho \quad$    \\
  $\dr\sigma $    & $\rho \cdot \rho \br \lambda \quad$\\
                  &  $\lambda \br \varphi \quad$ \\

\end{tabular}
\end{center}

We aim at proving the following sequent:
\[\sigma\br\rho\,;\rho\cdot\rho\br \lambda\,;\lambda\br \phi\,;\dr\sigma\vdash \dc\dr\phi.\]

We do it in several steps: first, in the following derivation $\pi_1$, we prove that for any reusable resource $\sigma$, if $\sigma\br\rho$ then $\sigma\br\rho\cdot\rho$:
\begin{center}
\AX$\sigma\fCenter\sigma$
\AX$\rho\fCenter\rho$
\BI$\sigma\br\rho\fCenter \sigma\BR\rho$
\UI$\sigma\LABR\sigma\br\rho\fCenter \rho$
\AX$\sigma\fCenter\sigma$
\AX$\rho\fCenter\rho$
\BI$\sigma\br\rho\fCenter \sigma\BR\rho$
\UI$\sigma\LABR\sigma\br\rho\fCenter \rho$
\BI$(\sigma\LABR\sigma\br\rho)\odot(\sigma\LABR\sigma\br\rho)\fCenter\rho\cdot\rho$
\LeftLabel{Scalab}
\UI$(\sigma\odot\sigma)\LABR(\sigma\br\rho\,;\sigma\br\rho)\fCenter\rho\cdot\rho$
\UI$\sigma\br\rho\,;\sigma\br\rho\fCenter \sigma\odot\sigma\BR\rho\cdot\rho$
\UI$\sigma\odot\sigma\fCenter (\sigma\br\rho\,;\sigma\br\rho) \ABR\rho\cdot\rho$
\LeftLabel{Contr}
\UI$\sigma\fCenter (\sigma\br\rho\,;\sigma\br\rho) \ABR\rho\cdot\rho$
\UI$\sigma\br\rho\,;\sigma\br\rho\fCenter  \sigma\BR\rho\cdot\rho$
\UI$\sigma\br\rho\fCenter  \sigma\BR\rho\cdot\rho$
\UI$\sigma\br\rho\fCenter  \sigma\br\rho\cdot\rho$
\DisplayProof
\end{center}
Second, in the following derivation $\pi_2$, we prove an instance of the transitivity axiom:
\begin{center}
\AX$\sigma \fCenter \sigma$
\AX$\rho \fCenter \rho$
\AX$\rho \fCenter \rho$
\BI$\rho \odot \rho \fCenter \rho \cdot \rho$
\UI$\rho \cdot \rho \fCenter \rho \cdot \rho$
\BI$\sigma\br\rho\cdot\rho\fCenter \sigma\BR\rho\cdot\rho$
\UI$\sigma\LABR\sigma\br\rho\cdot\rho\fCenter \rho\cdot\rho$
\AX$\lambda\fCenter\lambda$
\BI$\rho\cdot\rho\br \lambda\fCenter \sigma\LABR\sigma\br\rho\cdot\rho\BR \lambda$
\UI$(\sigma\LABR\sigma\br\rho\cdot\rho)\LABR\rho\cdot\rho\br \lambda\fCenter  \lambda$
\LeftLabel{Trans}
\UI$\sigma\LABR(\sigma\br\rho\cdot\rho\,;\rho\cdot\rho\br \lambda)\fCenter  \lambda$
\UI$\sigma\br\rho\cdot\rho\,;\rho\cdot\rho\br \lambda\fCenter \sigma\BR \lambda$
\UI$\sigma\br\rho\cdot\rho\,;\rho\cdot\rho\br \lambda\fCenter \sigma\br \lambda$
\DisplayProof
\end{center}

Similarly, a derivation $\pi_3$ can be given of the following instance of the transitivity axiom:

\[\sigma\br\lambda\,; \lambda \br \phi\vdash \sigma\br \phi.\]

Finally, the following derivation $\pi_4$ is the missing piece:

\begin{center}
\AXC{\ \ \ \ \ \ \ \ \ \ \ \ \ \ \ $\vdots$ \raisebox{1mm}{\fns{\begin{tabular}{@{}l}proof for\\BD2\\ \end{tabular} }}}
%\AXC{\ \ \ \ \ $\vdots$ \fns BD3}
\noLine
\UI$\sigma\br\phi\fCenter \sigma\bc \dr\phi$
\AXC{\ \ \ \ \ \ \ \ \ \ \ \ \ \ \ $\vdots$ \raisebox{1mm}{\fns{\begin{tabular}{@{}l}proof for\\BD1\\ \end{tabular} }}}
%\AXC{\ \ \ \ \ $\vdots$ \fns BD2}
\noLine
\UI$\sigma\bc\dr\phi\, ; \dr \sigma\fCenter \dc \dr\phi$
\UI$\sigma\bc\dr\phi\fCenter \dc \dr\phi< \dr \sigma$
\LeftLabel{Cut}
\BI$\sigma\br\phi\fCenter\dc \dr\phi< \dr \sigma$
\UI$\sigma\br\phi\,; \dr \sigma\fCenter\dc \dr\phi$
\DisplayProof
\end{center}

The requested sequent can be then derived using $\pi_1$-$\pi_4$ via cuts and display postulates.

\subsection{Resources having different roles  (The Gift of the Magi)}
\label{ssec:gifts}
{\em The Gift of the Magi} is a short story, written by O.\ Henry and first appeared in 1905, about a young married  couple of very modest means,  Jim (\texttt{j}) and Della (\texttt{d}), who have only two possessions between them which are of value (both monetarily and in the sense that  they take pride in them): Della's unusually long hair ($\eta$),  and Jim's family gold watch ($\omega$).
On Christmas Eve, Della sells her hair to buy a chain ($\gamma$) for Jim's watch, and Jim sells his watch to buy an ivory brush ($\beta$) for Della.

\smallskip
Jim and Della are materially worse off at the end of  the story than at the beginning, since, while the resources $\omega$ and $\eta$ could be used/enjoyed on their own, $\gamma$ and $\beta$ can only be used when coupled with $\omega$ and $\eta$ respectively. In fact, the very choice of $\gamma$ and $\beta$ as presents is a direct consequence of the fact that---besides being used by their respective owners as a means to get the money to buy a present for the other---the resources $\omega$ and $\eta$ are used by the {\em partner} of their respective owners  as  {\em beacons} guiding them in their choice of a present. For instance, their final situation would not have been as bad if Della had bought Jim a new overcoat or  a pair of gloves, or if Jim had bought Della  replacements for her old brown jacket or  hat, the need for which is indicated in the short story. However, each wants to make their present as meaningful as possible to the other one, and hence each {\em targets} his/her present at the one possession the other takes pride in.

%\marginnote{Shall we keep the observation about uniqueness as a footnote? or remove it altogether? It seems to me that, although relevant to the story, we do not really use uniqueness explicitly in the formalization. Or do we do and I am not seeing it?}
Finally, the  uniqueness of the meaningful resource of each agent is the reason why ``the whole affair has something of the dark inevitability of Greek tragedy'' (cit.\ P.~G.~Wodehouse, {\em Thank you, Jeeves}): indeed, $\omega$ (resp.\ $\eta$) is  both the only target for a meaningful present for Jim (resp.\ Della), and also the only means he (resp.\ she) has to acquire such a present for her (resp.\ him). %Contrast this with e.g.\  the situation in which Jim and Della have  {\em two} meaningful resources each, $\omega_1$ and $\omega_2$ for Jim, and  $\eta_1$ and $\eta_2$ for Della, each of which can be used  both as a beacon and as a means to get the money to buy a present. Assume that $\gamma_1$ (resp.\ $\gamma_2$) fits only with $\omega_1$  (resp.\ $\omega_2$) but not with $\omega_2$ (resp.\ $\omega_1$), and $\beta_1$ (resp.\ $\beta_2$) fits only with $\eta_1$  (resp.\ $\eta_2$) but not with $\eta_2$ (resp.\ $\eta_1$). Then, of the 16 possible  outcomes (e.g.\, (a) Jim sells $\omega_1$ to buy $\beta_1$ and Della sells $\eta_2$ to buy $\gamma_2$; (a') Jim sells $\omega_1$ to buy $\beta_2$ and Della sells $\eta_1$ to buy $\gamma_2$; (b) Jim sells $\omega_1$ to buy $\beta_1$ and Della sells $\eta_1$ to buy $\gamma_2$; (b') Jim sells $\omega_1$ to buy $\beta_2$ and Della sells $\eta_2$ to buy $\gamma_2$; etc.), only 4 are as bad as the one in the short story.

%The end of this story is bittersweet rather than bitter, since,having sacrificed their own most valuable possessions in exchange for gifts that turned out to be useless, Jim and Della are comforted by the awareness  that their partner values their happiness more than  his/her own most cherished possession.
%and neither of them  achieves the realization of their respective goals ($G_j$ and $G_d$). This negative outcome has two causes: first, that each agent's (cap)ability to realize his/her goal depends on the other agent's being in possession of a certain resource (specifically, on the other agent's keeping on possessing the resource he/she is initially in possession of); second, their lack of coordination.The story has a moral coda in which O.\ Henry quite cryptically claims that  the ``two foolish children'' are in fact ``the wisest''. Below, we give a formalization in LRC why this is indeed the case.

\medskip
To formalize the observations above, we will need a modification of the language of LRC  capturing the fact, which is sometimes relevant, that resources might have different {\em roles} e.g.\ in the generation or the acquisition of a given resource. %, either objectively or in the perspective of agents.
For instance, in the production of bread, the oven has a different role  as a resource  than  water and flour;
 in shooting sports, the shooter uses a shooting device, projectiles and a target in different roles, etc. Roles cannot be reduced to how resources are combined  irrespective of agency (this aspect is modelled by the pure-resource connectives $\rand$ and $\cdot$); rather, assigning roles to resources is a facet of agency. % and the realization of his/her capabilities entails both the destruction of the projectiles (or at least that they are not anymore in the shooter's possession) and the permanence of the target in the shooter's possession {\em in its role as target}.
 Accordingly, we consider the following  ternary connective for each agent: \[[-, -]\br -: \mathsf{Res}\times \mathsf{Res} \times \mathsf{Res}\rightarrow \mathsf{Fm},\] the intended meaning of which is `the agent is capable of obtaining the resource in the third coordinate, whenever in possession of the resources in the first two coordinates {\em in their respective roles}'.
Algebraically (and axiomatically), this connective is finitely join-reversing in the first two coordinates and finitely meet-preserving in the third one. Its introduction rules and display postulates are as expected:
\begin{center}
\begin{tabular}{c}
\AX$\Gamma\fCenter \alpha$
\AX$\Theta\fCenter \beta$
\AX$\gamma\fCenter \Sigma$
\TI$[\alpha, \beta]\br \gamma\fCenter [\Gamma, \Theta]\BR \Sigma$
\DisplayProof
$\quad$
\AX$X\fCenter [\alpha, \beta]\BR \gamma$
\UI $X\fCenter [\alpha, \beta]\br \gamma$
\DisplayProof
 \\
 \\
\AXC{$X\vdash [\Gamma, \Theta]\BR \Sigma$}
\UIC{$[\Gamma, \Theta]\LABR X \vdash \Sigma$}
\DisplayProof
$\quad$ 
\AX$X\fCenter [\Gamma, \Theta]\BR \Sigma$
\UI$\Gamma \fCenter [X, \Theta]\ABR^{\!1\,}\Sigma$
\DisplayProof
$\quad$
\AX$X\fCenter [\Gamma, \Theta]\BR \Sigma$
\UI$\Theta \fCenter [\Gamma, X]\ABR^{\!2\,}\Sigma$
\DisplayProof
\\
\end{tabular}
\end{center}
In addition, we need two unary diamond operators $\dr^1, \dr^2: \mathsf{Res}\rightarrow \mathsf{Fm}$ for each agent, the intended meaning of which is `the agent is in possession of the resource (in the argument) {\em in the first} (resp.\ {\em second}) {\em role}'.  The basic algebraic and axiomatic behaviour of $\dr^1$ and $\dr^2$ coincides with that of $\dr$, hence the introduction and display rules relative to these connectives are like those given for $\dr$. The various roles and their differences can be understood and formalized in different ways relative to different settings.  In the specific situation of the short story, we stipulate that $\dr^2$ has the meaning usually attributed to $\dr$, and understand $\dr^1\sigma$ as `the agent has resource $\sigma$ available in the role of target (or beacon)'. %These diamonds encode a more specified way in which a certain resource is in possession of an agent, which supports the plausibility of the axioms $\dr^1\sigma\pra \dr\sigma$ and $\dr^2\sigma\pra \dr\sigma$.

The interaction of these connectives, and the difference in meaning between $\dr^1$ and $\dr^2$, are captured by the following axiom:
\begin{equation}
\label{eq:roles-bdtwo}
\dr^1\sigma\land \dr^2\xi\land [\sigma, \xi]\br\chi\pra \dc\dr^2 \chi,\end{equation}
which is equivalent on perfect LRC-algebras to the following analytic rule:
\begin{center}
\AX$\DC\DR^{2\,} [\Sigma, \Xi]\LABR X\fCenter Y$
\RightLabel{RR}
\UI$\DR^1 \Sigma\,; \DR^{2\,}\Xi\,; X\fCenter Y$
\DisplayProof
\end{center}

%Also, the following axiom is plausible in this situation for any agent:
%\[\dc A\land \dc(A\pra B) \pra \dc B,\]
%which is equivalent to the following analytic rule:
%\begin{center}
%\AX$\DC (X \,; Y) \fCenter Z$
%\LeftLabel{MP$_{\dc}$}
%\UI$\DC X \,; \DC Y \fCenter Z$
%\DisplayProof
%\end{center}

Finally, in the specific case at hand, we will use the rules corresponding to the following slightly modified multi-agent versions of axiom \eqref{eq:roles-bdtwo}:

\[\dr_{\aj}^{1} \sigma \land \dr_{\aj}^{2} \xi \land [\sigma, \xi] \br_{\aj} \chi \pra \dc_{\aj} \dr_{\ad}^{2} \chi \quad \mbox{and} \quad \dr_{\ad}^{1} \sigma \land \dr_{\ad}^{2} \xi \land [\sigma, \xi] \br_{\ad} \chi \pra \dc_{\ad} \dr_{\aj}^{2} \chi.\]
The following table shows the assumptions of the present case study:
\begin{center}
\begin{tabular}{rr||l|l|l}
            &                  & Initial state      & Capabilities      & Abilities                                                                                                                          \\
           % &                  & assumptions        & assumptions     & assumptions                                                                                                                 %\\
\hline
Jim      & $\texttt{j}$ & $\dr^{1}_{\aj} \eta\,$ $\dr^{2}_{\aj} \omega $   & $[\eta, \omega]\br_{\aj} \beta$   &  $\dc_{\aj} \dr^{2}_{\ad} \beta \pra \dc_{\aj} \neg \dr^{2}_{\aj} \omega$                \\ %$\omega\br_\aj\beta \pra \dc_\aj \neg \dr_\aj\omega$ \\
Della   & $\texttt{d}$ & $\dr^{1}_{\ad} \omega\,$ $\dr^{2}_{\ad} \eta $    & $[\omega, \eta] \br_{\ad} \gamma$  & $\dc_{\ad} \dr^{2}_{\aj} \gamma \pra \dc_{\ad} \neg \dr^{2}_{\ad} \eta$      \\ % $\eta\br_\ad\gamma \pra \dc_\ad \neg \dr_\ad\eta$        \\
\end{tabular}
\end{center}
Let $H$ be the structural conjunction of the assumptions above. We aim at deriving the following sequent in the calculus D.LRC to which the analytic rules introduced above have been added:
   \[H \vdash \dc_{\aj} \neg \dr^{2}_{\aj} \omega \land \dc_{\aj} \dr^{2}_{\ad} \beta \land \dc_{\ad} \neg \dr^{2}_{\ad} \eta \land \dc_{\ad} \dr^{2}_{\aj} \gamma.\]

We do it in several steps: first,  the following derivation $\pi_1$:
\begin{center}
\begin{tabular}{c}
\AX$\eta\fCenter \eta$
\AX$\omega\fCenter \omega$
\AX$\beta\fCenter \beta$
\TI$[\eta, \omega] \br_{\aj} \beta \fCenter [\eta, \omega] \BR_{\aj} \beta$
\UI$[\eta, \omega] \LABR_{\aj} [\eta, \omega] \br_{\aj} \beta \fCenter \beta$
\UI$\DR^{2}_{\ad} \Big([\eta, \omega] \LABR_{\aj} [\eta, \omega] \br_{\aj} \beta\Big)\fCenter \dr^{2}_{\ad} \beta$
\UI$\DC_{\aj} \DR^{2}_{\ad} \Big([\eta, \omega] \LABR_{\aj} [\eta, \omega] \br_{\aj} \beta\Big) \fCenter \dc_{\aj} \dr^{2}_{\ad} \beta$
\RightLabel{RR$_{\aj\ad}$}
\UI$(\DR^{1}_{\aj} \eta\, ; \DR^{2}_{\aj} \omega)\, ; [\eta, \omega] \br_{\aj} \beta \fCenter \dc_{\aj} \dr^{2}_{\ad} \beta$
\dashedLine
\UI$(\dr^{1}_{\aj} \eta\, ; \dr^{2}_{\aj} \omega)\, ; [\eta, \omega] \br_{\aj} \beta \fCenter \dc_{\aj} \dr^{2}_{\ad} \beta$
\DisplayProof
 \\
\end{tabular}
\end{center}
With an analogous derivation $\pi_2$ we can prove that
\[\dr^{1}_{\ad} \omega\, ; \dr^{2}_{\ad} \eta\, ; [\omega, \eta] \br_{\ad} \gamma \fCenter \dc_{\ad} \dr^{2}_{\aj} \gamma.\]
Next, let  $\pi_3$ be the following derivation:
\begin{center}
\begin{tabular}{c}
\AX$\beta \fCenter \beta$
\UI$\DR^{2}_{\ad} \beta \fCenter \dr^{2}_{\ad} \beta$
\UI$\DC_{\aj} \dr^{2}_{\ad} \beta \fCenter \dc_{\aj} \dr^{2}_{\ad} \beta$
\UI$\dc_{\aj} \dr^{2}_{\ad} \beta \fCenter \dc_{\aj} \dr^{2}_{\ad} \beta$

\AX$\omega \fCenter \omega$
\UI$\DR^{2}_{\aj} \omega \fCenter \dr^{2}_{\aj} \omega$
\UI$\dr^{2}_{\aj} \omega \fCenter \dr^{2}_{\aj} \omega$

\AX$\bot \fCenter \bot$
\BI$\dr^{2}_{\aj} \omega \pra \bot \fCenter \dr^{2}_{\aj} \omega > \bot$
\UI$\dr^{2}_{\aj} \omega \pra \bot \fCenter \dr^{2}_{\aj} \omega \pra \bot$
\RightLabel{\fns def}
\UI$\neg \dr^{2}_{\aj} \omega \fCenter \neg \dr^{2}_{\aj} \omega$
\UI$\DC_{\aj} \neg \dr^{2}_{\aj} \omega \fCenter \dc_{\aj} \neg \dr^{2}_{\aj} \omega$
\UI$\dc_{\aj} \neg \dr^{2}_{\aj} \omega \fCenter \dc_{\aj} \neg \dr^{2}_{\aj} \omega$

\BI$\dc_{\aj} \dr^{2}_{\ad} \beta \pra \dc_{\aj} \neg \dr^{2}_{\aj} \omega \fCenter \dc_{\aj} \dr^{2}_{\ad} \beta > \dc_{\aj} \neg \dr^{2}_{\aj} \omega$
\UI$\dc_{\aj} \dr^{2}_{\ad} \beta \,; \dc_{\aj} \dr^{2}_{\ad} \beta \pra \dc_{\aj} \neg \dr^{2}_{\aj} \omega \fCenter \dc_{\aj} \neg \dr^{2}_{\aj} \omega$
\DisplayProof
 \\
\end{tabular}
\end{center}

With an analogous derivation $\pi_4$ we can prove that
\[\dc_{\ad} \dr^{2}_{\aj} \gamma\, ; \dc_{\ad} \dr^{2}_{\aj} \gamma \pra \dc_{\ad} \neg \dr^{2}_{\ad} \eta \fCenter \dc_{\ad} \neg \dr^{2}_{\ad} \eta.\]
Then, by applying cut (and left weakening) on $\pi_{1}$ and $\pi_{3}$ one derives:
\[\dr^{1}_{\aj} \eta\, ; \dr^{2}_{\aj} \omega\, ; [\eta, \omega] \br_{\aj} \beta \, ; \dc_{\aj} \dr^{2}_{\ad} \beta \pra \dc_{\aj} \neg \dr^{2}_{\aj} \omega \fCenter \dc_{\aj} \neg \dr^{2}_{\aj} \omega.\]
Likewise, by applying cut (and left weakening) on $\pi_{2}$ and $\pi_{4}$ one derives:
\[\dr^{1}_{\ad} \omega\, ; \dr^{2}_{\ad} \eta\, ; [\omega, \eta] \br_{\ad} \gamma \, ; \dc_{\ad} \dr^{2}_{\aj} \gamma \pra \dc_{\ad} \neg \dr^{2}_{\ad} \eta \fCenter \dc_{\ad} \neg \dr^{2}_{\ad} \eta.\]
The derivation is concluded with applications of right-introduction of $\land$ and left contraction rules.

\subsection{From local to global resilience (two production lines)}
\label{ssec:resilience}
Resilience is the ability of an agent or an organization to realize their goals notwithstanding unexpected changes and disruptions. The language of LRC provides a natural way to understand resilience as the capability to realize one's goal(s) in a range of situations characterized by the reduced availability of key resources. Consider for example a factory with two production lines for products $\gamma_1$ and $\gamma_2$. Product $\gamma_1$ is of higher quality than $\gamma_2$ and can only be produced using  resource $\alpha$, the availability of which is subject to fluctuations. Product $\gamma_2$ can be produced using either resource $\alpha$ or $\beta$, and the availability of $\beta$ is not subject to fluctuations. It is interesting to note that the `local' resilience in the production of $\gamma_2$ (namely, the fact that any shortage in  $\alpha$  can be dealt with by switching to $\beta$)  results in the resilience of both production lines. Indeed, when  $\alpha$ is available for only one of the two production lines, all of it can be employed in the production line for $\gamma_1$, and the production of $\gamma_2$ is switched to $\beta$.
In the formal treatment that follows, we notice that the  axioms $\dr\sigma\land \sigma\br\pi\pra \dr\pi$ and $\sigma\br \chi\land \pi\br \xi\pra \sigma\cdot\pi\br \chi\cdot\xi$ hold for the setting described above. These axioms are analytic and are equivalent on perfect LRC-algebras to  the following rules:
\begin{center}
\begin{tabular}{r c l}
\AX$X\fCenter \Gamma\BR \, \ADR Y$
\LeftLabel{BDR}
\UI$X\fCenter \DR\Gamma > Y$
\DisplayProof
&$\quad\quad$&
\AX$(\Gamma\LABR X)\odot(\Pi\LABR Y) \fCenter \Delta$
\RightLabel{Scalab}
\UI$(\Gamma\odot \Pi)\LABR(X\, ; Y)  \fCenter \Delta$
\DisplayProof
\\
\end{tabular}
\end{center}

\begin{center}
\begin{tabular}{c|c}
Resources                                                    & Capabilities                               \\
%assumptions                                                 & assumptions                             \\
\hline
$\dr (((\alpha \cdot \alpha)\ror \alpha)\cdot\beta) $    %                  \\
%$\dr\beta$       & \\%$\alpha \ror \beta \bc G_2$ \\
 &$\alpha \br \gamma_1$ \\ %\\ $\gamma_1\cdot\gamma_2 \bc (G_1\land G_2)$  \\
& $\alpha \ror \beta \br \gamma_2$  \\
\end{tabular}
\end{center}
We aim at showing that the assumptions above are enough to conclude that the factory is able to realize the production of both $\gamma_1$ and $\gamma_2$:
\[\dr (((\alpha \cdot \alpha)\ror \alpha)\cdot\beta)\, ; \alpha \br \gamma_1\, ; \alpha \ror \beta \br \gamma_2\,\vdash \dr(\gamma_1\cdot\gamma_2).\]
%\[\dr (((\alpha \cdot \alpha)\ror \alpha)\cdot\beta)\, ; \alpha \br \gamma_1\, ; \alpha \ror \beta \br \gamma_2\, ; \gamma_1\cdot\gamma_2 \bc (G_1\land G_2)\vdash \dc (G_1\land  G_2).\]
%Notice that the following is an instance of axiom BD2, and hence is derivable in D.LRC:
%\[\dr (\gamma_1\cdot\gamma_2)\land \gamma_1\cdot\gamma_2\bc (G_1\land G_2)\vdash \dc(G_1\land G_2). \]
%Hence modulo cut, it is enough to show that
%\[\dr (((\alpha \cdot \alpha)\ror \alpha)\cdot\beta)\, ; \alpha \br \gamma_1\, ; \alpha \ror \beta \br \gamma_2\vdash \dr (\gamma_1\cdot\gamma_2).\]

Notice that the following is an instance of $\dr\sigma\land \sigma\br\pi\pra \dr\pi$, and hence is derivable using the rule BDR:
\[\dr (((\alpha \cdot \alpha)\ror \alpha)\cdot\beta)\, ;((\alpha \cdot \alpha)\ror \alpha)\cdot\beta\br \gamma_1\cdot\gamma_2\vdash \dr (\gamma_1\cdot\gamma_2).\]

Hence, modulo cut and left weakening, it is enough to show that
\[ \alpha \br \gamma_1\, ; \alpha \ror \beta \br \gamma_2\vdash ((\alpha \cdot \alpha)\ror \alpha)\cdot\beta\br \gamma_1\cdot\gamma_2.\]
Notice that:
\begin{center}
\AXC{\ \ \ \ \ \ \ \ \ \ \ \ \ \ \ $\vdots$ \raisebox{1mm}{\fns{\begin{tabular}{@{}l}proof for\\R4\\ \end{tabular} }}}
%\AXC{\ \ \ $\vdots$ \fns R4}
\noLine
\UI$((\alpha\cdot\alpha)\ror\alpha)\cdot\beta\fCenter (\alpha\cdot\alpha)\cdot\beta\ror\alpha\cdot\beta $
\AX$\gamma_1 \fCenter \gamma_1$
\AX$\gamma_2 \fCenter \gamma_2$
\BI$\gamma_1 \odot \gamma_2 \fCenter \gamma_1 \cdot \gamma_2$
\UI$\gamma_1 \cdot \gamma_2 \fCenter \gamma_1 \cdot \gamma_2$
\BI$(\alpha\cdot\alpha)\cdot\beta\ror\alpha\cdot\beta\br \gamma_1\cdot\gamma_2\fCenter ((\alpha\cdot\alpha)\ror\alpha)\cdot\beta\BR \gamma_1\cdot\gamma_2$
\UI$(\alpha\cdot\alpha)\cdot\beta\ror\alpha\cdot\beta\br \gamma_1\cdot\gamma_2\fCenter ((\alpha\cdot\alpha)\ror\alpha)\cdot\beta\br \gamma_1\cdot\gamma_2$
\DisplayProof
\end{center}
Hence, modulo cut and left weakening, it is enough to show that
\[ \alpha \br \gamma_1\, ; \alpha \ror \beta \br \gamma_2\vdash (\alpha\cdot\alpha)\cdot\beta\ror\alpha\cdot\beta\br \gamma_1\cdot\gamma_2.\]
Indeed, a derivation for the sequent above is:

\begin{center}
\AXC{\ \ \ $\vdots$ \fns $\pi_1$}
\noLine
\UI$\alpha \br \gamma_1\, ; \alpha \ror \beta \br \gamma_2\fCenter \alpha\cdot (\alpha\ror\beta)\br \gamma_1\cdot\gamma_2$
\AXC{\ \ \ $\vdots$ \fns $\pi_2$}
\noLine
\UI$\alpha\cdot (\alpha\ror\beta)\br \gamma_1\cdot\gamma_2\fCenter (\alpha\cdot\alpha)\cdot\beta\ror\alpha\cdot\beta\br \gamma_1\cdot\gamma_2$
\BI$\alpha \br \gamma_1\, ; \alpha \ror \beta \br \gamma_2\fCenter (\alpha\cdot\alpha)\cdot\beta\ror\alpha\cdot\beta\br \gamma_1\cdot\gamma_2$
\DisplayProof
\end{center}

where $\pi_1$ is the following derivation:

\begin{center}

\AX$\alpha\fCenter \alpha$
\AX$\gamma_1\fCenter \gamma_1$
\BI$\alpha \br \gamma_1\fCenter \alpha\BR \gamma_1$
\UI$\alpha\LABR \alpha \br \gamma_1\fCenter\gamma_1$

\AX$\alpha \fCenter \alpha$
\AX$\beta \fCenter \beta$
\BI$\alpha \ror \beta \fCenter \alpha \RAND \beta$

\UI$\alpha\ror \beta\fCenter \alpha\ror\beta$
\AX$\gamma_2\fCenter \gamma_2$
\BI$\alpha\ror\beta \br \gamma_2\fCenter \alpha\ror\beta\BR \gamma_2$
\UI$(\alpha\ror\beta)\LABR \alpha \ror \beta \br \gamma_2\fCenter\gamma_2$
\BI$(\alpha\LABR \alpha \br \gamma_1)\odot (\alpha\ror\beta)\LABR \alpha \ror \beta \br \gamma_2\fCenter \gamma_1\cdot\gamma_2$
\LeftLabel{Scalab}
\UI$\alpha\odot (\alpha\ror\beta)\LABR (\alpha \br \gamma_1\, ; \alpha \ror \beta \br \gamma_2)\fCenter \gamma_1\cdot\gamma_2$
\UI$\alpha \br \gamma_1\, ; \alpha \ror \beta \br \gamma_2\fCenter \alpha\odot (\alpha\ror\beta)\BR \gamma_1\cdot\gamma_2$

\UI$\alpha\odot (\alpha\ror\beta) \fCenter \alpha \br \gamma_1\, ; \alpha \ror \beta \br \gamma_2\ABR \gamma_1\cdot\gamma_2$
\UI$\alpha\cdot (\alpha\ror\beta) \fCenter \alpha \br \gamma_1\, ; \alpha \ror \beta \br \gamma_2\ABR \gamma_1\cdot\gamma_2$

\UI$\alpha \br \gamma_1\, ; \alpha \ror \beta \br \gamma_2\fCenter \alpha\cdot (\alpha\ror\beta)\BR \gamma_1\cdot\gamma_2$

\UI$\alpha \br \gamma_1\, ; \alpha \ror \beta \br \gamma_2\fCenter \alpha\cdot (\alpha\ror\beta)\br \gamma_1\cdot\gamma_2$
\DisplayProof
\end{center}

and $\pi_2$ is the following derivation:
\begin{center}
{
\begin{tabular}{c}
\AX$\alpha \fCenter \alpha$
\UI$\alpha \odot \Phi \fCenter \alpha$
\UI$\Phi \fCenter \alpha \RCDOT \alpha$
\UI$\alpha \fCenter \alpha \RCDOT \alpha$
\UI$\alpha \odot \alpha \fCenter \alpha$
\UI$\alpha \cdot \alpha \fCenter \alpha$
\UI$\alpha \cdot \alpha \odot \beta \fCenter \alpha$
\UI$(\alpha \cdot \alpha \cdot \beta \fCenter \alpha$
\AX$\alpha \fCenter \alpha$
\UI$\alpha \odot \Phi \fCenter \alpha$
\UI$\Phi \fCenter \alpha \RCDOT \alpha$
\UI$\alpha \fCenter \alpha \RCDOT \alpha$
\UI$\alpha \odot \alpha \fCenter \alpha$
\UI$\alpha \odot \beta \fCenter \alpha$
\UI$\alpha \cdot \beta \fCenter \alpha$
\UI$\alpha \cdot \beta \fCenter \alpha \RAND \beta$
\UI$\alpha \cdot \beta \fCenter \alpha \ror \beta$
\BI$(\alpha\cdot\alpha\cdot\beta) \ror (\alpha\cdot\beta)\fCenter \alpha\cdot(\alpha\ror\beta)$

\AX$\gamma_1 \fCenter \gamma_1$
\AX$\gamma_2 \fCenter \gamma_2$
\BI$\gamma_1 \odot \gamma_2 \fCenter \gamma_1 \cdot \gamma_2$
\UI$\gamma_1 \cdot \gamma_2 \fCenter \gamma_1 \cdot \gamma_2$

\BI$\alpha\cdot (\alpha\ror\beta)\br \gamma_1\cdot\gamma_2\fCenter (\alpha\cdot\alpha\cdot\beta)\ror\alpha\cdot\beta\BR \gamma_1\cdot\gamma_2$
\UI$\alpha\cdot (\alpha\ror\beta)\br \gamma_1\cdot\gamma_2\fCenter (\alpha\cdot\alpha\cdot\beta)\ror\alpha\cdot\beta\br \gamma_1\cdot\gamma_2$
\DisplayProof
\end{tabular}
 }
\end{center}

{}

 \section{Conclusions and further directions}

\paragraph{Resources and capabilities.} In the present paper, a logical framework is introduced aimed at capturing and reasoning about  resource flow within organizations. This framework contributes to the line of investigation of the logics of agency (cf.~e.g.~\cite{dignum2004model,Dignum,elgesem1997modal,belnap1991backwards,chellas1995bringing}) by focusing specifically on the {\em resource}-dimension of agents' (cap)abilities (e.g.~to use resources to achieve goals, to transform resources into other resources, and to coordinate the use of resources with other agents). Formally, the logic of resources and capabilities (LRC) has been introduced in a language consisting of formula-terms and resource-terms. Besides pure-formula and pure-resource connectives, the language of LRC includes connectives bridging the two types in various ways. Although action-terms are not included in LRC, perhaps the logical system of which LRC is most reminiscent is the logic of capabilities introduced in \cite{van1994logic},  which formalizes   the capabilities of agents to perform actions. Indeed, looking past the differences between the two formalisms deriving from the inherent differences between actions and resources, the focus of both axiomatizations is {\em interaction},  between  (cap)abilities and actions in \cite{van1994logic}, and between (cap)abilities and resources in the present paper. Precisely its focus on interaction makes it  worthwhile  to recast the logical framework of  \cite{van1994logic} in a multi-type environment.
%\marginnote{Expand on the approach of v.d.Hoek and remark that while they have actions and not resources we have resources and not actions. Natural further direction is to have also resources, actions and capabilities}
%The logical formalism of  is similar in spirit to STIT logics [], in that it aims at describing actions which are specifically brought about by the activity of agents (rather than actions capturing a more general notion of change, which can be formulated independently of agency). Hence, rather than being treated as primitive, as is done in PDL, actions are regarded as emerging from a hierarchy of more primitive notions, the most basic of which are agency, and agents' capabilities.

\paragraph{A study in algebraic proof theory.} The main technical contribution of the paper is the introduction of the multi-type calculus D.LRC. The definition of this calculus and the proofs of its basic properties hinge on the integration of two  theories in algebraic logic and structural proof theory---namely, {\em unified correspondence} and {\em multi-type calculi}---which  originated independently of each other. This integration  contributes to the research program of {\em algebraic proof theory} \cite{ciabattoni2009expanding,ciabattoni2012algebraic},  to which the results of the present paper pertain. Specifically, the rules of D.LRC are introduced, and their soundness  proved, by applying (and adapting) the ALBA-based methodology of \cite{GMPTZ} (cf.\ also \cite{ciabattoni2016power} for a purely proof-theoretic perspective on the same methodology); cut elimination is proved `Belnap-style', by verifying that D.LRC satisfies the assumptions of the cut elimination metatheorem for multi-type calculi of \cite{Trends}; conservativity is proved  following the general proof strategy for conservativity illustrated in \cite{GMPTZ}, to which  the canonicity of the axioms of the Hilbert-style presentation of LRC is key.

It is perhaps worth stressing that  the theory of proper display calculi developed in \cite{GMPTZ} cannot  be applied directly to the Hilbert-style presentation of LRC, for two reasons. Firstly, the setting of \cite{GMPTZ} is a pure-formula setting, while the setting of the present paper is multi-type. However, the results of \cite{GMPTZ} can   be ported to the multi-type setting (as done also in \cite{Inquisitive,latticelogPdisplayed,linearlogPdisplayed}); indeed, the algorithm ALBA and the definition of analytic inductive inequalities are grounded in the order-theoretic properties of the algebraic interpretations of the logical connectives, and remain fundamentally unchanged when applied to maps with the required order-theoretic properties, irrespective of whether these maps are operations on one algebra or between different algebras. The second, more serious reason is that the algebraic interpretation of the capability connective $\bc$ is a map which reverses finite joins in its first coordinate but is only {\em monotone} (rather than finitely meet-preserving) in its second coordinate. Hence, (the multi-type version of) the  definition of (analytic) inductive inequalities given in \cite{GMPTZ} does not apply to many  axioms of the Hilbert-style presentation of LRC, and hence some  results (e.g.\ the canonicity results of Section \ref{ssec:algebraic canonicity}) could not be immediately inferred by directly applying the general theory. However, as we saw in Section \ref{ssec:soundness}, the algorithm ALBA is successful on the LRC axiomatization, which suggests the possibility of generalizing these results to arbitrary multi-type signatures in which operations are allowed to be only monotone or antitone in some coordinates. Moreover, unified correspondence theory covers various settings,  from general lattice-based propositional logics \cite{CoPa:non-dist,CFPPTW,CFPPTW17,CP:constructive}, to regular \cite{PaSoZh16} and monotone modal logics \cite{FrPaSa16},  (distributive) lattice-based mu-calculi \cite{CFPS15,CoCr14,CCPZ},  hybrid logic \cite{ConRob} and many-valued logic \cite{LeRoux:MThesis:2016}. It would be interesting to investigate whether structural proof calculi for  each of these settings (or for multi-type logics based on them) could be defined by suitably extending the techniques employed in the design of D.LRC.

\paragraph{Proof-theoretic formalizations of social behaviour.} In Section \ref{sec:casestudies}, we have  discussed the formalization of  situations revolving around some instances of resource flow. These situations have been captured as  inferences or sequents in the language of LRC, and derived in the basic calculus D.LRC or in some of its analytic extensions. This proof-theoretic analysis makes it possible to single out the steps and assumptions which are {\em essential} to a given situation. For instance, thanks to this analysis, it is clear that the full power of classical logic is {\em not essential} to any case study we treated. In fact, as can be readily verified by inspection, many  derivations treated in Section \ref{sec:casestudies} need  less than the full power of intuitionistic logic, which is the propositional base of LRC. Also, reasoning from assumptions in a given proof-theoretic environment corresponds semantically to reasoning on {\em all} the models of that environment satisfying those assumptions.  This is a {\em safer} practice than e.g.\ starting out with  an ad-hoc model, since %to   rather than on some ad hoc model %allows to  naturally glide over  inessential facts; for instance in the case study of the homework correction we do not know whether the fact that one agent has a certain set of assignments implies that the other does not and so on
it makes it impossible to rely on some implicit assumption or other  extra feature of a chosen model.
% \marginnote{the proof-theoretic approach on modelling situations; firstly classical logic inessential; also great advantage: it makes it possible to not saying anything about inessential facts; for instance in the case study of the homework correction we do not know whether the fact that one agent has a certain set of assignments implies that the other does not and so on.}

 \paragraph{The pure-resource fragment.} In Section \ref{ssec:HLRC} we mentioned that  the fact that $1$ coincides with the weakest resource entails (and is in fact equivalent to) the validity of the sequents $\alpha\cdot\beta\vdash \alpha$ and $\alpha\cdot\beta\vdash \beta$, which in some contexts seems too restrictive. How to relax this restriction is current work in progress. However, this restriction brings also some advantages. Indeed, as discussed earlier on in Section \ref{ssec:HLRC}, this restriction makes the pure resource fragment of LRC very similar to (the exponential-free fragment of) linear affine logic, which, unlike general linear logic, is decidable \cite{kopylov1995decidability,OkadaTerui}. Hence, this leaves open the question of the decidability of LRC (see also below).

 \paragraph{Agents as first-class citizens.} In the present paper, we focused on the basic setting of LRC, and for the sake of not overloading notation and machinery, we have treated agents as parameters. However, a fully multi-type treatment  would include terms of type $\mathsf{Ag}$ (agents) in the language, as done e.g.\ in \cite{Multitype}. This will be particularly relevant to the formalization of organization theory, where terms of type $\mathsf{Ag}$ will represent members of an organization, and  $\mathsf{Ag}$ might be endowed with additional structure: for instance it can be a graph (capturing networks of agents), or a partial order (capturing hierarchies), or partitioned in coalitions or teams. Having agents as first-class citizens of the language will also make possible to attribute {\em roles} to them, analogously to the way roles  are attributed to resources in Section \ref{ssec:gifts}. Roles in turn could provide concrete handles  towards the modelling of agent coordination.

 \paragraph{Group capabilities.} Closely related to the issue of the previous paragraph is the formalization of  various forms of  group capabilities. This theme is particularly relevant to organization theory, since it might help to capture e.g.\ the contribution of leadership to the results of an organization, versus the advantages of self-organization. Another interesting notion in organization theory which could benefit from a formal theory of group capabilities is {\em tacit group knowledge} \cite{shamsie2013looking}, emerging from the individual capabilities to adapt, often implicitly,  to the behaviour of others. %, especially where it regards the capability of a group to perform a certain task well because of the web of reciprocal, and often implicit, expectations group members have of each other's behaviour.

 \paragraph{Different types of resources.} Key to the analysis of the case study of  Section \ref{ssec:crow} was the interplay between reusable and non-reusable resources. The treatment of this case study suggests that analytic extensions of D.LRC can be used to develop a formal theory of resource flow that also captures other differences between resources (e.g.\ storable vs.\ non storable, scalable vs.\ non scalable),   their interaction, direct or mediated by agents, in the production process, or in  facilitating more generally the  competitive success of the organization \cite{mol2011resources}.

 \paragraph{Pre-orderings on resources.} In Section \ref{ssec:gifts}, we mentioned that  the resources the agents possess at the end  of the story cannot be used without those they possess at the beginning, while these can be used on their own. This observation suggest that alternative or additional orderings of resources can be considered and studied, such as the `dependence' preorder  between resources, which might be relevant to the analysis of some situations.

\paragraph{Comparing capabilities.} %There is no axiom in LRC encoding specific information on capabilities which would support inferences deriving formulas of the form $\neg(\alpha\triangleright A)$.\marginnote{do you like this explanation? shall I add it to the conclusions?} Axioms encoding specific such conditions need to be added for each specific context.  For instance, it would be very straightforward, but perhaps not very illuminating, to rephrase   the case study in Section 5.3 (the Gift of the Magi) so as to make it into an example of failure of capabilities (neither Jim nor Della is  capable to make the other happy), by defining what it means to {\em not} be capable in the particular context of this case study.   From the viewpoint of the general motivation for the introduction of the LRC framework  (capture and reason about the resource flow within organizations)  the main interest is planning and decision-making, and hence being able to formulate sufficient conditions for the {\em success} of a plan, on the base of what the agents are capable of. In this respect, rather than  formulating general criteria for establishing whether an agent is {\em not} capable in absolute terms,
 The logic LRC  provides a formal environment where to explore the consequences for organizations of some agent's being {\em more} capable than some other agent at bringing about a certain state of affairs. In this environment, we can express that  agent $\aga$ is {\em at least as capable} than agent $\agb$ at bringing about $A$ e.g.~when $\alpha\bc_{\!\aga}A$ and $\beta\bc_{\!\agb}A$, and $\beta\vdash \alpha$ (i.e.~to bring about the same state of affairs, $\agb$  uses a resource which is at least as powerful as, possibly more powerful than, the resource used by $\aga$). Ricardo's economic theory of comparative advantage with regard to the division of labour in organizations \cite{ricardo1891principles} can be formalized on the basis of capabilities differentials.

 \paragraph{Algebraic canonicity and relational semantics.} The theory of canonical extensions provides a way to extract  relational semantics from the algebraic semantics via algebraic canonicity.
In Section \ref{ssec:algebraic canonicity}, we have shown  that the logic LRC is  complete w.r.t.\ {\em perfect} LRC-algebraic models. Via standard discrete Stone-type duality, perfect LRC-algebraic models can be associated with set-based structures similar to Kripke models, thus providing complete relational semantics for LRC. The specification of this relational semantics and its properties is part of future work.

\paragraph{Semantics of Petri nets.} We are currently studying Petri nets as an alternative semantic framework for LRC. In particular, the reachability problem for finite Petri nets is equivalent to the deducibility problem for sequents in finitely axiomatized theory in the pure-tensor fragment of linear logic \cite{MartiMeseguer,Troelstra92}. More recently, \cite{EngbergWinskel} proved completeness for several versions of linear logic w.r.t.\ Petri nets. We are  investigating similar issues in the setting of LRC.

\paragraph{Decidability, finite model property, complexity.} The computational properties of LRC such as decidability and complexity are certainly of interest. In particular, two, in general distinct, problems are to be considered: the decidability of the set of theorems, and the decidability of the (finite) consequence relation.\footnote{The two problems coincide in presence of  deduction theorem, which is available in intuitionistic logic and for the formula-fragment of  LRC, but not for the pure-resource fragment of LRC.} %However, in our setting we have to take in account the fact that we use additional modal rules in the Hilbert system, and, more importantly, that we work in a multi-type setting.}.

A standard argument establishing decidability is via the so-called finite model property (FMP), i.e.\ proving  that any non-theorem  can be refuted in a finite structure. Together with finite axiomatizability and completeness of the underlying logic, FMP entails the decidability of the set of theorems. For the second problem a stronger property is needed: the finite embeddability property, which can be seen as the finite model property for quasi-identities and, together with finite axiomatizability and completeness, entails the decidability of the finite consequence relation of the underlying logic.

We wish to stress that the decidability problems for LRC subsume the complexity and decidability of certain substructural logics. Indeed, as mentioned earlier, the pure-resource fragment of LRC is similar to (propositional, exponential-free) linear affine logic, which essentially coincides with the distributive Full Lambek calculus with weakening, a logic for which the finite consequence relation, and hence the set of theorems, are known to be decidable (see \cite{ono1998decidability,OkadaTerui}); FEP for integral residuated groupoids has been proved in \cite{blok2005finite}, for a simple proof of FEP in the distributive setting see also \cite{hanikova2014finite}, where a coNEXP upper bound is obtained. We hope we can use the algebraic semantics of LRC to investigate, and hopefully establish decidability of LRC and its variants using FMP or FEP.

\paragraph{Syntactic decidability.} %Categorizing analytic rules and .}
An alternative path  towards decidability for LRC consists in adapting the techniques developed in %following the strategy introduced by Gentzen to prove the decidability of classical propositional logic, in
\cite{kopylov1995decidability}, where a syntactic proof is given of  the decidability of full propositional affine linear logic, by showing that it is enough to consider sequents in a suitable normal form. %A future direction concerns  investigating whether the proof of the decidability of LRC can be obtained by suitably adapting these techniques.
An encouraging hint is the fact that the full Lambek calculus with weakening is decidable \cite{OkadaTerui,ono1998decidability}. However, it is also known that, for certain substructural logics, distributivity is problematic for decidability.

\appendix

%\marginnote{The following section is already edited. Please read it and let us know if you have any suggestions.}
\section{Proper multi-type calculi and their cut elimination}
\label{sec:properly semi displayable}
In the present section, we report on the Belnap-style meta-theorem that we appeal to in order to show that the calculus introduced in Section \ref{sec: Display-style sequent calculus D.LRC} enjoys cut elimination. This meta-theorem was proven in \cite{Trends} for the so-called {\em proper multi-type calculi}. In order to make the exposition self-contained, in what follows we will report the definition of proper multi-type calculi and the statement of the meta-theorem.

\begin{definition}\label{def:gen strongly type unif}
A sequent $x \vdash y$ is {\em type-uniform} if $x$ and $y$ are of the same type $\mathsf{T}$ (cf.\ \cite[Definition 3.1]{Multitype}).
\end{definition}

\begin{definition}\label{def:proper multi-type calculus}
A {\em proper multi-type calculus} is any calculus in a multi-type language satisfying the following list of conditions:\footnote{See \cite{GAV} for a discussion on C$'_5$ and C$''_5$.}

\noindent \textbf{C$_1$: Preservation of operational terms.\;} Each operational term occurring in a premise of an inference rule {\em inf} is a subterm of some operational term in the conclusion of {\em inf}.

\noindent \textbf{C$_2$: Shape-alikeness of parameters.\;} Congruent parameters (i.e.\ non-active terms in the application of a rule) are occurrences of the same structure.

\noindent \textbf{C$'_2$: Type-alikeness of parameters.\;}  Congruent parameters have exactly the same type. This condition bans the possibility that a parameter changes type along its history.

\noindent \textbf{C$_3$: Non-proliferation of parameters.\;} Each parameter in an inference rule {\em inf} is congruent to at most one constituent in the conclusion of {\em inf}.

\noindent \textbf{C$_4$: Position-alikeness of parameters.\;} Congruent parameters are either all precedent or all succedent parts of their respective sequents. In the case of calculi enjoying the display property, precedent and succedent parts are defined in the usual way (see \cite{Belnap}). Otherwise, these notions can still be defined by induction on the shape of the structures, by relying on the polarity of each coordinate of the structural connectives.

\noindent \textbf{C$'_5$: Quasi-display of principal constituents.\;} If an operational term $a$ is principal in the conclusion sequent $s$ of a derivation $\pi$, then $a$ is in display, unless $\pi$ consists only of its conclusion sequent $s$ (i.e.\ $s$ is an axiom).

\noindent \textbf{C$''_5$: Display-invariance of axioms.} If $a$ is principal in an axiom $s$, then $a$ can be isolated by applying Display Postulates and the new sequent is still an axiom.

\noindent \textbf{C$'''_5$: Closure of axioms under surgical cut.} If $(x\vdash y)([a]^{pre}, [a]^{suc})$,  $a\vdash z [a]^{suc}$ and $v[a]^{pre}\vdash a$ are axioms, then $(x\vdash y)([a]^{pre}, [z/a]^{suc})$ and $(x\vdash y)([v/a]^{pre}, [a]^{suc})$  are again axioms.

\noindent \textbf{C$'_6$: Closure under substitution for succedent parts within each type.\;} Each rule is closed under simultaneous substitution of arbitrary structures for congruent operational terms occurring in succedent position, {\em within each type}.

\noindent \textbf{C$'_7$: Closure under substitution for precedent parts within each type.\;} Each rule is closed under simultaneous substitution of arbitrary structures for congruent operational terms occurring in precedent position, {\em within each type}.

\noindent Condition C$_6'$ (and likewise C$_7'$) ensures, for instance, that if the following inference is an application of the rule $R$:

\begin{center}
\AX$(x \fCenter y) \big([a]^{suc}_{i} \,|\, i \in I\big)$
\RightLabel{$R$}
\UI$(x' \fCenter y') [a]^{suc}$
\DisplayProof
\end{center}

\noindent and $\big([a]^{suc}_{i} \,|\, i \in I\big)$ represents all and only  the occurrences of $a$ in the premiss which are congruent to the occurrence of $a$  in the conclusion (if $I = \varnothing$, then the occurrence of $a$ in the conclusion is congruent to itself), %is congruent to all the occurrences $[A]^{suc}_{i}$ if the set of indices $I$ is nonempty (otherwise $A$ is congruent to itself),
then also the following inference is an application of the same rule $R$:

\begin{center}
\AX$(x \fCenter y) \big([z/a]^{suc}_{i} \,|\, i \in I\big)$
\RightLabel{$R$}
\UI$(x' \fCenter y') [z/a]^{suc}$
\DisplayProof
\end{center}

\noindent where the structure $z$ is substituted for $a$.

\noindent This condition caters for the step in the cut elimination procedure in which the cut needs to be ``pushed up'' over rules in which the cut-formula in  succedent position  is parametric (cf.~\cite[Section 4]{Trends}).

\noindent \textbf{C$'_8$: Eliminability of matching principal constituents.\;}
This condition requests a standard Gentzen-style checking, which is now limited to the case in which  both cut formulas are {\em principal}, i.e.~each of them has been introduced with the last rule application of each corresponding subdeduction. In this case, analogously to the proof Gentzen-style, condition C$'_8$ requires being able to transform the given deduction into a deduction with the same conclusion in which either the cut is eliminated altogether, or is transformed in one or more applications of the cut rule, involving proper subterms of the original operational cut-term. In addition to this, specific to the multi-type setting is the requirement that the new application(s) of the cut rule be also {\em strongly type-uniform} (cf.\ condition C$_{10}$ below).

\noindent \textbf{C$_9$: Type-uniformity of derivable sequents.} Each derivable sequent is type-uniform. %\marginpar{\raggedright\tiny{A: needs heavy editing: shorten present discussion, refer to first paper, expand on discussion relevant to the specific multi-type setting; regroup new conditions together???}}

%\marginpar{Vlasta: From the condition C$_9$ it trivially follows that the cut rule is type-uniform. Furthermore, because of the presence of display postulates we confined ourselves to a cut rule that acts upon formulas that are in isolation, we can show that if all other rules of a system preserve type-uniformity, also cut rule will preserve type-uniformity.}

\noindent \textbf{C$_{10}$: Preservation of type-uniformity of cut rules.} All cut rules preserve type-uniformity.
\end{definition}

In the context of proper multi-type calculi we say that a rule is {\em analytic} if it satisfies conditions C$_1$-C$'_7$ of the list above. Analytic rules can be added to a given proper multi-type calculus, and the resulting calculus  enjoys cut elimination and subformula property.

We state the cut-elimination metatheorem which we  appeal to when establishing the cut elimination for the calculus  introduced in Section \ref{sec: Display-style sequent calculus D.LRC}.

\begin{thm}
\label{Thm cut elimination}
Any calculus satisfying C$_2$, C$'_2$, C$'_3$, C$_4$, C$'''_5$, C$''''_5$, C$'_6$, C$'_7$, C$'_8$, C$''_8$, C$_9$ and C$_{10}$ is cut-admissible. If also C$_1$ is satisfied, then the calculus enjoys the subformula property.
\end{thm}

%%%

%\bibliography{BIB}
%\bibliographystyle{plain}

\bibliographystyle{asl}
\bibliography{BIB}

\end{document}